\RequirePackage{fix-cm}
\documentclass[11pt]{article}

\usepackage[a4paper,margin=1in]{geometry}
\usepackage{fontspec}
\usepackage{polyglossia}
\setmainlanguage{english}
\setotherlanguage{russian}
\defaultfontfeatures{Ligatures=TeX}
\newfontfamily\russianfont{Tempora-Regular.otf}[
  BoldFont=Tempora-Bold.otf,
  ItalicFont=Tempora-Italic.otf,
  BoldItalicFont=Tempora-BoldItalic.otf
]
\newfontfamily\cyrillicfont{Tempora-Regular.otf}[
  BoldFont=Tempora-Bold.otf,
  ItalicFont=Tempora-Italic.otf,
  BoldItalicFont=Tempora-BoldItalic.otf
]
\newfontfamily\cyrillicfonttt{FreeMono.otf}[
  BoldFont=FreeMonoBold.otf,
  ItalicFont=FreeMonoOblique.otf,
  BoldItalicFont=FreeMonoBoldOblique.otf
]
\usepackage{amsmath,amssymb,amsthm,mathtools,bm}
\DeclareMathSizes{5}{3.5}{2.8}{2.5}
\DeclareMathSizes{6}{4.5}{3.5}{3}
\DeclareMathSizes{7}{5.5}{4}{3.5}
\DeclareMathSizes{8}{6.5}{4.875}{4.06}
\DeclareMathSizes{9}{7.5}{5.83}{4.17}
\DeclareMathSizes{10}{8.5}{5.95}{4.25}
\DeclareMathSizes{10.95}{9.45}{6.9}{5.18}
\DeclareMathSizes{12}{10.5}{7}{5.25}
\DeclareMathSizes{12.5}{11}{7.7}{5.5}
\DeclareMathSizes{14.4}{12.9}{8.96}{6.27}
\DeclareMathSizes{16}{14.5}{10.15}{7.25}
\DeclareMathSizes{17.28}{15.78}{10.96}{8.22}
\DeclareMathSizes{20.74}{19.24}{13.36}{9.28}
\DeclareMathSizes{24.88}{23.38}{16.24}{11.28}
\usepackage{etoolbox}
\makeatletter
\patchcmd{\text@}{\f@size}{\tf@size}{}{\PackageError{mathsize}{Cannot resize display math text}{}}
\patchcmd{\text@}{\f@size}{\tf@size}{}{\PackageError{mathsize}{Cannot resize inline math text}{}}
\makeatother
\usepackage{booktabs,array,enumitem,microtype,needspace,graphicx,float}
\usepackage{xcolor}
\usepackage{titlesec}
\usepackage{tikz}
\usetikzlibrary{arrows.meta,decorations.markings}
\usepackage[numbers,sort&compress]{natbib}
\usepackage{xurl}
\usepackage[hypertexnames=false]{hyperref}
\definecolor{ink}{HTML}{000000}
\definecolor{softgray}{HTML}{F3F4F6}
\definecolor{rulegray}{HTML}{9CA3AF}
\definecolor{geoPtwo}{HTML}{17365D}
\definecolor{geoPthreehalf}{HTML}{2F5F8F}
\definecolor{geoPsixfifths}{HTML}{6F8FB5}
\definecolor{geoPone}{HTML}{9CBAD0}

\tikzset{
  geo axis/.style={->,draw=black!70,line width=.45pt},
  geo attainable/.style={solid,line width=.9pt,line cap=round,line join=round},
  geo unattainable/.style={draw=black!75,densely dashed,line width=.8pt,
    line cap=round,line join=round},
  geo support/.style={draw=black!48,line width=.55pt},
  geo excluded/.style={circle,draw=black,fill=white,line width=.65pt,
    inner sep=1.25pt},
  geo included/.style={circle,draw=black,fill=black,inner sep=1.05pt}
}

\hypersetup{
  colorlinks=true,
  allcolors=black,
  pdftitle={A Sharp Joint Bias-Energy Envelope for Radial Clipping},
  pdfauthor={Danila Litvinov},
  pdfsubject={Revised and expanded article: English and Russian versions},
  pdfkeywords={radial clipping, heavy tails, sharp inequality, bias, energy}
}

\titleformat{\section}
  {\normalfont\bfseries\color{ink}\fontsize{16pt}{19pt}\selectfont}
  {\thesection}{0.75em}{}
\titleformat{\subsection}
  {\normalfont\bfseries\color{ink}\fontsize{12.5pt}{15pt}\selectfont}
  {\thesubsection}{0.65em}{}
\titlespacing*{\section}{0pt}{3.5ex plus 1ex minus .2ex}{1.4ex}
\titlespacing*{\subsection}{0pt}{2.4ex plus .7ex minus .2ex}{.8ex}

\newcommand{\ArticleTheoremName}{Theorem}
\newcommand{\ArticleLemmaName}{Lemma}
\newcommand{\ArticleCorollaryName}{Corollary}
\newcommand{\ArticlePropositionName}{Proposition}
\newcommand{\ArticleRemarkName}{Remark}

\theoremstyle{remark}

\theoremstyle{plain}
\newtheorem{theoremru}{\ArticleTheoremName}[section]
\newtheoremstyle{italicunnumbered}{6pt}{6pt}{\normalfont}{}{\itshape}{.}{0.5em}%
{\thmname{#1}\thmnumber{ #2}\thmnote{ (#3)}}
\theoremstyle{italicunnumbered}
\newtheorem*{lemmaru}{\ArticleLemmaName}
\theoremstyle{plain}
\newtheorem{corollaryru}[theoremru]{\ArticleCorollaryName}
\newtheorem{propositionru}[theoremru]{\ArticlePropositionName}
\newtheoremstyle{boldremark}{6pt}{6pt}{\normalfont}{}{\bfseries}{.}{0.5em}{}
\theoremstyle{boldremark}
\newtheorem{remarkru}[theoremru]{\ArticleRemarkName}

\newcommand{\E}{\mathbb{E}}
\newcommand{\G}{\mathcal{G}}
\newcommand{\norm}[1]{\left\lVert #1\right\rVert}

\newcommand{\R}{\mathbb{R}}

\makeatletter
\renewenvironment{proof}[1][\(\blacktriangleright\)]{%
  \par
  \pushQED{\qed}%
  \normalfont \topsep6\p@\@plus6\p@\relax
  \trivlist
  \item[\hskip\labelsep #1]\ignorespaces
}{%
  \popQED\endtrivlist\@endpefalse
}
\makeatother

\setlist[itemize]{leftmargin=1.5em,itemsep=.35em,topsep=.45em}
\allowdisplaybreaks[2]

\newcommand{\makearticletitle}[2]{%
  \newpage\null\vskip 2em%
  \begin{center}%
    {\LARGE #1\par}%
    \vskip 1.5em%
    {\large #2\par}%
  \end{center}%
  \par\vskip 1.5em%
}
\newcommand{\makeenglishtitle}{%
  \makearticletitle{\vspace{-1.2cm}\textbf{A Sharp Joint Bias--Energy Envelope\\[2mm]
  for Radial Clipping}}{Danila Litvinov\\[2mm]
  {\normalsize This work was carried out without a scientific supervisor.}\\[1mm]
  {\normalsize A previous version was published as
  \href{https://arxiv.org/abs/2608.18999}{arXiv:2608.18999}.}\\[1mm]
  {\normalsize The present version has been revised and expanded.}}%
}

\newcommand{\makerussiantitle}{%
  \makearticletitle{\vspace{-1.2cm}\textbf{Точная совместная огибающая смещения\\[2mm]
  и энергии радиального клиппинга}}{Данила Литвинов\\[2mm]
  {\normalsize Работа выполнена без научного руководителя.}\\[1mm]
  {\normalsize Предыдущая версия работы опубликована на
  \href{https://arxiv.org/abs/2608.18999}{arXiv:2608.18999}.}\\[1mm]
  {\normalsize Настоящая версия переработана и дополнена.}}%
}

\begin{document}
\begingroup
\selectlanguage{english}
\pdfbookmark[0]{English version}{language.en}
\color{ink}
\makeenglishtitle
\vspace{-1.2em}

\begin{abstract}
Radial clipping does two things at once: it cuts off the part of a vector outside a ball and leaves a bounded vector inside the ball. The part cut off defines the residual. For a random vector, we study the magnitude of the bias in the mean and the energy, defined as the centered second moment of the clipped vector. Controlling these effects is important in stochastic optimization with heavy tails, where the choice of the clipping threshold affects convergence bounds for algorithms.

In the deterministic problem, for a fixed threshold \(\tau>0\), \(1\le p\le2\), and nonnegative coefficients \(\alpha,\beta\), we find the smallest constant that simultaneously bounds the residual and energy in terms of \(\norm{x}^p\). This same constant remains sharp in the stochastic counterpart for random vectors in a Hilbert space. We also find the exact supremum of the corresponding linear combination of normalized bias and energy for arbitrary \(\alpha,\beta\in\mathbb R\).

We then give a geometric picture in the energy--bias plane. Varying the weights, we describe the smallest closed convex region containing all normalized pairs and the attainable set itself, identifying unattainable boundary values. Every attainable pair at a fixed threshold can be realized by a real-valued random variable taking at most two values.

We also consider all thresholds simultaneously. For any integrable random vector in a Hilbert space, we construct a single \(\mathbb R^3\)-valued random vector preserving its norm distribution, bias magnitude, and energy simultaneously at every threshold \(\tau>0\).

Finally, we apply the joint bound to risk-adjusted reinsurance costs for a prescribed retention threshold and to the logarithmic shortfall of a clipped portfolio strategy relative to any constant-fraction strategy. AI was used in preparing this paper (see Section~\ref{en:sec:ai-use-ru}).
\end{abstract}



\section{Introduction}

Radial clipping is widely used when working with heavy-tailed data\textcolor{black}{,
 in particular in stochastic optimization~{\hypersetup{citecolor=black}\citep{en:koloskova2023revisitingRU,en:nguyen2023improvedRU}}.}
\textcolor{black}{It} replaces overly large observations with bounded vectors while preserving their
 direction. However, some information is lost in this truncation: the mean
 may shift, and the magnitude and dispersion of the remaining data may change.
\textcolor{black}{For a random vector, we are interested in the magnitude of the bias of the mean
 and in the energy---a measure of the dispersion of the clipped vector around its own mean.}

{\hypersetup{citecolor=black,linkcolor=black}\color{black}
In studies of algorithms with clipping, bias and variance are often bounded
 separately~\citep{en:li2023clippingRU}. In this paper, we ask for a
 sharp joint bound. A comparison with specific bounds from the literature
 is given in Section~\ref{en:sec:scope-ru}.
\par}

The main contribution of this paper is a sharp joint description of the bias
 and energy of radial clipping. We find the smallest constant in the
 corresponding inequality and show that it remains optimal for
 random vectors as well as fixed vectors. In addition, for the
 class of random vectors under consideration, we give a complete description of how energy
 and bias can relate to one another.

\subsection{Problem formulation}\label{en:subsec:intro-deterministic-ru}

Let us specify the main object of study. Let \(E\) be a nonzero
 real normed space and let \(\tau>0\). Define the
 radial clipping operator \(C_\tau\colon E\to E\) by
\[
C_\tau(x)=
\begin{cases}
x, & \norm{x}\le\tau,\\[1mm]
\displaystyle \tau\frac{x}{\norm{x}}, & \norm{x}>\tau.
\end{cases}
\]
It is then natural to define two quantities. The first---the residual---
 measures how much clipping changes the original vector. The second---
 the energy---characterizes the part of the vector that
 remains after clipping:
\[
R_\tau(x):=\norm{x-C_\tau(x)},
\qquad
Q_\tau(x):=\norm{C_\tau(x)}^2.
\]
To specify the relative importance of the residual and the energy, we introduce nonnegative
 coefficients \(\alpha\) and \(\beta\), respectively. We combine these quantities in a
 dimensionally consistent form:
\[
\alpha R_\tau(x)+\frac{\beta}{\tau}Q_\tau(x),
\]
where the factor \(1/\tau\) gives the energy the same dimension as the residual.

Fix \(1\le p\le2\). We seek to bound the quantity under consideration
 in terms of the \(p\)th power of the norm of the input vector. It is therefore natural
 to use the following quantity to control it:
\[
\tau\left(\frac{\norm{x}}{\tau}\right)^p.
\]
Here \(\norm{x}/\tau\) characterizes the relative magnitude of the norm \(\norm{x}\)
 compared with the clipping threshold, and the factor \(\tau\) preserves
 the dimension.

The resulting question can therefore be formulated as follows: what is
 the smallest constant \(K_p(\alpha,\beta)\) for which the following inequality
 holds for all \(x\in E\)?
\[
\alpha R_\tau(x)+\frac{\beta}{\tau}Q_\tau(x)
\le
K_p(\alpha,\beta)\,
\tau\left(\frac{\norm{x}}{\tau}\right)^p.
\]

It turns out that, for \(1<p\le2\), the smallest possible constant has the form
\[
K_p(\alpha,\beta)=
\begin{cases}
\beta, & \alpha\le p\beta,\\[2mm]
\displaystyle
\frac{(p-1)^{p-1}}{p^p}
\frac{\alpha^p}{(\alpha-\beta)^{p-1}},
& \alpha>p\beta,
\end{cases}
\]
In the endpoint case \(p=1\), it takes the form
\[
K_1(\alpha,\beta)=\max\{\alpha,\beta\}.
\]

\subsection{From a single vector to a random observation}

Having formulated the problem for a single fixed vector \(x\), we arrive at
 the next question: what changes if \(x\) is replaced by a random vector?
Does the constant found in the deterministic case remain optimal when
\(X\in L^p(\Omega;H)\), that is, when \(X\) is a random vector
\textcolor{black}{on an arbitrary probability space \(\Omega\)} taking values in a
\textcolor{black}{nonzero} real Hilbert space \(H\), with
\(\E[\norm{X}^p]<\infty\)?

After clipping, we obtain a random vector \(C_\tau(X)\) all of whose values
 lie inside the ball of radius \(\tau\). As in the deterministic case,
 clipping produces two effects. First, it may shift the mean: the
 original center \(\E[X]\) is replaced by the new center \(\E[C_\tau(X)]\).
Second, the values of \(C_\tau(X)\) still deviate from the new center, and
 the magnitude of this dispersion characterizes the energy remaining after clipping.
It is therefore natural to consider the two quantities
\[
\norm{\E[C_\tau(X)]-\E[X]},
\qquad
\E\!\left[\norm{C_\tau(X)-\E[C_\tau(X)]}^2\right].
\]

As in Subsection~1.1, we use nonnegative coefficients \(\alpha\) and
\(\beta\) to account for the contributions of bias and energy to the overall bound.
We again control the resulting quantity by the \(p\)th moment of the original random vector \(X\).

This leads to an analogous problem: find the smallest constant \(K\),
\textcolor{black}{common to all such probability spaces,} such
 that, for every \(X\in L^p(\Omega;H)\),
\[
\alpha\norm{\E[C_\tau(X)]-\E[X]}
+\frac{\beta}{\tau}
\E\!\left[\norm{C_\tau(X)-\E[C_\tau(X)]}^2\right]
\le
K\,
\tau\E\!\left[\left(\frac{\norm{X}}{\tau}\right)^p\right].
\]
It turns out that the smallest constant \(K\) in this inequality coincides
 exactly with the constant \(K_p(\alpha,\beta)\) found in
 Subsection~\ref{en:subsec:intro-deterministic-ru}.

We also note that we will consider a conditional version of the resulting
 inequality with respect to a \(\sigma\)-algebra \(\mathcal G\), in which
 all expectations are conditional. This version allows the inequality to be applied using
 only the information accumulated up to the current time, and is useful in the analysis
 of stochastic processes and sequential algorithms (for example, in
 stochastic optimization and online learning).

\Needspace{12\baselineskip}
\subsection{Geometry of bias and energy}

So far, we have fixed the coefficients \(\alpha\) and \(\beta\) and asked
 how large the corresponding combination of the resulting bias
 and energy can be. We now approach the same problem from another angle and ask:
 how do bias and energy relate to one another?

For this purpose, when \(0<\E[\norm{X}^p]<\infty\), we normalize the energy and bias using
 the \(p\)th moment of the original vector and the clipping radius. For fixed
\(p\) and \(\tau\), we introduce the energy coordinate \(V_p\) and bias coordinate \(D_p\):
\[
V_p(X):=
\frac{\tau^{p-2}\E\!\left[\norm{C_\tau(X)-\E[C_\tau(X)]}^2\right]}
{\E[\norm{X}^p]},
\qquad
D_p(X):=
\frac{\tau^{p-1}\norm{\E[C_\tau(X)]-\E[X]}}
{\E[\norm{X}^p]}.
\]
Fixing a random vector \(X\) and a clipping radius \(\tau\) gives
 a point \((V_p(X),D_p(X))\) in the ``energy--bias'' plane.

We also note that this normalization was chosen so that dividing the
 inequality from Subsection~1.2 by \(\tau^{1-p}\E[\norm{X}^p]\) gives the simple
 dimensionless form
\[
\boxed{\textcolor{black}{\beta V_p(X)+\alpha D_p(X)}\le K_p(\alpha,\beta).}
\]
This expression has a natural geometric interpretation. For fixed
\(\alpha,\beta\ge0\), not both zero, the equation
\[
\textcolor{black}{\beta v+\alpha d}=K_p(\alpha,\beta),
\]
defines a supporting line of the closure of the convex hull of the set of attainable pairs
\((V_p(X),D_p(X))\). In other words, \(K_p(\alpha,\beta)\) allows us to
 construct supporting lines of the \textcolor{black}{closed convex hull} of the set of attainable
 pairs. \textcolor{black}{We will first describe this hull and then the set of
 attainable pairs itself and its boundary.}

Once we have described the full geometry of this set, we will answer the
 following question: how complicated must a random variable be
 to realize a given attainable pair \((V_p(X),D_p(X))\)? \textcolor{black}{We}
 will show that every attainable normalized pair \(\textcolor{black}{(v,d)}\) can be realized
 by a real-valued random variable taking at most two values. This
 preserves the pair \(\textcolor{black}{(v,d)}\) itself, but not necessarily the original moments
 of the distribution.

\textcolor{black}{So far, however, we have reproduced a single pair at a fixed threshold.
We will next determine whether the original random vector can be replaced
 by a single new vector that gives the same bias magnitude
 and the same energy at every threshold \(\tau>0\), without having to choose it anew when
 the threshold changes. We will also require that the distribution of the norm
 of the original vector be preserved. We will show that such a random vector can be constructed
 in \(\R^3\).}

\subsection{Applications}

Our joint bound can be used wherever clipping changes two quantities
 simultaneously: it cuts off part of an observation and bounds what remains
 after truncation. In Section~5, we consider two such problems---reinsurance
 and portfolio management.

In the first problem, an insurance company chooses a payout limit for each
 claim in advance and transfers the amount above the limit to a reinsurer. We show
 how to bound the company's risk-adjusted costs when the first two moments
 of the loss sizes and of the number of claims are known. In the second problem, a fund changes the fraction
 invested in a risky asset from one period to the next, using clipping to bound the magnitude
 of the update signal. Here we bound how far the \textcolor{black}{logarithm of the terminal wealth} of this
 strategy can fall below the \textcolor{black}{logarithm of the wealth} of a strategy with any constant fraction.

\subsection{Organization of the paper}

The rest of the paper is organized as follows. In Section~2, we solve
 the deterministic problem and find the sharp constants for \(1\le p\le2\),
 treating the endpoint case \(p=1\) separately. In Section~3, we
 extend the bound to random vectors, prove that the same constant is sharp, and
 obtain a conditional form of the inequality. In Section~4, we study normalized
 energy--bias pairs: we describe their closed convex hull and the set
 of attainable pairs itself, find supporting lines for arbitrary weights, and show
 that every attainable pair can be realized by a two-point random variable.
We also prove there that a single random vector in \(\mathbb R^3\) reproduces
 the \textcolor{black}{bias magnitude} and energy at all thresholds while preserving the distribution of the norm.
In Section~5, we consider two applications: reinsurance with a random number
 of losses and a prescribed retention threshold, and a portfolio management algorithm with
 clipping. In Section~6, we compare our results with related work
 and explain the motivation for the study. Section~7 presents the formal verification
 of the numbered results of Sections~2--5 in Lean~4 and explains how to reproduce it.
In Section~8, we describe the use of AI in preparing the paper. Section~9
 concludes the paper.

\section{Sharp constants in the deterministic problem}

In this section, we solve the deterministic problem formulated in
 Subsection~\ref{en:subsec:intro-deterministic-ru} and find the sharp constant
\(K_p(\alpha,\beta)\) for all \(1\le p\le2\). We first consider the case
\(1<p\le2\), then treat the endpoint case \(p=1\) separately and explain
 why it requires separate consideration.

\Needspace{12\baselineskip}
\subsection[The case 1 < p <= 2]{The case \(\bm{1<p\le2}\)}

Let us fix the main notation to be used below.
Let \(E\) be a nonzero real normed space. For
\(\tau>0\), define the radial clipping operator by
\vspace{0.35em}
\[
C_\tau(x)=
\begin{cases}
x,&\norm{x}\le\tau,\\[2mm]
\tau\dfrac{x}{\norm{x}},&\norm{x}>\tau.
\end{cases}
\]
\vspace{0.35em}
Also define the function \(c\colon(1,2]\to(0,\infty)\) by
\[
c(p):=\frac{(p-1)^{p-1}}{p^p}.
\]
For \(x\in E\), define its normalized radius by
\[
r:=\frac{\norm{x}}{\tau}.
\]

\begin{theoremru}[Sharp joint envelope]\label{en:thm:deterministic-ru}
Let \(1<p\le2\) and \(\tau>0\). For \(\alpha,\beta\ge0\), define
\[
K_p(\alpha,\beta)=
\begin{cases}
\beta,&\alpha\le p\beta,\\[2mm]
c(p)\dfrac{\alpha^p}{(\alpha-\beta)^{p-1}},
&\alpha>p\beta.
\end{cases}
\]
Then \(K_p(\alpha,\beta)\) is the smallest possible uniform
 constant for which the following inequality holds:
\begin{equation}\label{en:eq:det-envelope-ru}
\boxed{
\forall\,x\in E\colon\quad
\alpha\norm{x-C_\tau(x)}
+\frac{\beta}{\tau}\norm{C_\tau(x)}^2
\le
K_p(\alpha,\beta)\,\tau
\left(\frac{\norm{x}}{\tau}\right)^p.}
\end{equation}
\end{theoremru}

\begin{proof}
The assertion is immediate when \(x=0\). Now let \(x\ne0\). We consider
 separately the cases where \(x\) lies in the closed ball of radius \(\tau\)
 and where it lies outside this ball.

\textit{1.} The vector \(x\) lies in the closed ball of radius \(\tau\). Then
\(0<r\le1\) and
\[
C_\tau(x)=x,\qquad
\norm{x-C_\tau(x)}=0,\qquad
\norm{C_\tau(x)}=r\tau.
\]
Thus, after division by \(\tau\), inequality~\eqref{en:eq:det-envelope-ru}
 takes the form
\[
\beta r^2\le K_p(\alpha,\beta)r^p.
\]
Consequently, the contribution of the first case to the desired constant is
\[
\sup_{0<r\le1}\beta r^{2-p}=\beta,
\]
since \(2-p\ge0\).

\textit{2.} The vector \(x\) lies outside the closed ball of radius \(\tau\). Then
\(r>1\). In this case,
\[
C_\tau(x)=\tau\frac{x}{\norm{x}},\qquad
\norm{x-C_\tau(x)}=(r-1)\tau,\qquad
\norm{C_\tau(x)}=\tau.
\]
Substituting into~\eqref{en:eq:det-envelope-ru} and dividing by \(\tau\)
 gives the inequality
\[
\alpha(r-1)+\beta\le K_p(\alpha,\beta)r^p.
\]
\textcolor{black}{At \(r=1\), we obtain the same value \(\beta\) as in the first case.
Thus the optimal constant is given by}
\[
K_p(\alpha,\beta)
=
\sup_{r\ge1}f(r),
\qquad
f(r):=\frac{\alpha(r-1)+\beta}{r^p}.
\]
\Needspace{7\baselineskip}
We have
\[
f'(r)
=\frac{p(\alpha-\beta)-(p-1)\alpha r}{r^{p+1}}
=\frac{\alpha-p\beta-(p-1)\alpha(r-1)}{r^{p+1}},
\qquad r\in[1,\infty),
\]
and the denominator is positive, so the sign of the derivative is determined by the numerator.
When \(\alpha\le p\beta\), the function \(f\) is nonincreasing and its maximum is
\(f(1)=\beta\). If \(\alpha>p\beta\), the maximum is attained at the unique
 point where the derivative vanishes:
\[
r_*=\frac{p(\alpha-\beta)}{\alpha(p-1)}>1,
\qquad
f(r_*)=K_p(\alpha,\beta)
=c(p)\frac{\alpha^p}{(\alpha-\beta)^{p-1}}.
\]
Every \(r>0\) is realized by choosing \(x=r\tau e\), where \(\norm{e}=1\), so
 this constant is sharp. Combining the results of the two
 cases yields the desired exact formula for \(K_p(\alpha,\beta)\) and proves its
 optimality.
\end{proof}

\subsection[The endpoint case p = 1]{The endpoint case \(p=1\)}

When \(p=1\), the derivative in the preceding proof degenerates: the term
\((p-1)\alpha r\) disappears. For this reason, we consider the case \(p=1\)
 separately.

\begin{theoremru}[Endpoint case]\label{en:thm:deterministic-p1-ru}
For \(\alpha,\beta\ge0\), the optimal constant in
 inequality~\eqref{en:eq:det-envelope-ru} for \(p=1\) is
\(K_1(\alpha,\beta)=\max\{\alpha,\beta\}\).
\end{theoremru}

\begin{proof}
The assertion is immediate when \(x=0\). As in the proof of Theorem~\ref{en:thm:deterministic-ru},
 for \(x\ne0\), set \(r=\norm{x}/\tau\). When \(0<r\le1\), the normalized
 left-hand side is \(\beta r\), and its supremum is \(\beta\). When \(r\ge1\),
 it is \(\alpha+(\beta-\alpha)/r\), and its supremum is
\(\max\{\alpha,\beta\}\). The value \(\beta\) is attained at \(r=1\), while
 the value \(\alpha\) is approached as \(r\to\infty\); hence
\(K_1(\alpha,\beta)=\max\{\alpha,\beta\}\).
\end{proof}

\section{The sharp stochastic envelope}

In this section, we extend Theorems~2.1 and~2.2 to the stochastic setting
 by replacing the fixed vector \(x\) with a random vector \(X\). We first prove
 the stochastic bound for \(1\le p\le2\), then establish optimality
 of the resulting constant for \(1\le p\le2\), and finally obtain a conditional form
 of the result.

\textcolor{black}{Hereafter, \(\Omega\) denotes an arbitrary probability
 space. This space may vary with the choice of the random vector \(X\).}

\subsection{The stochastic bound}

\begin{theoremru}[Stochastic envelope]
\label{en:thm:stochastic-ru}
Let \(\textcolor{black}{H}\) be a nonzero real Hilbert space,
\(1\le p\le2\), \(\tau>0\), and \textcolor{black}{\(X\in L^p(\Omega;H)\)}. Then, for any
 deterministic \(\alpha,\beta\ge0\),
\begin{equation}\label{en:eq:stochastic-upper-ru}
\boxed{
\alpha\norm{\E[C_\tau(X)]-\E[X]}
+\frac{\beta}{\tau}
\E\!\left[\norm{C_\tau(X)-\E[C_\tau(X)]}^2\right]
\le
K_p(\alpha,\beta)\tau\,
\frac{\E[\norm{X}^p]}{\tau^p}.}
\end{equation}
\end{theoremru}

\begin{proof}
Fix \(\omega\in\Omega\). If \(p>1\), apply Theorem~\ref{en:thm:deterministic-ru} to the vector
\(X(\omega)\); if \(p=1\), apply
 Theorem~\ref{en:thm:deterministic-p1-ru}. In either case, we obtain
\[
\alpha\norm{X(\omega)-C_\tau(X(\omega))}
+\frac{\beta}{\tau}\norm{C_\tau(X(\omega))}^2
\le
K_p(\alpha,\beta)\tau\,
\frac{\norm{X(\omega)}^p}{\tau^p}.
\]
Taking expectations on both sides (which is valid because
 expectation preserves inequalities between integrable random
 variables), we obtain
\[
\begin{array}{c@{\quad}c@{\quad}c@{\qquad}c}
\displaystyle
\alpha\E[\norm{X-C_\tau(X)}]
&
 +
&
\displaystyle
\frac{\beta}{\tau}\E[\norm{C_\tau(X)}^2]
&
\displaystyle
\le K_p(\alpha,\beta)\tau\,
\frac{\E[\norm{X}^p]}{\tau^p}
\\[1mm]
\rotatebox[origin=c]{-90}{\(\geqslant\)}\;{\scriptstyle\text{\rm(I)}}
&
{}
&
\rotatebox[origin=c]{-90}{\(\geqslant\)}\;{\scriptstyle\text{\rm(II)}}
&
{}
\\[1mm]
\displaystyle
\alpha\norm{\E[C_\tau(X)]-\E[X]}
&
 +
&
\displaystyle
\frac{\beta}{\tau}
\E\!\left[\norm{C_\tau(X)-\E[C_\tau(X)]}^2\right].
&
{}
\end{array}
\]
Bound \(\text{\rm(I)}\) holds because the norm of the mean vector does not exceed
 the mean of its norm:
\[
\norm{\E[C_\tau(X)]-\E[X]}
=\norm{\E[C_\tau(X)-X]}
\le\E[\norm{C_\tau(X)-X}].
\]
Bound \(\text{\rm(II)}\) follows from the variance identity (which applies
 because \(\norm{C_\tau(X)}\le\tau\), and hence
\(\E[\norm{C_\tau(X)}^2]\le\tau^2<\infty\)):
\[
\E\!\left[\norm{C_\tau(X)-\E[C_\tau(X)]}^2\right]
=\E[\norm{C_\tau(X)}^2]-\norm{\E[C_\tau(X)]}^2
\le\E[\norm{C_\tau(X)}^2].
\]
Adding bounds \(\text{\rm(I)}\) and \(\text{\rm(II)}\) gives
\eqref{en:eq:stochastic-upper-ru}.
\end{proof}

\subsection[Sharpness of the constant for 1 <= p <= 2]
{Sharpness of the constant for \(1\le p\le2\)}

\begin{propositionru}[Sharpness of the constant]\label{en:prop:stochastic-sharp-ru}
Let \(\textcolor{black}{H}\) be a nonzero real Hilbert space,
\(1\le p\le2\), \(\tau>0\), and \(\alpha,\beta\ge0\).
Let \textcolor{black}{\(X\in L^p(\Omega;H)\)} and \(\E[\norm{X}^p]\ne0\). Then
\begin{equation}\label{en:eq:stochastic-sup-ru}
\boxed{
\sup_X
\frac{
\alpha\norm{\E[C_\tau(X)]-\E[X]}
+\frac{\beta}{\tau}
\E\!\left[\norm{C_\tau(X)-\E[C_\tau(X)]}^2\right]
}{
\displaystyle \tau\,\frac{\E[\norm{X}^p]}{\tau^p}
}
=K_p(\alpha,\beta).}
\end{equation}
\end{propositionru}

\begin{proof}
By Theorem~\ref{en:thm:stochastic-ru}, the quantity \(K_p(\alpha,\beta)\) is an
 upper bound in~\eqref{en:eq:stochastic-sup-ru}. We show that this bound is sharp.
Fix \(e\in\textcolor{black}{H}\) with \(\norm{e}=1\), and numbers \(r\ge1\),
\(0<q<1\). Define the random vector
\[
X_{q,r}=
\begin{cases}
r\tau e,&\text{with probability }q,\\[1mm]
0,&\text{with probability }1-q.
\end{cases}
\]
Then
\[
\E[X_{q,r}]=qr\tau e,
\qquad
\E[C_\tau(X_{q,r})]=q\tau e,
\]
\[
\E\!\left[
\norm{C_\tau(X_{q,r})-\E[C_\tau(X_{q,r})]}^2
\right]
=q(1-q)\tau^2,
\qquad
\E[\norm{X_{q,r}}^p]=qr^p\tau^p.
\]
Substituting these values into the expression under the supremum
 in~\eqref{en:eq:stochastic-sup-ru} and cancelling the common factor \(q\tau\) gives
\[
\frac{
\alpha\norm{q\tau e-qr\tau e}
+\dfrac{\beta}{\tau}q(1-q)\tau^2
}{
\displaystyle \tau\,\frac{qr^p\tau^p}{\tau^p}
}
=
\frac{q\tau\bigl(\alpha(r-1)+\beta(1-q)\bigr)}{q\tau r^p}
=
\frac{\alpha(r-1)+\beta(1-q)}{r^p}.
\]
As \(q\downarrow0\), the last expression tends to
\(\dfrac{\alpha(r-1)+\beta}{r^p}\). This expression was already
 considered in the proofs of Theorems~\ref{en:thm:deterministic-ru}
 and~\ref{en:thm:deterministic-p1-ru}, where its supremum
 over \(r\ge1\) was shown to equal \(K_p(\alpha,\beta)\) for every \(p\in[1,2]\).
Thus the constant in~\eqref{en:eq:stochastic-sup-ru} is sharp.
\end{proof}

\subsection{Conditional form}

\begin{corollaryru}[Conditional form]\label{en:cor:conditional-ru}
Let \(\textcolor{black}{H}\) be a nonzero real Hilbert space,
\(1\le p\le2\), and \textcolor{black}{\(X\in L^p(\Omega;H)\)}. Let \(\G\) be a
 sub-\(\sigma\)-algebra, let \(\alpha,\beta\ge0\) be deterministic, and let
\(\tau:\Omega\to(0,\infty)\) be almost surely finite and
\(\G\)-measurable. Then, almost surely,
\[
\alpha\norm{\E[C_\tau(X)\mid\G]-\E[X\mid\G]}
+\frac{\beta}{\tau}
\E\!\left[
\norm{C_\tau(X)-\E[C_\tau(X)\mid\G]}^2
\,\middle|\,\G
\right]
\le
K_p(\alpha,\beta)\tau\,
\frac{\E[\norm{X}^p\mid\G]}{\tau^p}.
\]
\end{corollaryru}

\begin{proof}
As in the proof of Theorem~3.1, for \(p>1\), apply
 Theorem~2.1 to the vector \(X(\omega)\) with threshold \(\tau(\omega)\), and for
\(p=1\), apply Theorem~2.2. Almost surely, this gives
\[
\alpha\norm{X-C_\tau(X)}
+\frac{\beta}{\tau}\norm{C_\tau(X)}^2
\le
K_p(\alpha,\beta)\tau\,
\frac{\norm{X}^p}{\tau^p}.
\]
For nonnegative quantities, conditional expectations are understood in the
extended sense; the conditional bounds below are finite almost surely.
Next, take conditional expectations with respect to \(\G\) throughout
 and pull \(\tau\) outside the expectation (this is
 valid because \(\tau\) is \(\G\)-measurable). We obtain
\[
\begin{array}{c@{\quad}c@{\quad}c@{\qquad}c}
\displaystyle
\alpha\E[\norm{X-C_\tau(X)}\mid\G]
&
 +
&
\displaystyle
\frac{\beta}{\tau}\E[\norm{C_\tau(X)}^2\mid\G]
&
\displaystyle
\le K_p(\alpha,\beta)\tau\,
\frac{\E[\norm{X}^p\mid\G]}{\tau^p}
\\[1mm]
\rotatebox[origin=c]{-90}{\(\geqslant\)}\;{\scriptstyle\text{\rm(I)}}
&
{}
&
\rotatebox[origin=c]{-90}{\(\geqslant\)}\;{\scriptstyle\text{\rm(II)}}
&
{}
\\[1mm]
\displaystyle
\alpha\norm{\E[C_\tau(X)\mid\G]-\E[X\mid\G]}
&
 +
&
\displaystyle
\frac{\beta}{\tau}
\E\!\left[
\norm{C_\tau(X)-\E[C_\tau(X)\mid\G]}^2
\,\middle|\,\G
\right].
&
{}
\end{array}
\]
Bound \(\text{\rm(I)}\) holds because the norm of the conditional mean vector
 does not exceed the conditional mean of its norm:
\[
\norm{\E[C_\tau(X)\mid\G]-\E[X\mid\G]}
=\norm{\E[C_\tau(X)-X\mid\G]}
\le\E[\norm{C_\tau(X)-X}\mid\G].
\]
Bound \(\text{\rm(II)}\) follows from the conditional variance identity:
\[
\E\!\left[
\norm{C_\tau(X)-\E[C_\tau(X)\mid\G]}^2
\,\middle|\,\G
\right]
=\E[\norm{C_\tau(X)}^2\mid\G]
-\norm{\E[C_\tau(X)\mid\G]}^2
\le\E[\norm{C_\tau(X)}^2\mid\G].
\]
Adding bounds \(\text{\rm(I)}\) and \(\text{\rm(II)}\) shows that
 the lower row does not exceed the upper row. The upper row is already bounded by the right-hand
 side, which proves the assertion.
\end{proof}

\noindent\textit{Remark.}
Applying the conditional variance identity requires
\(C_\tau(X)\) to be integrable and to have a finite conditional second moment. Both
 conditions hold: \(\norm{C_\tau(X)}\le\norm{X}\) and \(\norm{X}\in L^1\),
 while \(\E[\norm{C_\tau(X)}^2\mid\G]\le\tau^2<\infty\) almost surely, since
\(\norm{C_\tau(X)}\le\tau\), and \(\tau\) is almost surely finite and
\(\G\)-measurable.

\noindent\textcolor{black}{\textit{Comment.} Intuitively, what happens here is
 the following: we take into account the information already known to us and average only what
 remains random. The clipping threshold is already known, so
 the argument remains the same.}

\section[Geometric interpretation of the set of attainable pairs]
{Geometric interpretation of\\the set of attainable pairs}
\label{en:sec:geometry-ru}

We have learned how to bound a weighted linear combination of two nonnegative
 effects---energy and bias. It is now natural to ask how
 these two quantities can relate to one another.

\subsection{Coordinates and geometric interpretation of the stochastic result}
\label{en:subsec:geometry-support-ru}

We retain the notation and assumptions of
\textcolor{black}{Proposition}~\ref{en:prop:stochastic-sharp-ru}. For each admissible random
 vector \(X\), define the normalized energy and bias coordinates by
\begin{center}
\makebox[\textwidth][c]{%
{\(
\displaystyle
V_p(X):=
\dfrac{
\dfrac{1}{\tau}\E\!\left[
\norm{C_\tau(X)-\E[C_\tau(X)]}^2\right]
}{
\tau\,\dfrac{\E[\norm{X}^p]}{\tau^p}
}
\)}%
\hspace{4em}
{\(
\displaystyle
D_p(X):=
\dfrac{
\norm{\E[C_\tau(X)]-\E[X]}
}{
\tau\,\dfrac{\E[\norm{X}^p]}{\tau^p}
}
\)}%
}
\end{center}
Next, we define the visualization map
\(\Phi_p(X):=\bigl(V_p(X),D_p(X)\bigr)\in\R_+^2\). Considering all
 admissible random vectors \(X\) gives the set \(\mathcal A_p\),
 consisting of all such attainable pairs \((v,d)\):
\[
\mathcal A_p:=
\bigl\{(v,d)\in\R_+^2:
\text{there exists an admissible random vector }X,
\ (v,d)=\Phi_p(X)\bigr\}.
\]

Let us now return to the result of \textcolor{black}{Proposition}~\ref{en:prop:stochastic-sharp-ru}.
In the coordinates introduced above, it states that, for nonnegative \(\alpha\)
 and \(\beta\), we have
\begin{equation}\label{en:eq:positive-support-ru}
\sup_{(v,d)\in\mathcal A_p}
\bigl\langle(\beta,\alpha),(v,d)\bigr\rangle
=\sup_{(v,d)\in\mathcal A_p}(\textcolor{black}{\beta v+\alpha d})
=K_p(\alpha,\beta),
\end{equation}
\textcolor{black}{When \((\alpha,\beta)\ne(0,0)\), the lines
\(\beta v+\alpha d=K_p(\alpha,\beta)\) support the closed
 convex hull of \(\mathcal A_p\). We will find the region in \(\R_+^2\) defined by
 the corresponding inequalities and prove that it coincides with this hull.}

\textcolor{black}{Henceforth, assume \(1<p\le2\)\textcolor{black}{. The case} \(p=1\)
 will be treated separately in Remark~\ref{en:rem:attainable-p1-ru}.}

\noindent\textcolor{black}{\textit{Comment.} The many technical
 manipulations that follow can make it easy to lose sight of the main idea. What we are actually doing
 is this: starting from the supporting lines in directions with \(\alpha,\beta\ge0\),
 we use standard methods of convex analysis to find an explicit formula
 for the boundary of the region they define. This is precisely the purpose of the following calculations.}

\textcolor{black}{For \(\alpha>0\), set \(\lambda:=\frac{\beta}{\alpha}\ge0\).} The explicit
 formula for \(K_p\) gives \(K_p(\alpha,\beta)=\alpha K_p(1,\lambda)\). Thus
 all supporting inequalities can be written as
\(d+\lambda v\le K_p(1,\lambda)\), \(\lambda\ge0\), or, equivalently,
\(d\le K_p(1,\lambda)-\lambda v\).
Furthermore, \(v\ge0\) by definition, and choosing \((\alpha,\beta)=(0,1)\)
 gives \(v\le K_p(0,1)=1\) (that is, our set has no points to the right of the line
\(v=1\)). We can therefore write the \textcolor{black}{upper boundary of this outer region}
 as the lower envelope of these lines (\(v\in[0,1]\)):
\[
\begin{aligned}
F_p(v)
&:=\inf_{\lambda\ge0}
\bigl\{K_p(1,\lambda)-\lambda v\bigr\}
\qquad\text{(expand \(K_p(1,\lambda)\) using
 Theorem~\ref{en:thm:deterministic-ru}):}\\
&=\min\left\{
\inf_{\textcolor{black}{0\le\lambda\le1/p}}
\bigl(c(p)(1-\lambda)^{1-p}-\lambda v\bigr),
\inf_{\lambda\ge1/p}\lambda(1-v)
\right\}
\qquad
\textcolor{black}{\begin{gathered}
\text{(expand \(c(p)\) and write}\\
\text{the second expression explicitly):}
\end{gathered}}\\
&=\min\left\{
\inf_{\textcolor{black}{0\le\lambda\le1/p}}
\left(\frac{(p-1)^{p-1}}{p^p}(1-\lambda)^{1-p}-\lambda v\right),
\frac{1-v}{p}
\right\}.
\end{aligned}
\]
\textcolor{black}{The point \(\lambda=1/p\) is included in both branches because their
 values coincide there.} It remains to compute the first infimum. Differentiating the expression with respect to \(\lambda\)
 gives
\[
\frac{d}{d\lambda}
\left[
\frac{(p-1)^{p-1}}{p^p}(1-\lambda)^{1-p}-\lambda v
\right]
=\left(\frac{p-1}{p}\right)^p(1-\lambda)^{-p}-v.
\]
The derivative is increasing in \(\lambda\), and at \(\lambda=0\) it equals
\(\left(\frac{p-1}{p}\right)^p-v\). Thus, when
\(v\le\left(\frac{p-1}{p}\right)^p\), the minimum is attained
 at \(\lambda=0\). Otherwise, the derivative vanishes exactly once,
 and the minimum is attained at
\(\lambda_*(v)=1-\frac{p-1}{p}v^{-1/p}\)\textcolor{black}{, including the right-hand
 endpoint \(\lambda_*(1)=1/p\)}. All that remains is to substitute these
 values of \(\lambda\) into the formula for \(F_p(v)\):
\[
F_p(v)=
\begin{cases}
\displaystyle\frac{(p-1)^{p-1}}{p^p},
&\displaystyle 0\le v\le\left(\frac{p-1}{p}\right)^p,\\[3mm]
v^{(p-1)/p}-v,
&\displaystyle \left(\frac{p-1}{p}\right)^p\le v\le1.
\end{cases}
\]
We call the first branch of \(F_p\) the flat part of the boundary and the second
 its curved part. It is easy to check that \(F_p\) is concave on
\([0,1]\)\textcolor{black}{. Consequently,} its hypograph is convex.

\textcolor{black}{We now justify the reverse inclusion. The deterministic laws
\(X=\tau e\) and \(X=\frac{p}{p-1}\tau e\), where \(\norm{e}=1\), give
 the points \((0,0)\) and \((0,c(p))\), while the symmetric law on \(\pm\tau e\)
 gives \((1,0)\). For the laws \(X_{q,r}\) from the proof of
 Proposition~\ref{en:prop:stochastic-sharp-ru}, as \(q\downarrow0\), the pairs converge
 to \((r^{-p},(r-1)/r^p)\). For \(1\le r\le p/(p-1)\), these points trace out
 the entire curved part of the boundary. The flat part is the segment joining
\((0,c(p))\) to the left endpoint of this arc, and the lower side is the segment joining
\((0,0)\) and \((1,0)\). Thus convex combinations of boundary points fill
 the entire region \(0\le v\le1\), \(0\le d\le F_p(v)\), and this region is contained in
\(\overline{\operatorname{conv}}(\mathcal A_p)\).}
Therefore,
\(\overline{\operatorname{conv}}(\mathcal A_p)
=\{(v,d):0\le v\le1,\ 0\le d\le F_p(v)\}\).

The figure below shows the closed convex hulls of
\(\mathcal A_{3/2}\) and \(\mathcal A_{6/5}\). The flat parts of the boundaries are highlighted
 in green and the curved parts in orange.

\begin{figure}[H]
\centering
\begin{tikzpicture}[
  font=\small,line cap=round,line join=round,
  support/.style={draw=black!60,line width=.7pt},
  regionedge/.style={draw=geoPthreehalf,line width=1pt},
  flatpart/.style={draw=green!55!black,line width=1.55pt},
  curvedpart/.style={draw=orange!88!black,line width=1.55pt}]
  \begin{scope}[x=5.3cm,y=4.1cm]
    \pgfmathsetmacro{\q}{1/3}
    \pgfmathsetmacro{\a}{pow(1/3,3/2)}
    \pgfmathsetmacro{\c}{2/(3*sqrt(3))}
    \pgfmathsetmacro{\vzero}{.3}
    \pgfmathsetmacro{\fzero}{pow(\vzero,\q)-\vzero}
    \pgfmathsetmacro{\slope}{\q*pow(\vzero,\q-1)-1}
    \pgfmathsetmacro{\intercept}{\fzero-\slope*\vzero}

    \path[fill=geoPthreehalf!15]
      (0,0)--(0,\c)--(\a,\c)--
      plot[domain=\a:1,samples=80,smooth] (\x,{pow(\x,\q)-\x})--cycle;
    \draw[support] (0,\c)--(1.04,\c)
      node[pos=.91,above,font=\scriptsize,text=black!78] {$\ell_0$};
    \draw[support] (0,\intercept)--(1.04,{\intercept+\slope*1.04})
      node[pos=.88,above right,font=\scriptsize,text=black!78] {$\ell_1$};
    \draw[support] (1,-.025)--(1,.96)
      node[above,font=\scriptsize,text=black!78] {$\ell_2$};
    \draw[->,draw=black!80,line width=.65pt] (0,0)--(1.08,0)
      node[right,font=\normalsize,text=black] {$v$};
    \draw[->,draw=black!80,line width=.65pt] (0,0)--(0,1.07)
      node[above,font=\normalsize,text=black] {$d$};
    \draw[regionedge] (0,0)--(0,\c);
    \draw[regionedge] (0,0)--(1,0);
    \draw[flatpart] (0,\c)--(\a,\c);
    \draw[curvedpart]
      plot[domain=\a:1,samples=80,smooth] (\x,{pow(\x,\q)-\x});
    \draw[draw=black!62,densely dotted,line width=.65pt] (\a,0)--(\a,\c);
    \draw[draw=black!70,line width=.55pt] (\a,0)--(\a,-.012);
    \fill[black!72] (\a,\c) circle[radius=1.05pt];
    \fill[black!62] (\vzero,\fzero) circle[radius=.9pt];
    \node[below left,inner sep=3pt,font=\footnotesize] at (0,0) {$0$};
    \node[below,inner sep=3pt,font=\footnotesize] at (1,0) {$1$};
    \draw[draw=black!70,line width=.55pt] (0,1)--(-.012,1);
    \node[left,inner sep=3pt,font=\footnotesize] at (0,1) {$1$};
    \node[anchor=east,inner sep=3pt,text=black!82] at (-.005,\c)
      {$\displaystyle\frac{(\frac12)^{1/2}}{(\frac32)^{3/2}}$};
    \node[anchor=north,inner sep=2pt,font=\scriptsize,xshift=4pt]
      at (\a,-.012)
      {$\displaystyle\left(\frac13\right)^{3/2}$};
    \node[text=geoPthreehalf,font=\small] at (.50,.15)
      {$\overline{\operatorname{conv}}(\mathcal A_p)$};
    \node[font=\small] at (.54,1.13) {$p=\frac32$};
  \end{scope}

  \begin{scope}[xshift=7.1cm,x=5.3cm,y=4.1cm]
    \pgfmathsetmacro{\q}{1/6}
    \pgfmathsetmacro{\a}{pow(1/6,6/5)}
    \pgfmathsetmacro{\c}{5/pow(6,6/5)}
    \pgfmathsetmacro{\vzero}{.3}
    \pgfmathsetmacro{\fzero}{pow(\vzero,\q)-\vzero}
    \pgfmathsetmacro{\slope}{\q*pow(\vzero,\q-1)-1}
    \pgfmathsetmacro{\intercept}{\fzero-\slope*\vzero}

    \path[fill=geoPthreehalf!15]
      (0,0)--(0,\c)--(\a,\c)--
      plot[domain=\a:1,samples=80,smooth] (\x,{pow(\x,\q)-\x})--cycle;
    \draw[support] (0,\c)--(1.04,\c)
      node[pos=.91,above,font=\scriptsize,text=black!78] {$\ell_0$};
    \draw[support] (0,\intercept)--(1.04,{\intercept+\slope*1.04})
      node[pos=.88,above right,font=\scriptsize,text=black!78] {$\ell_1$};
    \draw[support] (1,-.025)--(1,.96)
      node[above,font=\scriptsize,text=black!78] {$\ell_2$};
    \draw[->,draw=black!80,line width=.65pt] (0,0)--(1.08,0)
      node[right,font=\normalsize,text=black] {$v$};
    \draw[->,draw=black!80,line width=.65pt] (0,0)--(0,1.07)
      node[above,font=\normalsize,text=black] {$d$};
    \draw[regionedge] (0,0)--(0,\c);
    \draw[regionedge] (0,0)--(1,0);
    \draw[flatpart] (0,\c)--(\a,\c);
    \draw[curvedpart]
      plot[domain=\a:1,samples=80,smooth] (\x,{pow(\x,\q)-\x});
    \draw[draw=black!62,densely dotted,line width=.65pt] (\a,0)--(\a,\c);
    \draw[draw=black!70,line width=.55pt] (\a,0)--(\a,-.012);
    \fill[black!72] (\a,\c) circle[radius=1.05pt];
    \fill[black!62] (\vzero,\fzero) circle[radius=.9pt];
    \node[below left,inner sep=3pt,font=\footnotesize] at (0,0) {$0$};
    \node[below,inner sep=3pt,font=\footnotesize] at (1,0) {$1$};
    \draw[draw=black!70,line width=.55pt] (0,1)--(-.012,1);
    \node[left,inner sep=3pt,font=\footnotesize] at (0,1) {$1$};
    \node[anchor=east,inner sep=3pt,text=black!82] at (-.005,\c)
      {$\displaystyle\frac{(\frac15)^{1/5}}{(\frac65)^{6/5}}$};
    \node[anchor=north,inner sep=2pt,font=\scriptsize,xshift=4pt]
      at (\a,-.012)
      {$\displaystyle\left(\frac16\right)^{6/5}$};
    \node[text=geoPthreehalf,font=\small] at (.50,.23)
      {$\overline{\operatorname{conv}}(\mathcal A_p)$};
    \node[font=\small] at (.54,1.13) {$p=\frac65$};
  \end{scope}
\end{tikzpicture}
\vspace{-1.5mm}
\caption{The closed convex hull of the set \(\mathcal A_p\)
 and its supporting lines for \(p=\frac32\) and \(p=\frac65\).}
\label{en:fig:support-hull-schematic-ru}
\end{figure}

We have now obtained a complete description of the closed convex hull. But what does
 the set of attainable pairs itself look like? We will find out in the next section.

\Needspace{14\baselineskip}
\subsection{The exact set of attainable pairs}
\label{en:subsec:attainable-set-ru}

\begin{theoremru}[Exact description of the attainable set]
\label{en:thm:attainable-set-ru}
For \(1<p\le2\), the attainable set is
\[
\mathcal A_p=
\overline{\operatorname{conv}}(\mathcal A_p)
\setminus
\left\{
\left(v,v^{\frac{p-1}{p}}-v\right):
\left(\frac{p-1}{p}\right)^p\le v<1
\right\}.
\]
In other words, we remove only the curved part of the upper boundary, including
its left endpoint, from the closure of the convex hull found above. The point
\((1,0)\) remains in the set.
\end{theoremru}

\begin{proof}
\textit{1. Unattainability of the arc.}
We first show that the curved part of the upper boundary is not contained in
\(\mathcal A_p\). To this end, normalize the threshold to \(\tau=1\)
(this entails no loss of generality:
\(C_1\!\left(\frac{X}{\tau}\right)=\frac{C_\tau(X)}{\tau}\),
and both normalized coordinates from Subsection 4.1 remain unchanged
under this substitution).
Suppose, to the contrary, that an admissible \(X\) yields a point
on the curved part, that is,
\[
\left(\frac{p-1}{p}\right)^p\le V_p(X)<1,
\qquad
D_p(X)=V_p(X)^{\frac{p-1}{p}}-V_p(X)>0.
\]
\begingroup\color{black}
\begin{lemmaru}[Pointwise bound]
Let \(1<p\le2\). For every \(x\in H\),
\[
\norm{x-C_1(x)}+\norm{C_1(x)}^2
\le \norm{x}\,\norm{C_1(x)}^{\frac{2(p-1)}{p}}.
\]
For \(1<p<2\), equality holds only when \(x=0\) or
\(\norm{x}\ge1\)\textcolor{black}{. For} \(p=2\), equality holds for every \(x\).
\end{lemmaru}
\endgroup

\begin{proof}
If \(\norm{x}>1\), then \(\norm{C_1(x)}=1\), and both sides equal
\(\norm{x}\). If \(\norm{x}\le1\), clipping leaves the vector unchanged,
and the bound becomes
\(\norm{x}^2\le\norm{x}^{1+\frac{2(p-1)}{p}}\).
This holds because \(1+\frac{2(p-1)}{p}\le2\). For \(p=2\),
this exponent equals \(2\), so equality holds here\textcolor{black}{.
For} \(1<p<2\), it is less than \(2\), and for \(0<\norm{x}<1\)
the inequality is strict.
\end{proof}

We now apply the pointwise lemma to bound \(D_p(X)\):
\[
\begin{aligned}
D_p(X)
&=\frac{\norm{\E[C_1(X)]-\E[X]}}{\E[\norm{X}^p]}
\underset{\mathrm{I}}{\le}
\frac{\E[\norm{X-C_1(X)}]}{\E[\norm{X}^p]}=\\[0.45em]
&=\frac{\E[\norm{X-C_1(X)}]+\E[\norm{C_1(X)}^2]
-\E[\norm{C_1(X)}^2]}{\E[\norm{X}^p]}
\underset{\mathrm{II}}{\le}\\[0.45em]
&\underset{\mathrm{II}}{\le}
\frac{\E[\norm{X}\,\norm{C_1(X)}^{\frac{2(p-1)}{p}}]
-\E[\norm{C_1(X)}^2]}{\E[\norm{X}^p]}
\underset{\mathrm{III}}{\le}\\[0.45em]
&\underset{\mathrm{III}}{\le}
\frac{\E[\norm{X}^p]^{\frac1p}\E[\norm{C_1(X)}^2]^{\frac{p-1}{p}}
-\E[\norm{C_1(X)}^2]}{\E[\norm{X}^p]}=\\[0.45em]
&=\left(\frac{\E[\norm{C_1(X)}^2]}{\E[\norm{X}^p]}\right)^{\frac{p-1}{p}}
-\frac{\E[\norm{C_1(X)}^2]}{\E[\norm{X}^p]}
\underset{\mathrm{IV}}{\le}\\[0.45em]
&\underset{\mathrm{IV}}{\le}
V_p(X)^{\frac{p-1}{p}}-V_p(X)=D_p(X).
\end{aligned}
\]
\noindent Step \(\mathrm{I}\) uses the inequality
\(\norm{\E[Y]}\le\E[\norm{Y}]\).

\noindent Step \(\mathrm{II}\) uses the pointwise lemma
after taking expectations.

\noindent In step \(\mathrm{III}\), H\"older's inequality is applied
to the factors \(\norm{X}\) and
\(\norm{C_1(X)}^{\frac{2(p-1)}{p}}\) with exponents
\(p\) and \(\frac{p}{p-1}\).

\noindent Step \(\mathrm{IV}\) follows from the variance
identity and the assumption that the point lies on the curved part:
\[
\makebox[\textwidth][c]{$\displaystyle
\frac{\E[\norm{C_1(X)}^2]}{\E[\norm{X}^p]}
=\frac{\E[\norm{C_1(X)-\E[C_1(X)]}^2]}{\E[\norm{X}^p]}
+\frac{\norm{\E[C_1(X)]}^2}{\E[\norm{X}^p]}
=V_p(X)+\frac{\norm{\E[C_1(X)]}^2}{\E[\norm{X}^p]}
\ge V_p(X)\ge\left(\frac{p-1}{p}\right)^p.$}
\]
\vspace{-2pt}
The function \(f(s)=s^{\frac{p-1}{p}}-s\) is strictly decreasing on
\(\left[\left(\frac{p-1}{p}\right)^p,\infty\right)\). Hence
the preceding bound gives
\(f\!\left(\frac{\E[\norm{C_1(X)}^2]}{\E[\norm{X}^p]}\right)
\le f(V_p(X))\), which yields step \(\mathrm{IV}\).

\bigskip

The first and last terms of the full chain coincide: since we are considering
a point on the curved boundary, \(D_p(X)=f(V_p(X))\).
Thus equality holds in every step \(\mathrm{I}\)--\(\mathrm{IV}\).

Since \(f\) is strictly decreasing, equality in step
\(\mathrm{IV}\) requires the additional term
\(\dfrac{\norm{\E[C_1(X)]}^2}{\E[\norm{X}^p]}\) to be zero.
Its denominator is positive, so~\(\E[C_1(X)]=0\).

\bigskip

Equality in step \(\mathrm{II}\) means that equality in the
pointwise lemma holds almost surely: the difference between its right-hand and
left-hand sides is nonnegative, and its expectation is zero.

\medskip

Equality in step \(\mathrm{III}\) and the equality condition in
H\"older's inequality mean that the ratio
\(\frac{\norm{X}^p}{\norm{C_1(X)}^2}\) is almost surely constant
where \(X\ne0\). Since \(D_p(X)>0\), we have
\(\mathbb P(\norm{X}>1)>0\) (otherwise \(\norm{X}\le1\) almost surely
and \(C_1(X)=X\), giving \(D_p(X)=0\)). On this event,
\(\norm{C_1(X)}=1\), so constancy of this ratio
requires the same length \(r>1\) almost surely.

\Needspace{8\baselineskip}
Let us show that, almost surely, no other nonzero values occur. When
\(\norm{X}=1\), the ratio from step \(\mathrm{III}\) equals \(1\),
whereas on outcomes with \(\norm{X}=r>1\) it equals \(r^p>1\)\textcolor{black}{.
This} contradicts constancy of the ratio. For \(0<\norm{X}<1\),
we distinguish two cases for the exponent: \mbox{if \(1<p<2\),} the pointwise
inequality is strict, contradicting equality in step
\(\mathrm{II}\)\textcolor{black}{.} \mbox{\textcolor{black}{If} \(p=2\), then}
\mbox{\(\dfrac{\norm{X}^2}{\norm{C_1(X)}^2}
=\dfrac{\norm{X}^2}{\norm{X}^2}=1\)}, contradicting
step \(\mathrm{III}\). Consequently,
almost surely either \(X=0\) or \(\norm{X}=r>1\). Hence almost surely,
including outcomes with \(X=0\), we have \(X=rC_1(X)\).

\begin{samepage}
Taking expectations gives
\[
\E[X]=r\E[C_1(X)]=0.
\]
Thus \(D_p(X)=0\), contradicting \(D_p(X)>0\). Therefore,
no point of the specified curved arc is attained.
\end{samepage}

\textit{2. Attainability of every point.}
We now show that every point of the closed convex region found above
outside the excluded arc is attained. Choose a unit vector \(e\in\textcolor{black}{H}\).
Fix a height \(d\in[0,c(p)]\), and let \(s_d\) denote the right
endpoint of the horizontal section of the region at this height. In other words,
\(s_d\) indicates how far to the right one can move at fixed
\(d\) while remaining in the region:
\[
s_d:=\max\bigl\{t\in[0,1]:(t,d)\in
\overline{\operatorname{conv}}(\mathcal A_p)\bigr\}.
\]
We consider two cases.

\textit{Case 1: \(d=0\).}
Here \(0\le v\le1\). Let \(X=\tau e\) with probability
\(\frac{1+\sqrt{1-v}}2\) and \(X=-\tau e\) with the remaining probability.
Then \(C_\tau(X)=X\), so \(D_p(X)=0\), and a direct calculation
gives \(V_p(X)=v\). Thus the entire lower side of the region is attained,
including the point \((1,0)\).

\textit{Case 2: \(0<d\le c(p)\).}
The right branch of the boundary found above decreases strictly from \(c(p)\) to zero.
Hence its point at height \(d\) has horizontal coordinate \(s_d\), with
\[
\left(\frac{p-1}{p}\right)^p\le s_d<1,
\qquad d=s_d^{\frac{p-1}{p}}-s_d
\quad\text{(by the formula for the curved part of the boundary).}
\]
At height \(d\), the region occupies the interval \(0\le v\le s_d\)
(this also holds for the flat part of the boundary). Since the right endpoint
\((s_d,d)\) lies on the arc already excluded, it remains to realize
the points \(0\le v<s_d\).

\textcolor{black}{We use the same construction as in the proof of
Proposition~\ref{en:prop:stochastic-sharp-ru}: the radius determines the height \(d\),
and the probability of a nonzero value determines the position on the horizontal section.
Thus we take} \(X=s_d^{-1/p}\tau e\) with
probability \(1-\frac{v}{s_d}>0\), and \(X=0\) with probability
\(\frac{v}{s_d}\). This law is admissible because
\(\E[\norm{X}^p]=\frac{\tau^p}{s_d}\left(1-\frac{v}{s_d}\right)>0\).
Computing the normalized coordinates gives the desired point \((v,d)\):
\begin{align*}
V_p(X)
&=\dfrac{\left(1-\dfrac{v}{s_d}\right)\dfrac{v}{s_d}}
{\left(1-\dfrac{v}{s_d}\right)\dfrac{1}{s_d}}=v,\\[6pt]
D_p(X)
&=\dfrac{\left(1-\dfrac{v}{s_d}\right)\left(s_d^{-1/p}-1\right)}
{\left(1-\dfrac{v}{s_d}\right)\dfrac{1}{s_d}}
=s_d^{\frac{p-1}{p}}-s_d=d.
\end{align*}
Thus every point to the left of \((s_d,d)\) has been realized, and
the unattainability of the right endpoint has already been proved. This completes
the proof.

\end{proof}

\noindent\textcolor{black}{\textit{Comment.} These calculations are based on
a simple idea. If the bias is to be zero, take two opposite
values on the clipping sphere with complementary probabilities.
If positive bias is required, take one value outside the ball
and the value zero.}

\begin{remarkru}[The special case \(p=1\)]
\label{en:rem:attainable-p1-ru}
In this case,
\[
\mathcal A_1=
\{(v,d)\in\R_+^2:v+d\le1\}\setminus\{(0,1)\}.
\]
\end{remarkru}

\begin{proof}
Theorem~\ref{en:thm:stochastic-ru} and \(K_1(1,1)=1\) imply
\(V_1(X)+D_1(X)\le1\). Moreover,
\[
D_1(X)\le1-\frac{\E[\norm{C_\tau(X)}]}{\E[\norm{X}]}<1,
\]
so \((0,1)\) is unattainable.
Fix a unit vector \(e\). For \(d<1\) and \(v<1-d\),
the desired pair is given by \(X=\frac{\tau}{1-d}e\) with probability
\(1-\frac{v}{1-d}\) and \(X=0\) otherwise.
On the side \(v+d=1\), \(v>0\), take the law assigning equal probability to
\(\frac{1+d}{1-d}\tau e\) and \(-\tau e\).
\end{proof}

Having fully described \(\mathcal A_p\), we can find supporting lines
in every direction, that is, also for coefficients \(\alpha,\beta\) of arbitrary
signs, rather than only for \(\alpha,\beta\ge0\) as above:

\begin{corollaryru}[Full support function]
\label{en:cor:signed-support-ru}
Let \(1<p\le2\) and \(\alpha,\beta\in\R\). Then
\[
\sup_{(v,d)\in\mathcal A_p}(\textcolor{black}{\beta v+\alpha d})=
\begin{cases}
0, & \beta<0,\ \alpha\le0,\\[1mm]
\alpha c(p), & \beta<0,\ \alpha>0,\\[1mm]
\beta, & \beta\ge0,\ \alpha\le p\beta,\\[1mm]
\displaystyle
c(p)\frac{\alpha^p}{(\alpha-\beta)^{p-1}},
& \beta\ge0,\ \alpha>p\beta.
\end{cases}
\]
For \(p=1\), this supremum equals \(\max\{0,\alpha,\beta\}\).
\end{corollaryru}

Now that we have an exact description of the set of attainable pairs and its supporting lines,
let us look at how energy and bias relate to one another for random vectors
and what this geometry gives us in applications:

\begin{figure}[H]
\centering
\begin{tikzpicture}[
  x=2.70cm,y=2.70cm,font=\scriptsize,line cap=round,line join=round]

  \begin{scope}[xshift=0cm]
    \path[fill=geoPtwo!18]
      (0,0)--(0,.25)--(.25,.25)
      -- plot[domain=.25:1,samples=90,smooth]
        (\x,{sqrt(\x)-\x}) --cycle;
    \draw[geo axis] (0,0)--(1.06,0) node[right] {\(v\)};
    \draw[geo axis] (0,0)--(0,1.06) node[above] {\(d\)};
    \draw[geo attainable,draw=geoPtwo]
      (0,0)--(0,.25)--(.25,.25);
    \draw[geo attainable,draw=geoPtwo] (0,0)--(1,0);
    \draw[geo unattainable]
      plot[domain=.25:1,samples=90,smooth]
        (\x,{sqrt(\x)-\x});
    \node[geo excluded] at (.25,.25) {};
    \node[geo included] at (1,0) {};
    \node at (.52,1.15) {\(\bm{p=2}\)};
  \end{scope}

  \begin{scope}[xshift=3.55cm]
    \def\a{0.19245009}
    \def\c{0.38490018}
    \path[fill=geoPthreehalf!21]
      (0,0)--(0,\c)--(\a,\c)
      -- plot[domain=\a:1,samples=90,smooth]
        (\x,{pow(\x,0.3333333333)-\x}) --cycle;
    \draw[geo axis] (0,0)--(1.06,0) node[right] {\(v\)};
    \draw[geo axis] (0,0)--(0,1.06) node[above] {\(d\)};
    \draw[geo attainable,draw=geoPthreehalf]
      (0,0)--(0,\c)--(\a,\c);
    \draw[geo attainable,draw=geoPthreehalf] (0,0)--(1,0);
    \draw[geo unattainable]
      plot[domain=\a:1,samples=90,smooth]
        (\x,{pow(\x,0.3333333333)-\x});
    \node[geo excluded] at (\a,\c) {};
    \node[geo included] at (1,0) {};
    \node at (.52,1.15) {\(\bm{p=\frac32}\)};
  \end{scope}

  \begin{scope}[xshift=7.10cm]
    \def\a{0.11647119}
    \def\c{0.58235593}
    \path[fill=geoPsixfifths!23]
      (0,0)--(0,\c)--(\a,\c)
      -- plot[domain=\a:1,samples=90,smooth]
        (\x,{pow(\x,0.1666666667)-\x}) --cycle;
    \draw[geo axis] (0,0)--(1.06,0) node[right] {\(v\)};
    \draw[geo axis] (0,0)--(0,1.06) node[above] {\(d\)};
    \draw[geo attainable,draw=geoPsixfifths]
      (0,0)--(0,\c)--(\a,\c);
    \draw[geo attainable,draw=geoPsixfifths] (0,0)--(1,0);
    \draw[geo unattainable]
      plot[domain=\a:1,samples=90,smooth]
        (\x,{pow(\x,0.1666666667)-\x});
    \node[geo excluded] at (\a,\c) {};
    \node[geo included] at (1,0) {};
    \node at (.52,1.15) {\(\bm{p=\frac65}\)};
  \end{scope}

  \begin{scope}[xshift=10.65cm]
    \path[fill=geoPone!32] (0,0)--(0,1)--(1,0)--cycle;
    \draw[geo axis] (0,0)--(1.06,0) node[right] {\(v\)};
    \draw[geo axis] (0,0)--(0,1.06) node[above] {\(d\)};
    \draw[geo attainable,draw=geoPone]
      (0,0)--(0,1)--(1,0)--(0,0);
    \node[geo excluded] at (0,1) {};
    \node[geo included] at (1,0) {};
    \node at (.52,1.15) {\(\bm{p=1}\)};
  \end{scope}
\end{tikzpicture}
\caption{Structure of the set \(\mathcal A_p\) for
\(p=2,\ \frac32,\ \frac65,\ 1\).}
\label{en:fig:attainable-sets-ru}
\end{figure}

As we pass from \(p=1\) to \(p=2\), the region of admissible normalized pairs
shrinks. This agrees with familiar probabilistic intuition: finiteness
of a higher moment imposes a stronger constraint on the tails
of the distribution (although the normalization also changes with \(p\)). In particular,
if the second moment is finite, the normalized bias \(D_2(X)\) does not exceed
\(1/4\), whereas a crude bound on the residual would give only \(D_2(X)\leq 1\).

The plots also clearly show that energy and bias cannot simultaneously
attain their maximum values: at maximum energy, the bias
is zero, while large bias requires reducing the energy. This is precisely
why it makes sense to bound them jointly. Adding separate worst-case
bounds ignores this relation and may give an overly large bound. Our result
accounts for this relation and gives an exact bound on their weighted sum.

\Needspace{7\baselineskip}
We would also like to point out that realizing an attainable pair
\((v,d)\) does not require complicated random vectors. The explicit constructions above already
give a fairly strong result:

\begin{corollaryru}[Two-atom realization]
\label{en:cor:two-atom-ru}
Every attainable pair in \(\mathcal A_p\) is realized by a real random
variable taking at most two values.
\end{corollaryru}

\begin{proof}
All realizations in the proof of Theorem~\ref{en:thm:attainable-set-ru}
and in Remark~\ref{en:rem:attainable-p1-ru}
are one-dimensional and have at most two atoms.
\end{proof}

\subsection{Simultaneous realization at all thresholds}
\label{en:subsec:all-thresholds-ru}

Corollary~\ref{en:cor:two-atom-ru} shows that, at a fixed threshold,
every attainable pair can be realized by a one-dimensional two-atom random
variable. What changes when the threshold varies? Can a single random
vector reproduce the pairs simultaneously for all \(\tau>0\)? The following
theorem answers this question.

\begin{theoremru}[Simultaneous three-dimensional realization]
\label{en:thm:all-thresholds-ru}
Let \(\textcolor{black}{H}\) be a nonzero real Hilbert space,
\textcolor{black}{and let \(X\in L^1(\Omega;H)\)}. Then there exists an
\(\R^3\)-valued random vector \(Z\) satisfying the following conditions
(the probability spaces on which \(X\) and \(Z\) are defined may
differ): the random variables \(\norm{X}\) and \(\norm{Z}\) have the same
distribution, and for every \(\tau>0\),
\begin{align*}
\norm{\E[C_\tau(Z)]-\E[Z]}
&=\norm{\E[C_\tau(X)]-\E[X]},\\
\E\!\left[\norm{C_\tau(Z)-\E[C_\tau(Z)]}^2\right]
&=\E\!\left[\norm{C_\tau(X)-\E[C_\tau(X)]}^2\right].
\end{align*}
\end{theoremru}

In other words, no matter how complicated the law of the original random vector is,
its bias and energy at all thresholds can be reproduced by a single
three-dimensional random vector.

\begin{proof}
\textit{1. The mean residual function as \(\tau\) varies.}
In the theorem, we vary the threshold \(\tau\) while keeping the random vector \(X\)
fixed. To keep track, for each \(\tau\), of the information about \(X\)
that matters for the bias, let us introduce the function
\(R(\tau)=\E[X-C_\tau(X)]\). It is the mean of the part of \(X\)
removed by clipping\textcolor{black}{. Its} norm is precisely the bias. Below, we describe
\textcolor{black}{the form and properties of this function.}

\textcolor{black}{We will then construct a polygonal path in \(\R^3\), and recover the probabilities
of opposite atoms from the jumps of its derivative. For these probabilities
to be nonnegative, the norm of each jump must not exceed the mass
of the corresponding sphere. Once we have treated the finite case, we will pass to an arbitrary law.}

\textcolor{black}{First suppose that \(\norm{X}\) takes finitely many
values. If \(X=0\) almost surely, take \(Z=0\). Otherwise, denote
the positive values of the norm that occur with nonzero probability by
\(0<r_1<\cdots<r_m\).} The function
\(R\) is then given by
\begin{gather*}
R(\tau)=\E[X-C_\tau(X)]
=\sum_{i=1}^{m}
\E\!\left[\frac{X}{r_i}\,\middle|\,\norm{X}=r_i\right]
\,\mathbb P(\norm{X}=r_i)\,\max\{r_i-\tau,0\},\\[-2pt]
\text{(by the definition of clipping, for }x\ne0\text{ we have }
x-C_\tau(x)=\frac{x}{\norm{x}}\,\max\{\norm{x}-\tau,0\}\text{\textcolor{black}{.}}\\[-2pt]
\text{\textcolor{black}{Partitioning} outcomes by the values }\norm{X}=r_i\text{ and taking expectations gives the formula above).}
\end{gather*}
Also define \(R(0)=\E[X]\) \textcolor{black}{(this is well defined because,
as \(\tau\downarrow0\), the length of the clipped vector does not
exceed \(\tau\), so both \(C_\tau(X)\) and its
expectation tend to zero. Consequently,
\(R(\tau)=\E[X]-\E[C_\tau(X)]\) tends to \(\E[X]\)).}

We now examine the properties of \(R\):

\textit{\textcolor{black}{(P1)} Values and continuity.}
For every \(\tau\ge0\), the value \(R(\tau)\) lies in the same space
\(\textcolor{black}{H}\) as \(X\). Continuity of \(R\) follows directly from its
form: each factor \(\max\{r_i-\tau,0\}\) is continuous in \(\tau\).

\textit{\textcolor{black}{(P2)} Value at large thresholds.}
Also, for \(\tau\ge r_m\), we have \(R(\tau)=0\) (then
\(C_\tau(X)=X\) almost surely, so clipping produces no bias).

\textit{\textcolor{black}{(P3)} Linearity between radii.}
On each interval \((r_j,r_{j+1})\), each factor
\(\max\{r_i-\tau,0\}\) is either identically zero or equals
\(r_i-\tau\). Hence on such an interval, \(R(\tau)\) depends linearly
on \(\tau\), possibly with a constant term, and its
derivative is constant. The same holds on the intervals \((0,r_1)\) and
\((r_m,\infty)\). Thus the derivative can change only
at the points \(r_i\).

\Needspace{8\baselineskip}
\textit{\textcolor{black}{(P4)} Change in the derivative.}
We now examine how much the derivative can change. On the indicated
intervals, the derivative is defined, and a calculation gives
\[
R'(\tau)=-\sum_{r_i>\tau}\mathbb P(\norm{X}=r_i)\,
\E\!\left[\frac{X}{r_i}\,\middle|\,\norm{X}=r_i\right],
\qquad \tau>0,\quad \tau\notin\{r_1,\ldots,r_m\}.
\]
As we pass through \(r_i\), exactly one term disappears from the
expression above. Therefore, for the right and left limits of the derivative at \(r_i\),
denoted by \(R'(r_i+0)\) and \(R'(r_i-0)\), we obtain
\[
R'(r_i+0)-R'(r_i-0)
=\mathbb P(\norm{X}=r_i)\,
\E\!\left[\frac{X}{r_i}\,\middle|\,\norm{X}=r_i\right].
\]
Taking norms on both sides gives
\[
\norm{R'(r_i+0)-R'(r_i-0)}
=\mathbb P(\norm{X}=r_i)\,
\norm{\E\!\left[\frac{X}{r_i}\,\middle|\,\norm{X}=r_i\right]}
\le\mathbb P(\norm{X}=r_i).
\]
\textcolor{black}{The last inequality holds because, on the event
\(\norm{X}=r_i\), the vector \(X/r_i\) has unit norm, and the norm of its
conditional mean does not exceed one.}
Thus we have described the function \(R(\tau)\): its values, linearity
between the possible lengths of \(X\), and changes in its derivative.

\textcolor{black}{To help the reader understand the structure of this function,
consider an example in \(\R^3\).} Let
\(e_1,e_2,e_3\) be the standard basis vectors, and let \(X\) take
the values \(e_1,2e_2,3e_3,\frac45(3e_1-4e_2)\)
with probability \(1/4\) each. Here \(r_1=1\), \(r_2=2\),
\(r_3=3\), \(r_4=4\).
As \(\tau\) increases, the endpoint of the vector \(R(\tau)\) traverses four
straight segments, and for \(\tau\ge4\) it remains at zero.

\begin{figure}[H]
\centering
\begin{tikzpicture}[
  x={(-3.22cm,-1.15cm)},y={(2.76cm,-1.15cm)},z={(0cm,4.6cm)},
  line cap=round,line join=round,font=\scriptsize,
  residualaxis/.style={draw=black!65,line width=.5pt,
    -{Stealth[length=3.3pt,width=2.8pt]}},
  residualpath/.style={draw=geoPtwo,line width=1.05pt},
  residualarrow/.style={decoration={markings,
    mark=at position .56 with {\arrow{Stealth[length=4pt,width=3pt]}}},
    postaction={decorate}}]
  \coordinate (O) at (0,0,0);
  \coordinate (Rfour) at (0,0,0);
  \coordinate (Rthree) at (0.15,-0.20,0);
  \coordinate (Rtwo) at (0.30,-0.40,0.25);
  \coordinate (RzeroXY) at (0.85,-0.30,0);
  \coordinate (RoneXY) at (0.45,-0.35,0);
  \coordinate (RtwoXY) at (0.30,-0.40,0);
  \coordinate (Rone) at (0.45,-0.35,0.50);
  \coordinate (Rzero) at (0.85,-0.30,0.75);
  \draw[residualaxis] (O)--(1.15,0,0) node[below left] {$x$};
  \draw[residualaxis] (O)--(0,0.90,0) node[below right] {$y$};
  \draw[residualaxis] (O)--(0,0,0.85) node[above] {$z$};
  \draw[draw=black!25,line width=.35pt]
    (RzeroXY)--(RoneXY)--(RtwoXY)--(Rthree);
  \draw[draw=black!35,line width=.35pt,densely dashed]
    (Rzero)--(RzeroXY) (Rone)--(RoneXY) (Rtwo)--(RtwoXY);
  \foreach \P in {RzeroXY,RoneXY,RtwoXY}
    \fill[black!40] (\P) circle[radius=.7pt];
  \foreach \A/\B in {Rzero/Rone,Rone/Rtwo,Rtwo/Rthree,Rthree/O}
    \draw[residualpath,residualarrow] (\A)--(\B);
  \foreach \P in {Rzero,Rone,Rtwo,Rthree,O}
    \fill[geoPtwo] (\P) circle[radius=1.25pt];
  \node[above left,inner sep=3pt] at (Rzero) {$R(0)$};
  \node[above right,inner sep=3pt] at (Rone) {$R(r_1)$};
  \node[left,inner sep=4pt] at (Rtwo) {$R(r_2)$};
  \node[above right,inner sep=3pt] at (Rthree) {$R(r_3)$};
  \node[above right,inner sep=3pt] at (O) {$R(r_4)=0$};
\end{tikzpicture}
\vspace{-1.5mm}
\caption{An example of the trajectory of the function \(R(\tau)\) in \(\R^3\).}
\label{en:fig:residual-path-ru}
\end{figure}

\textit{2. Geometric formulation of the problem.}
For the vector \(Z\) we wish to construct, denote its mean residual by
\(\widetilde R(\tau):=\E[Z-C_\tau(Z)]\). Let us see which conditions
on \(\widetilde R\) ensure the two equalities in the theorem. The first
equality means precisely that, for every \(\tau>0\),
\(\norm{\widetilde R(\tau)}=\norm{R(\tau)}\).

For the second equality, observe that if
\(\norm{X}\) and \(\norm{Z}\) have the same distribution, then the quantities
\(\E[\norm{C_\tau(X)}^2]\) and \(\E[\norm{C_\tau(Z)}^2]\) coincide: indeed,
\(\norm{C_\tau(X)}=\min\{\norm{X},\tau\}\), and the same holds for \(Z\).
\textcolor{black}{By the variance identity, the energy equals
\[
\E[\norm{C_\tau(X)}^2]-\norm{\E[C_\tau(X)]}^2.
\]
The first term already coincides for \(X\) and \(Z\), so to preserve
the energy it suffices to preserve the norm of the mean clipped vector.}
For the vector \(Z\) to be constructed, we also set \(\widetilde R(0)=\E[Z]\).
Since \(\E[C_\tau(X)]=R(0)-R(\tau)\) and
\(\E[C_\tau(Z)]=\widetilde R(0)-\widetilde R(\tau)\), the second equality
in the theorem is equivalent to the condition
\(\norm{\widetilde R(0)-\widetilde R(\tau)}
=\norm{R(0)-R(\tau)}\).

\textcolor{black}{Thus we need to construct a new polygonal path in \(\R^3\):
for every value of \(\tau\), the corresponding point must have the same distances
from zero and from the initial vertex as the corresponding point of the original path.
We then need to verify that the new polygonal path is indeed
the mean residual function of some random vector \(Z\) whose norm has the same
distribution as that of \(X\).}

\textit{3. Construction of the three-dimensional realization.}
We now explicitly construct the polygonal path \(\widetilde R\) in \(\R^3\).
Set \(0=r_0\). The original vertices are
\(R(r_0),R(r_1),\ldots,R(r_m)\), with \(R(r_0)=\E[X]\)
and \(R(r_m)=0\).

Introduce a coordinate system in \(\R^3\), and set
\(0=(0,0,0)\), \(\widetilde R(0)=(\norm{\E[X]},0,0)\).
Also fix the plane \(xOy\)\textcolor{black}{. The points} \(0\) and
\(\widetilde R(0)\) already lie in it. Place the next vertex so that
\[
\begin{aligned}
\widetilde R(r_1)&\in xOy,\\
\norm{\widetilde R(r_1)}&=\norm{R(r_1)},\\
\norm{\widetilde R(r_1)-\widetilde R(0)}
&=\norm{R(r_1)-R(0)}.
\end{aligned}
\]
This is possible: the triangle \(0,R(0),R(r_1)\) can be placed
in the plane \(xOy\) while preserving its side lengths, even
if some vertices coincide.
Note that we will add all subsequent vertices in the same
space \(\R^3\) fixed above, without moving the points already constructed.

\textcolor{black}{To construct the next vertex, we need the following lemma.}

\begingroup\color{black}
\begin{lemmaru}
Let \(1\le i<m\), and suppose that the points
\(0,\widetilde R(0),\widetilde R(r_{i-1}),\widetilde R(r_i)\)
(not necessarily distinct) are given in \(\R^3\) and that, for all
\(j,k\in\{0,i-1,i\}\), the following equalities hold:
\begin{align*}
\norm{\widetilde R(r_j)}&=\norm{R(r_j)},\\
\norm{\widetilde R(r_j)-\widetilde R(r_k)}&=\norm{R(r_j)-R(r_k)}.
\end{align*}
(In the first application, \(i=1\), we have
\(\widetilde R(r_{i-1})=\widetilde R(r_0)=\widetilde R(0)\), so
the already placed points \(0\), \(\widetilde R(0)\), and
\(\widetilde R(r_1)\) suffice.)
Then one can choose \(\widetilde R(r_{i+1})\in\R^3\)
so that the following three equalities hold:
\begin{equation}
\label{en:eq:next-vertex-three-distances-ru}
\begin{aligned}
\norm{\widetilde R(r_{i+1})}&=\norm{R(r_{i+1})},\\
\norm{\widetilde R(r_{i+1})-\widetilde R(0)}&=\norm{R(r_{i+1})-R(0)},\\
\norm{\widetilde R(r_{i+1})-\widetilde R(r_i)}&=\norm{R(r_{i+1})-R(r_i)}.
\end{aligned}
\mathrlap{\hspace{0.25em}\color{black!75}\left(
\begin{gathered}
\text{\makebox[3.05cm][c]{preserving}}\\
\text{the two distances}\\
\text{from step 2}\\
\text{and segment length}
\end{gathered}
\right)}
\end{equation}
Moreover, when the adjacent segments are traversed linearly over the same threshold
intervals \([r_{i-1},r_i]\) and \([r_i,r_{i+1}]\), for every
\(\tau\in[r_{i-1},r_{i+1}]\) the following equalities hold:
\begin{equation}
\label{en:eq:next-vertex-interior-distances-ru}
\begin{aligned}
\norm{\widetilde R(\tau)}&=\norm{R(\tau)},\\
\norm{\widetilde R(\tau)-\widetilde R(0)}&=\norm{R(\tau)-R(0)}.
\end{aligned}
\mathrlap{\hspace{1.75em}\color{black!75}\left(
\begin{gathered}
\text{\makebox[3.05cm][c]{the first two}}\\
\text{equalities }\eqref{en:eq:next-vertex-three-distances-ru}\\
\text{at all points of}\\
\text{adjacent segments}
\end{gathered}
\right)}
\end{equation}
The jump in the derivative satisfies the bound
(thus preserving the constraint from property \textcolor{black}{(P4)}):
\begin{equation}
\label{en:eq:next-vertex-jump-bound-ru}
\norm{\widetilde R'(r_i+0)-\widetilde R'(r_i-0)}
\le \norm{R'(r_i+0)-R'(r_i-0)}.
\end{equation}
\end{lemmaru}
\endgroup

\begin{proof}
By the assumptions of the lemma, the points \(0\), \(\widetilde R(0)\), and
\(\widetilde R(r_i)\) are already fixed. Choose the next vertex
\(\widetilde R(r_{i+1})\) in the same \(\R^3\), preserving
the three distances in~\eqref{en:eq:next-vertex-three-distances-ru}.
Such a position exists: the quadruple
\(0,R(0),R(r_i),R(r_{i+1})\) has affine dimension
at most three, and the corresponding first three points in \(\R^3\)
have already been placed with distances preserved.
\textcolor{black}{Among the admissible positions, choose the new vertex so
that \(\widetilde R\) preserves the constraint on the jump in the derivative
from property \textcolor{black}{(P4)} of the function \(R\). We will then verify the equalities~
\eqref{en:eq:next-vertex-interior-distances-ru} and the bound~
\eqref{en:eq:next-vertex-jump-bound-ru}. We consider the possible configurations
of the reference points.}

\medskip
\noindent\textit{First case:} the affine rank of the points
\(0,\widetilde R(0),\widetilde R(r_i)\) is two, so they
determine a plane.
\textcolor{black}{The three prescribed distances determine the projection of the new vertex
onto this plane and its distance from it. Thus there are
two mirror-image positions outside the plane, or one position in the plane itself.}

\smallskip
\noindent
\begin{minipage}[t]{.46\textwidth}
\vspace{0pt}
\noindent\textit{a) \(\widetilde R(r_{i+1})\) lies in this plane.}
Its position satisfying~\eqref{en:eq:next-vertex-three-distances-ru}
is unique\textcolor{black}{. Choose} this position.

\smallskip
\noindent\textit{b) \(\widetilde R(r_{i+1})\) does not lie in this plane.}
There are two positions related by reflection in this plane.
If \(\widetilde R(r_{i-1})\) is outside the plane, choose
the position on the opposite side of the plane from
\(\widetilde R(r_{i-1})\). This choice is shown in
Fig.~\ref{en:fig:next-vertex-plane-ru}. If \(\widetilde R(r_{i-1})\) lies
in the plane, either position is suitable.
\end{minipage}\hfill
\begin{minipage}[t]{.51\textwidth}
\vspace{0pt}
\centering
\begin{tikzpicture}[
  x={(1.65cm,0cm)},y={(.45cm,.55cm)},z={(0cm,1.18cm)},
  line cap=round,line join=round,font=\scriptsize,
  chosenpath/.style={draw=geoPtwo,line width=1.05pt},
  patharrow/.style={decoration={markings,
    mark=at position .56 with {\arrow{Stealth[length=4pt,width=3pt]}}},
    postaction={decorate}}]
  \coordinate (planeA) at (-1.4,-1,0);
  \coordinate (planeB) at (1.25,-1,0);
  \coordinate (planeC) at (1.25,1,0);
  \coordinate (planeD) at (-1.4,1,0);
  \coordinate (origin) at (-1.08,-.52,0);
  \coordinate (start) at (-.95,.65,0);
  \coordinate (current) at (.15,-.10,0);
  \coordinate (previous) at (-.75,.10,1.15);
  \coordinate (previousProjection) at (-.75,.10,0);
  \coordinate (chosen) at (.75,.15,-1.15);
  \coordinate (chosenProjection) at (.75,.15,0);
  \coordinate (mirror) at (.75,.15,1.15);
  \fill[geoPtwo!6] (planeA)--(planeB)--(planeC)--(planeD)--cycle;
  \draw[draw=geoPtwo!35,line width=.5pt]
    (planeA)--(planeB)--(planeC)--(planeD)--cycle;
  \fill[cyan!40!blue,fill opacity=.035]
    (origin)--(start)--(mirror)--cycle;
  \fill[cyan!40!blue,fill opacity=.035]
    (start)--(current)--(mirror)--cycle;
  \fill[cyan!40!blue,fill opacity=.035]
    (current)--(origin)--(mirror)--cycle;
  \fill[cyan!40!blue,fill opacity=.05]
    (origin)--(start)--(chosen)--cycle;
  \fill[cyan!40!blue,fill opacity=.05]
    (start)--(current)--(chosen)--cycle;
  \fill[cyan!40!blue,fill opacity=.05]
    (current)--(origin)--(chosen)--cycle;
  \draw[draw=geoPtwo!24,line width=.45pt]
    (origin)--(start)--(current)--cycle;
  \draw[draw=cyan!40!blue,draw opacity=.23,line width=.4pt,densely dashed]
    (origin)--(mirror) (start)--(mirror) (current)--(mirror);
  \draw[draw=cyan!40!blue,draw opacity=.3,line width=.45pt]
    (origin)--(chosen) (start)--(chosen) (current)--(chosen);
  \draw[draw=geoPtwo!35,line width=.5pt,densely dashed]
    (previous)--(previousProjection)
    (mirror)--(chosenProjection)
    (chosenProjection)--(chosen);
  \draw[draw=geoPtwo!40,line width=.55pt,densely dashed]
    (current)--(mirror);
  \draw[chosenpath,patharrow] (previous)--(current);
  \draw[chosenpath,patharrow] (current)--(chosen);
  \foreach \P in {origin,start,current,previous,chosen}
    \fill[geoPtwo] (\P) circle[radius=1.25pt];
  \foreach \P in {previousProjection,chosenProjection}
    \fill[geoPtwo!55] (\P) circle[radius=.9pt];
  \draw[draw=geoPtwo!55,fill=white,line width=.55pt]
    (mirror) circle[radius=1.5pt];
  \node[below left,inner sep=3pt] at (origin) {$0$};
  \node[above left,inner sep=3pt] at (start) {$\widetilde R(0)$};
  \node[below left,inner sep=3pt] at (current) {$\widetilde R(r_i)$};
  \node[above left,inner sep=3pt] at (previous)
    {$\widetilde R(r_{i-1})$};
  \node[below right,inner sep=3pt] at (chosen)
    {$\widetilde R(r_{i+1})$};
  \node[above right,inner sep=3pt,text=black!55] at (mirror)
    {other position};
  \node[above right,xshift=2pt,yshift=4pt,inner sep=1pt,
    text=geoPtwo!70]
    at (previousProjection) {$\widetilde R_{\mathcal P}(r_{i-1})$};
  \node[above right,xshift=6pt,yshift=4pt,inner sep=1pt,
    text=geoPtwo!70]
    at (chosenProjection) {$\widetilde R_{\mathcal P}(r_{i+1})$};
\end{tikzpicture}
\par\vspace{1mm}
\refstepcounter{figure}\label{en:fig:next-vertex-plane-ru}
{\small\raggedright Fig.~\thefigure: Choosing the new vertex on the opposite
side of the plane. \textcolor{black}{Here \(\mathcal P\) is the plane
through \(0,\widetilde R(0),\widetilde R(r_i)\), and the subscript
\(\mathcal P\) denotes orthogonal projection onto it.}\par}
\end{minipage}
\par\smallskip

\medskip
\noindent\textit{Second case:} the points
\(0,\widetilde R(0),\widetilde R(r_i)\) lie on a line
(their affine rank is~\(1\)), but at least two of them are
distinct. If \(R(r_{i+1})\) also lies on the line through
\(0,R(0),R(r_i)\), then the conditions~\eqref{en:eq:next-vertex-three-distances-ru}
uniquely determine \(\widetilde R(r_{i+1})\).
\textcolor{black}{Next, consider the case where the next vertex lies off the reference line.}
\textcolor{black}{Rotation about the reference line preserves the distances
to all reference points. Thus a vertex off the line can be rotated
into the desired plane and placed on a chosen side of the line.
To choose such a position, we consider two subcases.}

\smallskip
\noindent
\begin{minipage}[t]{.46\textwidth}
\vspace{0pt}
\noindent\textit{a)} The vertex \(\widetilde R(r_{i-1})\) does not lie
on the line through \(0,\widetilde R(0),\widetilde R(r_i)\).
Then this line and the point \(\widetilde R(r_{i-1})\) determine a plane.
Choose \(\widetilde R(r_{i+1})\) in this plane, on the opposite
side of this line from \(\widetilde R(r_{i-1})\),
preserving the distances in~\eqref{en:eq:next-vertex-three-distances-ru}.
\end{minipage}\hfill
\begin{minipage}[t]{.51\textwidth}
\vspace{0pt}
\centering
\begin{tikzpicture}[
  x={(1.65cm,0cm)},y={(.55cm,.72cm)},
  line cap=round,line join=round,font=\scriptsize,
  chosenpath/.style={draw=geoPtwo,line width=1.05pt},
  patharrow/.style={decoration={markings,
    mark=at position .56 with {\arrow{Stealth[length=4pt,width=3pt]}}},
    postaction={decorate}}]
  \coordinate (planeA) at (-1.45,-1.05);
  \coordinate (planeB) at (1.3,-1.05);
  \coordinate (planeC) at (1.3,1.3);
  \coordinate (planeD) at (-1.45,1.3);
  \coordinate (origin) at (-1.25,0);
  \coordinate (start) at (-.6,0);
  \coordinate (current) at (.3,0);
  \coordinate (previous) at (-.4,.9);
  \coordinate (chosen) at (.95,-.85);
  \coordinate (mirror) at (.95,.85);
  \fill[geoPtwo!6] (planeA)--(planeB)--(planeC)--(planeD)--cycle;
  \draw[draw=geoPtwo!35,line width=.5pt]
    (planeA)--(planeB)--(planeC)--(planeD)--cycle;
  \fill[cyan!40!blue,fill opacity=.035]
    (origin)--(current)--(mirror)--cycle;
  \fill[cyan!40!blue,fill opacity=.05]
    (origin)--(current)--(chosen)--cycle;
  \draw[draw=cyan!40!blue,draw opacity=.23,line width=.4pt,densely dashed]
    (origin)--(mirror) (current)--(mirror);
  \draw[draw=cyan!40!blue,draw opacity=.3,line width=.45pt]
    (origin)--(chosen) (current)--(chosen);
  \draw[draw=geoPtwo!65,line width=.6pt]
    (-1.45,0)--(1.3,0);
  \draw[draw=geoPtwo!40,line width=.55pt,densely dashed]
    (current)--(mirror);
  \draw[chosenpath,patharrow] (previous)--(current);
  \draw[chosenpath,patharrow] (current)--(chosen);
  \foreach \P in {origin,start,current,previous,chosen}
    \fill[geoPtwo] (\P) circle[radius=1.25pt];
  \draw[draw=geoPtwo!55,fill=white,line width=.55pt]
    (mirror) circle[radius=1.5pt];
  \node[below left,inner sep=3pt] at (origin) {$0$};
  \node[above left,inner sep=3pt] at (start) {$\widetilde R(0)$};
  \node[below left,inner sep=3pt] at (current) {$\widetilde R(r_i)$};
  \node[above left,inner sep=3pt] at (previous)
    {$\widetilde R(r_{i-1})$};
  \node[below right,inner sep=3pt] at (chosen)
    {$\widetilde R(r_{i+1})$};
  \node[above right,inner sep=3pt,text=black!55] at (mirror)
    {other position};
\end{tikzpicture}
\par\vspace{1mm}
\refstepcounter{figure}\label{en:fig:next-vertex-line-ru}
{\small\raggedright Fig.~\thefigure: Choosing the new vertex on the opposite
side of the line.\par}
\end{minipage}
\par\smallskip

\smallskip
\noindent\textit{b)} The vertex \(\widetilde R(r_{i-1})\) lies
on the line through \(0,\widetilde R(0),\widetilde R(r_i)\).

Here there is no need to choose a side of the line:
take any position of \(\widetilde R(r_{i+1})\) that preserves
the distances in~\eqref{en:eq:next-vertex-three-distances-ru}.

\medskip
\noindent\textit{Third case:}
\(0=\widetilde R(0)=\widetilde R(r_i)\).
Then the assumptions of the lemma give \(R(0)=R(r_i)=0\), and the new vertex
is required only to satisfy \(\norm{\widetilde R(r_{i+1})}=\norm{R(r_{i+1})}\).
If \(R(r_{i+1})=0\), take \(\widetilde R(r_{i+1})=0\).
Otherwise, there are two subcases.

\smallskip
\noindent
\begin{minipage}[t]{.46\textwidth}
\vspace{0pt}
\noindent\textit{a)} The preceding vertex satisfies \(\widetilde R(r_{i-1})\ne0\).

On the line through \(0\) and \(\widetilde R(r_{i-1})\),
choose \(\widetilde R(r_{i+1})\) on the opposite side of zero
from \(\widetilde R(r_{i-1})\), at distance
\(\norm{R(r_{i+1})}\) from zero. This choice is shown in
Fig.~\ref{en:fig:next-vertex-point-ru}.
\end{minipage}\hfill
\begin{minipage}[t]{.51\textwidth}
\vspace{0pt}
\centering
\begin{tikzpicture}[
  x={(1.65cm,0cm)},y={(0cm,1.15cm)},
  line cap=round,line join=round,font=\scriptsize,
  chosenpath/.style={draw=geoPtwo,line width=1.05pt},
  patharrow/.style={decoration={markings,
    mark=at position .56 with {\arrow{Stealth[length=4pt,width=3pt]}}},
    postaction={decorate}}]
  \coordinate (previous) at (-1.2,.48);
  \coordinate (current) at (0,0);
  \coordinate (chosen) at (1,-.4);
  \draw[draw=cyan!40!blue,draw opacity=.09,line width=9pt]
    (-1.42,.568)--(1.22,-.488);
  \draw[draw=geoPtwo!35,line width=.5pt]
    (-1.42,.568)--(1.22,-.488);
  \draw[chosenpath,patharrow] (previous)--(current);
  \draw[chosenpath,patharrow] (current)--(chosen);
  \foreach \P in {previous,current,chosen}
    \fill[geoPtwo] (\P) circle[radius=1.25pt];
  \node[above left,inner sep=3pt] at (previous)
    {$\widetilde R(r_{i-1})$};
  \node[below,inner sep=4pt,align=center] at (current)
    {$0=\widetilde R(0)$\\[-1pt]$=\widetilde R(r_i)$};
  \node[below right,inner sep=3pt] at (chosen)
    {$\widetilde R(r_{i+1})$};
\end{tikzpicture}
\par\vspace{1mm}
\refstepcounter{figure}\label{en:fig:next-vertex-point-ru}
{\small\raggedright Fig.~\thefigure: Choosing the new vertex on the opposite
side of zero.\par}
\end{minipage}
\par\smallskip

\smallskip
\noindent\textit{b)} The preceding vertex satisfies \(\widetilde R(r_{i-1})=0\).

Choose \(\widetilde R(r_{i+1})\) in any direction from zero
at distance \(\norm{R(r_{i+1})}\).

\Needspace{7\baselineskip}
\medskip
Having defined the new vertex, we verify
identities~\eqref{en:eq:next-vertex-interior-distances-ru}:
the distances to zero and to \(\widetilde R(0)\) are preserved not only
at the vertices but also at every point of the adjacent segments.
\medskip

In any Hilbert space \(\textcolor{black}{H}\), for any
\(A,B,C\in\textcolor{black}{H}\) and \(s\in[0,1]\), the following identity holds
(by expanding the squared norm):
\[
\norm{(1-s)A+sB-C}^2
=(1-s)\norm{A-C}^2+s\norm{B-C}^2-s(1-s)\norm{B-A}^2.
\]
Fix \(j\in\{i-1,i\}\) and
\(\tau\in[r_j,r_{j+1}]\). Set
\(s=\dfrac{\tau-r_j}{r_{j+1}-r_j}\in[0,1]\).
Then, by linearity along the segment,
\(\widetilde R(\tau)=(1-s)\widetilde R(r_j)+s\widetilde R(r_{j+1})\)
and \(R(\tau)=(1-s)R(r_j)+sR(r_{j+1})\).

Substitute \(A=\widetilde R(r_j)\),
\(B=\widetilde R(r_{j+1})\), and \(C=0\).
For both values of \(j\), the endpoint norms and the segment length agree
with the original ones: for \(j=i-1\), this is an assumption of the lemma; for
\(j=i\), it follows from that assumption and~\eqref{en:eq:next-vertex-three-distances-ru}.
We therefore obtain:
\begingroup
\renewcommand{\norm}[1]{\bigl\lVert #1\bigr\rVert}
\[
\begin{aligned}
\norm{\widetilde R(\tau)}^2
&=(1-s)\norm{\widetilde R(r_j)}^2+s\norm{\widetilde R(r_{j+1})}^2
-s(1-s)\norm{\widetilde R(r_{j+1})-\widetilde R(r_j)}^2\\
&=(1-s)\norm{R(r_j)}^2+s\norm{R(r_{j+1})}^2
-s(1-s)\norm{R(r_{j+1})-R(r_j)}^2\\
&=\norm{R(\tau)}^2.
\end{aligned}
\]

Now substitute \(A=\widetilde R(r_j)\),
\(B=\widetilde R(r_{j+1})\), and \(C=\widetilde R(0)\)\textcolor{black}{.
For} the original function, take \(A=R(r_j)\),
\(B=R(r_{j+1})\), and \(C=R(0)\).
Similarly, we obtain:
\[
\begin{aligned}
&\norm{\widetilde R(\tau)-\widetilde R(0)}^2\\
&=(1-s)\norm{\widetilde R(r_j)-\widetilde R(0)}^2
+s\norm{\widetilde R(r_{j+1})-\widetilde R(0)}^2
-s(1-s)\norm{\widetilde R(r_{j+1})-\widetilde R(r_j)}^2\\
&=(1-s)\norm{R(r_j)-R(0)}^2+s\norm{R(r_{j+1})-R(0)}^2
-s(1-s)\norm{R(r_{j+1})-R(r_j)}^2\\
&=\norm{R(\tau)-R(0)}^2.
\end{aligned}
\]
\endgroup
Taking square roots gives both
identities~\eqref{en:eq:next-vertex-interior-distances-ru}
for every \(\tau\in[r_{i-1},r_{i+1}]\).

\Needspace{17\baselineskip}
\medskip
We now verify bound~\eqref{en:eq:next-vertex-jump-bound-ru}
for the derivative jump.
As in the choice of the new vertex, we consider three cases depending
on the configuration of \(0,\widetilde R(0),\widetilde R(r_i)\).

\medskip
\noindent\textit{First case.}
Suppose that the points \(0,\widetilde R(0),\widetilde R(r_i)\)
determine a plane.
By the assumptions of the lemma, \(\norm{\widetilde R(0)}=\norm{R(0)}\),
\(\norm{\widetilde R(r_i)}=\norm{R(r_i)}\), and
\(\norm{\widetilde R(r_i)-\widetilde R(0)}=\norm{R(r_i)-R(0)}\).
To visualize this, imagine superimposing the plane through
\(0,R(0),R(r_i)\) on the plane through
\(0,\widetilde R(0),\widetilde R(r_i)\), aligning the corresponding
reference points. Denote the resulting plane by \(\mathcal P\).
Denote the orthogonal projection of \(R(\tau)\) onto \(\mathcal P\) by
\(R_{\mathcal P}(\tau)\)\textcolor{black}{. Use analogous} notation for the projections of
\(\widetilde R(\tau)\) and the derivatives.
(The plane \(\mathcal P\) depends on \(i\) and changes at each step.)

\Needspace{6\baselineskip}
\textcolor{black}{First compare the projections of the jumps onto the plane
\(\mathcal P\).}
The standard inner product identity,
identities~\eqref{en:eq:next-vertex-interior-distances-ru},
and the preservation of the length of the segment \([R(r_{i-1}),R(r_i)]\) assumed in the lemma give:
\[
\begin{array}{@{}c@{\;\Longrightarrow\;}l@{}}
2\langle A,B\rangle=\norm{A}^2+\norm{B}^2-\norm{A-B}^2
&\begin{aligned}
\langle\widetilde R(r_{i-1}),\widetilde R(0)\rangle
&=\langle R(r_{i-1}),R(0)\rangle,\\
\langle\widetilde R(r_{i-1}),\widetilde R(r_i)\rangle
&=\langle R(r_{i-1}),R(r_i)\rangle.
\end{aligned}
\end{array}
\]
\textcolor{black}{Once the planes are identified, the reference vectors \(R(0)\) and \(R(r_i)\)
span the plane \(\mathcal P\). Inner products with these vectors
uniquely determine the projection onto this plane. Hence
\(R_{\mathcal P}(r_{i-1})=\widetilde R_{\mathcal P}(r_{i-1})\).
For the next pair of vertices, preservation of the three distances
in~\eqref{en:eq:next-vertex-three-distances-ru} similarly gives
\(R_{\mathcal P}(r_{i+1})=\widetilde R_{\mathcal P}(r_{i+1})\).
The current vertices \(R(r_i)\) and \(\widetilde R(r_i)\) have been identified,
so their projections also agree.}

\Needspace{7\baselineskip}
\medskip
We now write the derivatives on the adjacent segments explicitly.
Since \(R\) is linear on each of them, we obtain:
\begin{equation}\label{en:eq:adjacent-link-derivatives-ru}
R'(r_i-0)=\frac{R(r_i)-R(r_{i-1})}{r_i-r_{i-1}},
\qquad
R'(r_i+0)=\frac{R(r_{i+1})-R(r_i)}{r_{i+1}-r_i}.
\end{equation}
Analogous formulas hold for \(\widetilde R\).
We have \(R_{\mathcal P}(r_{i-1})=\widetilde R_{\mathcal P}(r_{i-1})\)
and \(R_{\mathcal P}(r_i)=\widetilde R_{\mathcal P}(r_i)\).
Subtracting the projections and dividing by \(r_i-r_{i-1}\) therefore gives
\(R'_{\mathcal P}(r_i-0)=\widetilde R'_{\mathcal P}(r_i-0)\).
\textcolor{black}{Similarly, using the equality of the projections of the current and next
vertices and dividing their difference by \(r_{i+1}-r_i\), we obtain
\(R'_{\mathcal P}(r_i+0)=\widetilde R'_{\mathcal P}(r_i+0)\).}
Consequently, the projections of the jumps agree:
\[
\widetilde R'_{\mathcal P}(r_i+0)-\widetilde R'_{\mathcal P}(r_i-0)
=R'_{\mathcal P}(r_i+0)-R'_{\mathcal P}(r_i-0).
\]
Thus the projections of the jumps onto \(\mathcal P\) have equal norms.

\textcolor{black}{Now compare the perpendicular components
of the jumps. The derivative jump is the difference between the velocities on adjacent
segments, so we first establish equality of the lengths of the perpendicular
components of the corresponding velocities for \(R\) and \(\widetilde R\).
Denote perpendicular components by the subscript \(\perp\).}
The length of the segment \([R(r_{i-1}),R(r_i)]\) is preserved by the assumptions
of the lemma, and that of \([R(r_i),R(r_{i+1})]\) by~\eqref{en:eq:next-vertex-three-distances-ru}.
In~\eqref{en:eq:adjacent-link-derivatives-ru} and the analogous
formulas for \(\widetilde R\), the corresponding denominators agree:
\(r_i-r_{i-1}\) for the left derivatives and
\(r_{i+1}-r_i\) for the right derivatives. \textcolor{black}{Hence
\(\norm{\widetilde R'(r_i-0)}=\norm{R'(r_i-0)}\) and
\(\norm{\widetilde R'(r_i+0)}=\norm{R'(r_i+0)}\).}
\textcolor{black}{Corresponding derivatives have the same norms and the same projections onto
\(\mathcal P\). By the Pythagorean theorem, their
perpendicular components therefore also have equal lengths:
\[
\begin{aligned}
\norm{\widetilde R'_\perp(r_i\pm0)}^2
&=\norm{\widetilde R'(r_i\pm0)}^2
-\norm{\widetilde R'_{\mathcal P}(r_i\pm0)}^2\\
&=\norm{R'(r_i\pm0)}^2-\norm{R'_{\mathcal P}(r_i\pm0)}^2
=\norm{R'_\perp(r_i\pm0)}^2.
\end{aligned}
\]
We now compare the differences between these components, using the chosen
positions of the adjacent vertices relative to the plane \(\mathcal P\).}

\medskip

\textcolor{black}{The vertex \(\widetilde R(r_i)\) lies in \(\mathcal P\).
If both perpendicular components are nonzero, the adjacent vertices
\(\widetilde R(r_{i-1})\) and \(\widetilde R(r_{i+1})\) lie outside
the plane and, by construction, on opposite sides of it.
When moving toward and away from \(\widetilde R(r_i)\), the perpendicular
velocity components therefore point in the same direction.
The length of their difference equals the absolute difference of their lengths\textcolor{black}{. This} equality also holds
when a component is zero. Using the established equality
of the lengths of corresponding components for \(R\) and \(\widetilde R\), we obtain:}
\[
\begin{aligned}
\norm{\widetilde R'_{\perp}(r_i+0)-\widetilde R'_{\perp}(r_i-0)}
&=\Bigl|\norm{R'_{\perp}(r_i+0)}-\norm{R'_{\perp}(r_i-0)}\Bigr|\\
&\le\norm{R'_{\perp}(r_i+0)-R'_{\perp}(r_i-0)}.
\end{aligned}
\]
The last inequality is an application of the triangle inequality.

\begin{figure}[H]
\centering
\begin{tikzpicture}[x=1cm,y=1cm,font=\small,
  line cap=round,line join=round,
  component/.style={line width=1.1pt,-{Stealth[length=5pt,width=3.5pt]}}]
  \foreach \shift/\height/\heading in {0/3.2/R,7.6/2.0/\widetilde R}{
    \begin{scope}[xshift=\shift cm]
      \node at (2,4.4) {$\heading$};
      \filldraw[fill=geoPtwo!7,draw=black!35,line width=.5pt]
        (-.65,-.65)--(3.65,-.65)--(4.65,.95)--(.35,.95)--cycle;
      \node[text=black!65] at (4.05,.62) {$\mathcal P$};
      \coordinate (O) at (0,0);
      \coordinate (B) at (3,.3);
      \coordinate (A) at (3,{.3+\height});
      \fill[geoPtwo!9] (O)--(B)--(A)--cycle;
      \draw[component,draw=geoPtwo] (O)--(B)
        node[midway,sloped,below=5pt] {projection onto $\mathcal P$};
      \draw[component,draw=black!60] (B)--(A)
        node[midway,right=6pt,yshift=4pt,align=left,font=\footnotesize]
        {perpendicular\\component};
      \draw[component,draw=geoPtwo,line width=1.5pt] (O)--(A)
        node[midway,sloped,above=5pt] {jump};
      \draw[black!55,line width=.6pt]
        (2.72,.272)--(2.72,.552)--(3,.58);
      \fill[geoPtwo] (O) circle[radius=1.4pt];
    \end{scope}
  }
  \node[align=center,font=\small] at (5.8,-1.2)
    {Projections agree\textcolor{black}{. The length} of the perpendicular component does not increase.};
\end{tikzpicture}
\caption{\textcolor{black}{Schematic orthogonal decomposition of the derivative jumps
for \(R\) and \(\widetilde R\). The projections agree, and the length of the perpendicular
component does not increase. By the Pythagorean theorem, the length of the jump
does not increase either.}}
\label{en:fig:normal-velocity-directions-ru}
\end{figure}

\textcolor{black}{It remains to bound the derivative jump using the Pythagorean
theorem. Its projection onto \(\mathcal P\) is preserved, and the norm of its
perpendicular component has not increased. Hence
\[
\begin{aligned}
&\norm{\widetilde R'(r_i+0)-\widetilde R'(r_i-0)}^2\\
&\quad=\norm{\widetilde R'_{\mathcal P}(r_i+0)-\widetilde R'_{\mathcal P}(r_i-0)}^2
+\norm{\widetilde R'_\perp(r_i+0)-\widetilde R'_\perp(r_i-0)}^2\\
&\quad\le\norm{R'_{\mathcal P}(r_i+0)-R'_{\mathcal P}(r_i-0)}^2
+\norm{R'_\perp(r_i+0)-R'_\perp(r_i-0)}^2\\
&\quad=\norm{R'(r_i+0)-R'(r_i-0)}^2.
\end{aligned}
\]
Taking square roots gives
bound~\eqref{en:eq:next-vertex-jump-bound-ru} in the plane case.}

\medskip
\noindent\textit{Second case.}
If the points \(0,\widetilde R(0),\widetilde R(r_i)\) lie on one
line and do not all coincide, repeat the argument with projections
onto this line and components perpendicular to it.
\textcolor{black}{As in the first case, the preserved distances and segment lengths
ensure equality of the projections of corresponding velocities and of the lengths of their
perpendicular components. If both components are nonzero, the adjacent
vertices lie, by construction, in a common plane with the reference line and on opposite
sides of it, so these velocity components point in the same direction.
Applying the same bound to their difference and using the Pythagorean theorem gives~
\eqref{en:eq:next-vertex-jump-bound-ru}\textcolor{black}{. The} argument remains valid
when a component is zero.}

\medskip
\noindent\textit{Third case.}
If \(0=\widetilde R(0)=\widetilde R(r_i)\), then, when
\(\widetilde R(r_{i-1})\ne0\), the next vertex is chosen by continuing
the motion from \(\widetilde R(r_{i-1})\) toward zero.
Thus the nonzero one-sided derivatives point in the same direction,
\textcolor{black}{and their norms equal the original ones because the segment lengths
and threshold intervals are preserved. Hence
\[
\begin{aligned}
\norm{\widetilde R'(r_i+0)-\widetilde R'(r_i-0)}
&=\Bigl|\norm{R'(r_i+0)}-\norm{R'(r_i-0)}\Bigr|\\
&\le\norm{R'(r_i+0)-R'(r_i-0)}.
\end{aligned}
\]
The last inequality holds because the absolute difference of the lengths of two vectors
does not exceed the distance between them.
If one derivative is zero, the same chain remains valid.}

\medskip
Thus inequality~\eqref{en:eq:next-vertex-jump-bound-ru}
holds in every case, completing
the proof of the lemma.
\end{proof}
\medskip

\Needspace{12\baselineskip}
Applying the lemma successively for \(i=1,\ldots,m-1\),
we place all remaining vertices. At each step, the preceding construction
ensures that the assumptions of the lemma hold. Join adjacent vertices
by segments, traversing them linearly in \(\tau\) over the same intervals
\([r_i,r_{i+1}]\), and set \(\widetilde R(\tau)=0\) for
\(\tau\ge r_m\). This gives the required function \(\widetilde R\):
the two required distances are preserved by \textcolor{black}{identities}~\eqref{en:eq:next-vertex-interior-distances-ru}
of the lemma, and the interior derivative jumps do not increase by
\textcolor{black}{bound}~\eqref{en:eq:next-vertex-jump-bound-ru}.
The jump at \(r_m\) does not increase either: the right derivative
is zero, and the length of the last segment is preserved. If \(m=1\),
the single segment was already placed in the initial construction\textcolor{black}{.
The identity} above preserves both distances along the entire segment.

\medskip
\noindent\textcolor{black}{\textit{4. Recovering a random vector from the polygonal path.}}
We now construct a random vector \(Z\) taking values in \(\R^3\)
such that \(X\) and \(Z\) have identically distributed norms
and the following properties hold:
\begin{equation}
\label{en:eq:constructed-z-residual-ru}
\E[\norm{Z}]<\infty,\qquad
\E[Z-C_\tau(Z)]=\widetilde R(\tau)\quad\text{for every }\tau>0.
\end{equation}
\Needspace{8\baselineskip}
Choose unit vectors \(e_i\in\R^3\), \(i=1,\ldots,m\),
in the directions of the corresponding derivative jumps:
\[
e_i=\frac{\widetilde R'(r_i+0)-\widetilde R'(r_i-0)}
{\norm{\widetilde R'(r_i+0)-\widetilde R'(r_i-0)}}.
\]
If the jump is zero, set \(e_i=(1,0,0)\).

\medskip
Define \(Z\) as follows (include the first two cases
for each \(1\le i\le m\)):
\[
Z=\begin{cases}
r_i e_i,&\text{with probability }
\dfrac{\mathbb P(\norm{X}=r_i)+
\norm{\widetilde R'(r_i+0)-\widetilde R'(r_i-0)}}{2},\\[3mm]
-r_i e_i,&\text{with probability }
\dfrac{\mathbb P(\norm{X}=r_i)-
\norm{\widetilde R'(r_i+0)-\widetilde R'(r_i-0)}}{2},\\[3mm]
0,&\text{with probability }\mathbb P(X=0).
\end{cases}
\]

\noindent\textcolor{black}{\textit{Comment:}
essentially, we choose \(Z\) so that the norm takes the value \(r_i\) with the same
probability as for \(X\), while the derivative jump of the mean residual agrees
with that of the constructed polygonal path. To do this, split the total probability
\(\mathbb P(\norm{X}=r_i)\) between the values
\(r_i e_i\) and \(-r_i e_i\): the probability of the former must exceed that of the latter
by exactly \(\norm{\widetilde R'(r_i+0)-\widetilde R'(r_i-0)}\).}

\textcolor{black}{We check that \(Z\) is well defined.
By the \textcolor{black}{jump bound in item 3} and property (P4) of \(R\),
we have \(\norm{\widetilde R'(r_i+0)-\widetilde R'(r_i-0)}
\le\mathbb P(\norm{X}=r_i)\), so both probabilities for each
radius \(r_i\) are nonnegative. Their sum is
\(\mathbb P(\norm{X}=r_i)\)\textcolor{black}{. Together} with the probability of the zero value,
these sum to one. Thus the distribution of \(Z\) is well defined,
and \(X\) and \(Z\) have identically distributed norms.
Since \(\norm{Z}\le r_m\), the vector \(Z\) is integrable.}

\medskip
\textcolor{black}{It remains to verify the identity for the mean residual
in~\eqref{en:eq:constructed-z-residual-ru}. Substituting the values of \(Z\)
and their probabilities gives:}
\[
\begin{aligned}
\E[Z-C_\tau(Z)]
&=\sum_{i=1}^{m}e_i
\norm{\widetilde R'(r_i+0)-\widetilde R'(r_i-0)}
\max\{r_i-\tau,0\}\underset{\mathrm{I}}{=}{}\\[2mm]
&\underset{\mathrm{I}}{=}\sum_{i=1}^{m}
\bigl(\widetilde R'(r_i+0)-\widetilde R'(r_i-0)\bigr)
\max\{r_i-\tau,0\}\underset{\mathrm{II}}{=}\widetilde R(\tau).
\end{aligned}
\]
\noindent Equality I follows from the choice of \(e_i\): multiplying this
direction by the norm of the jump gives the jump itself, including when it is zero.

\noindent Equality II holds because \textcolor{black}{the derivative of
\(\max\{r_i-\tau,0\}\) has a jump at \(r_i\) of size \(1\), so
the jumps in the derivative of the sum agree with the jumps in \(\widetilde R'\).
The derivative} of the difference
of these two continuous piecewise linear functions is constant between
adjacent \(r_i\) and has no jumps. The difference is therefore affine\textcolor{black}{.
It} vanishes for every \(\tau\ge r_m\), and hence everywhere.

As \(\tau\downarrow0\), we have
\(\norm{\E[C_\tau(Z)]}\le\tau\), so
\eqref{en:eq:constructed-z-residual-ru} and the continuity of \(\widetilde R\)
imply \(\E[Z]=\widetilde R(0)\).

Thus the constructed \(Z\) satisfies all the conditions
shown to be sufficient at the beginning of item 2.

We have therefore proved the theorem when \(\norm{X}\)
takes finitely many values. We now extend the result
to an arbitrary distribution of the norm of a random vector in the original
Hilbert space \(\textcolor{black}{H}\).

\textcolor{black}{\textit{5. An arbitrary distribution of the norm.}}
\textcolor{black}{We proceed as follows: first approximate \(X\)
by vectors \(X_n\) to which the construction in items 1--4 applies.
From the corresponding three-dimensional realizations \(Z_n\), extract a limiting
distribution and choose a vector \(Z\) with this distribution. Then verify that
the resulting vector \(Z\) satisfies the conditions of the theorem.}

\medskip
\noindent\textcolor{black}{\textit{5.1. Approximating the original vector.}}
For \(n\ge1\), cap the vector's length at \(n\), then round
it down to the nearest multiple of \(\frac1n\), preserving its direction.
In other words, set
\[
a_n(r):=\frac{\lfloor n\min\{r,n\}\rfloor}{n},\qquad r\ge0,
\qquad
X_n:=\begin{cases}
\displaystyle a_n(\norm{X})\frac{X}{\norm{X}},&X\ne0,\\[2mm]
0,&X=0.
\end{cases}
\]
The norm \(\norm{X_n}\) takes finitely many values:
\(0,\frac1n,\frac2n,\ldots,n\), with
\(0\le\norm{X_n}\le\norm{X}\). Moreover, \(X_n\) is integrable.
\textcolor{black}{We show that \(X_n\to X\) in \(L^1\).}
\textcolor{black}{The error comes from rounding the length down and truncating the part above \(n\):}
\[
\norm{X-X_n}=\norm{X}-a_n(\norm{X})
\le\frac1n+\max\{\norm{X}-n,0\}.
\]
Here \(\max\{\norm{X}-n,0\}\to0\) almost surely and
\(0\le\max\{\norm{X}-n,0\}\le\norm{X}\).
\textcolor{black}{By the dominated convergence theorem, since
\(\E[\norm{X}]<\infty\), we obtain:}
\begin{equation}\label{en:eq:radial-l1-error-ru}
\varepsilon_n:=\E[\norm{X-X_n}]
\le\frac1n+\E[\max\{\norm{X}-n,0\}]\longrightarrow0.
\end{equation}
Since, for each \(n\), the norm \(\norm{X_n}\) takes finitely
many values, the established case of the theorem applies to \(X_n\).
This gives random vectors \(Z_n\) taking values in \(\R^3\)
such that \(\norm{Z_n}\) and \(\norm{X_n}\) are identically distributed.
\textcolor{black}{Moreover, by the finite case already proved,} for every \(\tau>0\) we obtain:
\begin{equation}\label{en:eq:finite-radial-approximation-ru}
\begin{aligned}
\norm{\E[Z_n-C_\tau(Z_n)]}
&=\norm{\E[X_n-C_\tau(X_n)]},\\
\norm{\E[C_\tau(Z_n)]}
&=\norm{\E[C_\tau(X_n)]}.
\end{aligned}
\end{equation}

\medskip
\noindent\textcolor{black}{\textit{5.2. Choosing a limiting distribution.}}
The vectors \(Z_n\) themselves may be defined on different probability
spaces, so we consider their distributions.
For every \(M>0\), we have
\[
\mathbb P(\norm{Z_n}>M)
=\mathbb P(a_n(\norm{X})>M)
\le\mathbb P(\norm{X}>M)\longrightarrow0
\quad\text{as }M\to\infty.
\]
The right-hand side does not depend on \(n\). Thus, for every
\(\varepsilon>0\), we can choose a single closed ball in \(\R^3\)
that contains each \(Z_n\) with probability at least
\(1-\varepsilon\). Such balls are compact. Hence, by Prokhorov's
theorem, there exists a subsequence \(Z_{n_k}\)
\textcolor{black}{whose distributions converge weakly to a
probability distribution on \(\R^3\). Take a random vector \(Z\)
with this distribution. Then \(\E[f(Z_{n_k})]\longrightarrow\E[f(Z)]\)
for every bounded continuous function \(f\colon\R^3\to\R\).}

\textcolor{black}{We check that \(\norm{Z}\) and \(\norm{X}\) are identically
distributed. It suffices to compare the expectations of all bounded
continuous functions of these quantities. For any such function
\(g\colon[0,\infty)\to\R\), we have:}
\[
\E[g(\norm{Z})]
=\lim_{k\to\infty}\E[g(\norm{Z_{n_k}})]
=\lim_{k\to\infty}\E[g(a_{n_k}(\norm{X}))]
=\E[g(\norm{X})].
\]
The first equality follows from weak convergence: the function
\(z\mapsto g(\norm{z})\) is bounded and continuous on \(\R^3\).
The second follows from the construction of \(Z_n\): its norm has the same
distribution as \(\norm{X_n}=a_n(\norm{X})\).
For the last equality, note that
\(a_{n_k}(\norm{X})\to\norm{X}\) almost surely.
The continuity of \(g\) gives
\(g(a_{n_k}(\norm{X}))\to g(\norm{X})\), and its boundedness
allows us to pass to the limit under the expectation.

We conclude that \(\norm{Z}\) and \(\norm{X}\) are identically distributed
(and \(Z\) is integrable, since
\(\E[\norm{Z}]=\E[\norm{X}]<\infty\)).

\medskip
\noindent\textcolor{black}{\textit{5.3. Convergence of expectations.}}
To pass to the limit in~\eqref{en:eq:finite-radial-approximation-ru},
we need convergence of \(\E[Z_{n_k}]\) to \(\E[Z]\).
\textcolor{black}{First clip the vectors at an auxiliary
threshold \(\tau'>0\). Use weak convergence for the clipped vectors,
and bound the mean length of the residual in terms of \(\norm{X}\).
The operator \(C_{\tau'}\colon\R^3\to\R^3\) is continuous and bounded.
Applying weak convergence to each of its coordinates gives
\(\E[C_{\tau'}(Z_{n_k})]\longrightarrow\E[C_{\tau'}(Z)]\)
as \(k\to\infty\).}
Now bound the expected lengths of the parts of the vectors removed by this clipping:
\[
\begin{aligned}
\E[\norm{Z_n-C_{\tau'}(Z_n)}]
&=\E[\max\{a_n(\norm{X})-\tau',0\}]
\le\E[\max\{\norm{X}-\tau',0\}],\\
\E[\norm{Z-C_{\tau'}(Z)}]
&=\E[\max\{\norm{X}-\tau',0\}].
\end{aligned}
\]
\textcolor{black}{The triangle inequality gives:}
\[
\norm{\E[Z_{n_k}]-\E[Z]}
\le\norm{\E[C_{\tau'}(Z_{n_k})]-\E[C_{\tau'}(Z)]}
+2\E[\max\{\norm{X}-\tau',0\}].
\]
\textcolor{black}{First choose \(\tau'\) so that the second term
is arbitrarily small. This is possible by the dominated convergence
theorem, since \(\E[\norm{X}]<\infty\).
For this fixed \(\tau'\), the first term tends to zero
as \(k\to\infty\). Consequently,
\(\E[Z_{n_k}]\longrightarrow\E[Z]\).}

\medskip
\noindent\textcolor{black}{\textit{5.4. Preservation of bias and energy.}}
\textcolor{black}{It remains to verify preservation of bias and energy.
Fix an arbitrary threshold \(\tau>0\) and pass to the limit
in identities~\eqref{en:eq:finite-radial-approximation-ru}.}
\textcolor{black}{By construction, \(X_n\) and \(X\) lie on the same ray.
Along this ray, clipping and the residual are given by the functions
\(r\mapsto\min\{r,\tau\}\) and \(r\mapsto\max\{r-\tau,0\}\), each of which
has Lipschitz constant \(1\).}
Therefore
\[
\begin{aligned}
\norm{C_\tau(X_n)-C_\tau(X)}&\le\norm{X_n-X},\\
\norm{(X_n-C_\tau(X_n))-(X-C_\tau(X))}&\le\norm{X_n-X}.
\end{aligned}
\]
\textcolor{black}{The norm of the expectation of a difference does not exceed
the expected norm of that difference. Hence, by bound~\eqref{en:eq:radial-l1-error-ru},
where \(\varepsilon_n\to0\), we obtain}
\[
\E[C_\tau(X_n)]\longrightarrow\E[C_\tau(X)],
\qquad
\E[X_n-C_\tau(X_n)]\longrightarrow\E[X-C_\tau(X)].
\]
\textcolor{black}{For \(Z_{n_k}\), continuity and boundedness of \(C_\tau\)}
give \(\E[C_\tau(Z_{n_k})]\to\E[C_\tau(Z)]\).
Together with the established convergence \(\E[Z_{n_k}]\to\E[Z]\),
this allows us to pass to the limit on both sides of
identities~\eqref{en:eq:finite-radial-approximation-ru}:
\[
\begin{aligned}
\norm{\E[Z-C_\tau(Z)]}&=\norm{\E[X-C_\tau(X)]},\\
\norm{\E[C_\tau(Z)]}&=\norm{\E[C_\tau(X)]}.
\end{aligned}
\]
\textcolor{black}{The first equality preserves the bias. For the energy,
use the variance identity from item 2: the second moments of
the clipped vectors agree because \(\norm{Z}\) and \(\norm{X}\)
are identically distributed, and the norms of their means agree by the second
equality. The energies are therefore equal as well.}

The subsequence \(Z_{n_k}\) and its limiting distribution were chosen
before \(\tau\) was fixed. Thus the same vector \(Z\)
preserves bias and energy for every \(\tau>0\).
Together with the established equality
\(\mathbb P(\norm{Z}\le t)=\mathbb P(\norm{X}\le t)\)
for all \(t\ge0\), this completes the proof of the theorem.
\end{proof}

\noindent\textit{Comment and intuition:} The main point of this theorem is that the
same three-dimensional vector \(Z\) reproduces the bias and energy
of the original vector \(X\) simultaneously at all clipping thresholds.
There is no need to choose a new distribution for each new threshold.
The distribution of the vector's length is also preserved.
Thus, to study the entire dependence of bias and energy
on the clipping threshold, it suffices to consider random
vectors in \(\R^3\), even if the original space is infinite-dimensional.

These results give a clear description of how clipping
affects energy and bias. Theorem~\ref{en:thm:stochastic-ru}
gives a sharp bound on their weighted sum in terms of the clipping threshold and a moment
of the original vector. Theorem~\ref{en:thm:attainable-set-ru} describes
the possible relationships between normalized energy and bias
as a function of the parameter \(p\).
Finally, Theorem~\ref{en:thm:all-thresholds-ru} shows that a single random
vector in \(\R^3\) suffices to reproduce their dependence on the threshold.
\textcolor{black}{We now examine how the resulting joint
bounds apply to two practical problems.}
\section{Applications of the joint envelope}
\label{en:sec:applications-ru}

In this section we consider two examples showing how our estimates apply
to different problems. In the first, we bound an insurance company's
costs for a prescribed per-claim retention threshold, accounting jointly for expected
payments and risk. In the second, we consider a portfolio management
algorithm in which clipping limits changes in the risky-asset allocation,
and bound its shortfall relative to a constant-allocation strategy. The
problems are different, but in both cases our results help us bound
the effects of clipping jointly.

(In presenting these examples, the author has tried to retain their
practical meaning and the necessary mathematical content without
overloading the exposition with details that would obscure the precise
role of the joint estimate.)

\subsection{Reinsurance with a random number of losses}

Imagine the following situation: we are an insurance company that insures
large firms and wants to protect itself against excessively large
payments. To do so, we enter into the following agreement with another company,
the reinsurer: for each claim, we pay the loss up to a preselected
level \(\tau\), and the reinsurer pays the amount above that level
(this arrangement is indeed used in reinsurance).

However, this protection against risk carries an additional cost.
This raises a question: how does the chosen level \(\tau\) affect our
expected costs and the variability of our total payments? To study this
question, let us translate the situation into mathematical terms.

For simplicity, fix an observation period of one year.
Let \(N\) be the random number of claims during this period
(that is, the number of events giving rise to insurance payments, rather
than the number of insurance contracts).
Write \(X_1,X_2,\ldots\ge0\) for the corresponding loss amounts.
To simplify the model, assume that the loss amounts are independent and
identically distributed, and that the entire sequence is independent of \(N\).
\textcolor{black}{Also assume that \(\E[X_1^2]<\infty\) and \(\E[N^2]<\infty\).}
Use the standard notation \(\mu_N=\E[N]\) and
\(\sigma_N^2=\operatorname{Var}(N)\).

Also fix a retention threshold \(\tau>0\): our company pays at most \(\tau\) per claim,
so that \(C_\tau(X_i)=\min\{X_i,\tau\}\) (the same radial clipping
operator), while the reinsurer pays the residual amount
\(X_i-C_\tau(X_i)\). \textcolor{black}{Take the reinsurance premium to
be the reinsurer's expected total payment with a loading \(\alpha\ge0\).
The expected total payment is \(\mu_N\E[X_1-C_\tau(X_1)]\).}
Define our company's total payment over the period and the reinsurance
premium by
\[
\mathrm{Claims}_\tau:=\sum_{i=1}^{N}C_\tau(X_i),
\qquad
\mathrm{Cost}_\tau:=(1+\alpha)\mu_N\E[X_1-C_\tau(X_1)].
\]
Adding a risk penalty with a prescribed weight \(\gamma\ge0\), we obtain
\[
\mathrm{RiskAdjustedCost}_\tau=\E[\mathrm{Claims}_\tau]+\mathrm{Cost}_\tau
+\gamma\operatorname{Var}(\mathrm{Claims}_\tau).
\]
Thus \(\mathrm{RiskAdjustedCost}_\tau\) accounts for our expected costs,
namely our own claim payments and the reinsurance premium, together with
a penalty for the variability of our own payments. The larger
\(\gamma\), the more important it is to reduce this variability.

We can now ask a natural question: how can we obtain a reasonably simple
bound on \(\mathrm{RiskAdjustedCost}_\tau\)? The following result provides one.

\begin{corollaryru}
\label{en:cor:app-reinsurance-ru}
Let \(N\ge0\) be an integer-valued random variable with
\(\E[N^2]<\infty\). Let \((X_i)_{i\ge1}\) be a family of independent,
identically distributed nonnegative random variables, independent of
\(N\), with \(\E[X_1^2]<\infty\). Fix parameters
\(\tau>0\), \(\alpha,\gamma\ge0\). Then the risk-adjusted cost satisfies
\begin{equation}\label{en:eq:app-reinsurance-ru}
\mathrm{RiskAdjustedCost}_\tau\le \mu_N\E[X_1]
+\frac{\mu_N}{\tau}K_2(\alpha,\gamma\tau)\E[X_1^2]
+\gamma \sigma_N^2\bigl(\E[X_1]\bigr)^2.
\end{equation}
\end{corollaryru}

\begin{proof}
\textcolor{black}{The loss amounts are independent and identically
distributed, and their entire sequence is independent of \(N\).
Consequently, conditional on the number of claims, the payments are
independent.} The laws of total expectation and total variance give
the following expressions
\[
\begin{aligned}[t]
\E[\mathrm{Claims}_\tau]
&=\E\!\left[\E[\mathrm{Claims}_\tau\mid N]\right]=\\
&=\E[N\E[C_\tau(X_i)]]=\\
&=\mu_N\E[C_\tau(X_i)].
\end{aligned}
\hspace{4em}
\begin{aligned}[t]
&\operatorname{Var}(\mathrm{Claims}_\tau)=\\
&=\E[N\operatorname{Var}(C_\tau(X_i))]
 +\operatorname{Var}(N\E[C_\tau(X_i)])=\\
&=\mu_N\operatorname{Var}(C_\tau(X_i))
 +\sigma_N^2(\E[C_\tau(X_i)])^2.
\end{aligned}
\]

Next, substitute these expressions into the formula for
\(\mathrm{RiskAdjustedCost}_\tau\):
\[
\begin{aligned}
\mathrm{RiskAdjustedCost}_\tau
&=\underbrace{\mu_N\E[X_i]}_{\mathrm{(I)}}
+\underbrace{\mu_N\Bigl(\alpha\E[X_i-C_\tau(X_i)]
+\gamma\operatorname{Var}(C_\tau(X_i))\Bigr)}_{\mathrm{(II)}}
+\underbrace{\gamma\sigma_N^2\bigl(\E[C_\tau(X_i)]\bigr)^2}_{\mathrm{(III)}}.
\end{aligned}
\]
Bound part \(\mathrm{(II)}\) using
Theorem~\ref{en:thm:stochastic-ru} with \(p=2\)
(since \(0\le C_\tau(X_i)\le X_i\), we have
\(\E[X_i-C_\tau(X_i)]=|\E[X_i]-\E[C_\tau(X_i)]|\)):
\[
\alpha\E[X_i-C_\tau(X_i)]+\gamma\operatorname{Var}(C_\tau(X_i))
\le\frac{K_2(\alpha,\gamma\tau)}{\tau}\E[X_i^2].
\]
Multiplying by \(\mu_N\ge0\) gives a bound for part
\(\mathrm{(II)}\).

Bound part \(\mathrm{(III)}\) in the standard way
(\(0\le\E[C_\tau(X_i)]\le\E[X_i]\) and \(\gamma\sigma_N^2\ge0\)):
\[
\gamma\sigma_N^2\bigl(\E[C_\tau(X_i)]\bigr)^2
\le\gamma\sigma_N^2\bigl(\E[X_i]\bigr)^2.
\]
Adding part \(\mathrm{(I)}\) and the bounds for parts
\(\mathrm{(II)}\) and \(\mathrm{(III)}\) yields~\eqref{en:eq:app-reinsurance-ru}.
\end{proof}

\noindent\textit{Remark (separate bounds).} If the \textcolor{black}{terms in part}
\(\mathrm{(II)}\) are bounded separately, its contribution is at most
\(\frac{\mu_N}{\tau}(\frac{\alpha}{4}+\gamma\tau)\E[X_1^2]\):
one uses \(\E[X_1-C_\tau(X_1)]\le\frac{\E[X_1^2]}{4\tau}\) and
\(\operatorname{Var}(C_\tau(X_1))\le\E[X_1^2]\) separately. The joint bound
replaces the coefficient \(\frac{\alpha}{4}+\gamma\tau\) by
\(K_2(\alpha,\gamma\tau)\).

\noindent\textit{Comment and intuition:} What does this estimate give us?
Even without knowing the full loss distribution, we can obtain an upper
bound on the company's risk-adjusted costs. We only need the first two
moments of the losses and of the number of claims. The main point is that
our result bounds the reinsurance loading and the variability of our own
payments together, rather than separately. This gives a bound that is
\textcolor{black}{at least as tight} for the same reinsurance contract.

Cost and risk estimates for reinsurance have a long history. Li, Li, and
Young study \textcolor{black}{the choice of reinsurance and investment
strategies under a criterion involving the mean and variance of terminal
wealth}~\citep{en:li2017excesslossRU}.
Hu, Yang, and Zhang study \textcolor{black}{the choice of a reinsurance
threshold using the first two moments and the range of possible values
of the aggregate loss}~\citep{en:hu2015retentionRU}. In formulating our
insurance model, we drew on Bernis, Di Girolami, and
Scotti~\citep{en:bernis2026reinsuranceRU}: they consider clustered losses and
seek an optimal reinsurance contract. Our application uses a simplified
model with a random number of claims, loss amounts independent of that
number, and the insurer's per-claim retention threshold fixed in advance.

\subsection{Learning a trading allocation with clipping}

Imagine an investment fund that decides, at the start of each period,
what fraction of its wealth to invest in a risky asset and what fraction
to hold in cash. The fund needs to strike a balance: invest enough in the
risky asset to make a profit, without exposing too much wealth to risk.
We are interested in \textcolor{black}{the difference between the logarithms
of two terminal wealth levels}: \textcolor{black}{the wealth} of a fund whose
allocation changes from period to period and \textcolor{black}{the wealth}
of a fund that invests the same fraction in the risky asset in every period.

\textcolor{black}{(Logarithmic returns are convenient in financial
mathematics: the wealth growth factors over successive periods multiply,
whereas their logarithms add. Comparing terminal wealth levels can therefore
be reduced to summing the corresponding comparisons over individual periods.)}

This is the comparison studied by Helmbold, Schapire, Singer, and
Warmuth~\citep[Sections~2.1--2.2]{en:helmbold1998portfolioRU}, whose allocations
are updated multiplicatively. Here we consider a simple version of the
problem: one risky asset and a cash account, and a projected gradient
update with clipping in place of their rule.

(Asset selection in portfolio management is widely studied.
For example, Ito et al.~\citep{en:ito2018cardinalityRU} study sequential
portfolio selection under a constraint on the number of assets.)

We now formulate the problem mathematically. Let \(T\in\mathbb N_+\) be
the total number of periods and let \(t\) index the current period.
The portfolio consists of one risky asset and a cash account.
\textcolor{black}{(For simplicity, assume that the cash account pays no
interest, the asset makes no intermediate payments, and trading is
costless. This lets us focus on how the allocation update affects wealth.)}
Write \(P_{t-1}\) for the asset price at the start of period
\(t\), and \(P_t\) for its price at the end, and put
\(r_t:=P_t/P_{t-1}>0\). Before \(r_t\) is known, choose the fraction
\(w_t\in[0,1]\) of wealth invested in the risky asset\textcolor{black}{.
Hold the remainder} in the cash account. The condition \(w_t\in[0,1]\)
rules out borrowing, leverage, and short selling. Take initial wealth
to be one. Write \(W_T\) for the terminal wealth of this sequential
strategy, and \(W_T(u)\) for the terminal wealth of the comparison strategy (the portfolio is
rebalanced to a fixed risky-asset fraction \(u\in[0,1]\) in every period).
\textcolor{black}{Over one period, the cash component is unchanged,
while the invested component is multiplied by \(r_t\). Total wealth is
therefore multiplied by \((1-w_t)+w_t r_t=1+w_t(r_t-1)\).} Hence
\[
W_T=\prod_{t=1}^T\bigl(1+w_t(r_t-1)\bigr),
\qquad
W_T(u)=\prod_{t=1}^T\bigl(1+u(r_t-1)\bigr).
\]
We next specify the allocation update. Before trading begins, fix an
initial allocation \(w_1\in[0,1]\), a step size \(\mu>0\), and a clipping
threshold \(\tau>0\). At the end of period \(t\), once
\textcolor{black}{the factor \(r_t\) is known}, compute a signal \(g_t\).
It describes how the log wealth over the \textcolor{black}{completed}
period \textcolor{black}{would have changed} under a small increase in
the risky-asset fraction. Since the one-period wealth factor is
\(1+w_t(r_t-1)\), the signal is
\[
g_t=\frac{r_t-1}{1+w_t(r_t-1)}.
\]
Set the allocation for the next period by
\[
w_{t+1}=\mathcal{P}_{[0,1]}\bigl(w_t+\mu C_\tau(g_t)\bigr).
\]
\Needspace{4\baselineskip}
The parameter \(\mu\) sets the step size: the larger it is, the further
the signal \(g_t\) shifts the allocation before projection. Clipping
bounds the signal so that a sudden market shock does not cause an
excessive change in the allocation. The operator
\(\mathcal{P}_{[0,1]}\) projects the resulting value onto the feasible
interval \([0,1]\), keeping the invested fraction nonnegative and ensuring
that the amount invested does not exceed the fund's total wealth.

How far can \textcolor{black}{the logarithm of terminal wealth} \(W_T\) fall behind
\textcolor{black}{the logarithm of terminal wealth} \(W_T(u)\)
\textcolor{black}{corresponding to} a constant allocation \(u\)?
The following corollary gives a bound simultaneously for every \(u\in[0,1]\).

\begin{corollaryru}
\label{en:cor:app-trading-ru}
For any \(r_1,\ldots,r_T>0\), \(\mu,\tau>0\), and \(w_1\in[0,1]\),
the stated update gives, simultaneously for every \(u\in[0,1]\),
\begin{equation}\label{en:eq:app-trading-ru}
\log\frac{W_T(u)}{W_T}
\le \frac{(u-w_1)^2}{2\mu}
+\frac{K_2(1,\mu\tau/2)}{\tau}\sum_{t=1}^T g_t^2.
\end{equation}
\end{corollaryru}

\begin{proof}
To bound the cumulative \textcolor{black}{logarithmic} shortfall over
all \(T\) periods, first bound the one-period shortfall. Put
\(f_t(w)=\log(1+w(r_t-1))\). The function \(f_t\) is concave and
\(f_t'(w_t)=g_t\), \textcolor{black}{so}\par
\vspace{0.3\baselineskip}
\noindent
\begin{minipage}[c]{0.66\textwidth}
\vspace{0pt}
\[
\begin{aligned}
f_t(u)-f_t(w_t)&\le g_t(u-w_t)\\
&=C_\tau(g_t)(u-w_t)
+\bigl(g_t-C_\tau(g_t)\bigr)(u-w_t)\\
&\le C_\tau(g_t)(u-w_t)
+|g_t-C_\tau(g_t)|\,|u-w_t|\\
&\le C_\tau(g_t)(u-w_t)+|g_t-C_\tau(g_t)|.
\end{aligned}
\]
\end{minipage}\hfill
\begin{minipage}[b]{0.31\textwidth}
\vspace{0pt}
\centering
\(\left(
\begin{gathered}
\text{the last step uses}\\
\text{ }u,w_t\in[0,1],\\
\text{hence }|u-w_t|\le1
\end{gathered}
\right)\)
\end{minipage}\par
\vspace{0.3\baselineskip}

\textcolor{black}{Now compare the squared distances of \(w_t\) and
\(w_{t+1}\) from \(u\). Their difference bounds the term
\(C_\tau(g_t)(u-w_t)\), and successive squared-distance terms will cancel when we sum over periods:}
\[
\begin{gathered}
(w_{t+1}-u)^2
\le
(w_t+\mu C_\tau(g_t)-u)^2
=(w_t-u)^2+2\mu C_\tau(g_t)(w_t-u)
+\mu^2 C_\tau(g_t)^2.\\[0.6em]
\left(\begin{gathered}
\text{the first step uses the projection property:}\\
|\mathcal{P}_{[0,1]}(x)-u|\le|x-u|\quad (u\in[0,1])
\end{gathered}\right)
\end{gathered}
\]
Rearranging the squared-distance terms gives a bound for
\(C_\tau(g_t)(u-w_t)\):
\[
\begin{aligned}
2\mu C_\tau(g_t)(u-w_t)
&\le (u-w_t)^2-(u-w_{t+1})^2+\mu^2 C_\tau(g_t)^2,\\
C_\tau(g_t)(u-w_t)
&\le \frac{(u-w_t)^2-(u-w_{t+1})^2}{2\mu}
+\frac{\mu}{2}C_\tau(g_t)^2.
\end{aligned}
\]
Substitute this estimate into the preceding inequality for
\(f_t(u)-f_t(w_t)\). We obtain
\[
\begin{aligned}
f_t(u)-f_t(w_t)
&\le C_\tau(g_t)(u-w_t)+|g_t-C_\tau(g_t)|\\
&\le \frac{(u-w_t)^2-(u-w_{t+1})^2}{2\mu}
+|g_t-C_\tau(g_t)|+\frac{\mu}{2}C_\tau(g_t)^2\\
&=\underbrace{\frac{(u-w_t)^2-(u-w_{t+1})^2}{2\mu}}_{\mathrm{(I)}}
+\underbrace{|g_t-C_\tau(g_t)|}_{\mathrm{(II)}}
+\underbrace{\frac{\mu}{2}C_\tau(g_t)^2}_{\mathrm{(III)}}.
\end{aligned}
\]
Bound parts \(\mathrm{(II)}\) and \(\mathrm{(III)}\) jointly using
Theorem~\ref{en:thm:deterministic-ru}:
\[
|g_t-C_\tau(g_t)|+\frac{\mu}{2}C_\tau(g_t)^2
\le\frac{K_2(1,\mu\tau/2)}{\tau}g_t^2.
\]
\Needspace{20\baselineskip}
Now sum the inequality over all periods \(t=1,\ldots,T\).
The sum on the left is the difference between the logarithms of the terminal wealth levels:
\[
\begin{aligned}
\sum_{t=1}^T\bigl(f_t(u)-f_t(w_t)\bigr)
&=\sum_{t=1}^T\log\bigl(1+u(r_t-1)\bigr)
-\sum_{t=1}^T\log\bigl(1+w_t(r_t-1)\bigr)\\
&=\log W_T(u)-\log W_T
=\log\frac{W_T(u)}{W_T}.
\end{aligned}
\]
For part \(\mathrm{(I)}\), adjacent squared distances cancel:
\[
\begin{aligned}
\sum_{t=1}^T
\frac{(u-w_t)^2-(u-w_{t+1})^2}{2\mu}
&=\frac{(u-w_1)^2-(u-w_{T+1})^2}{2\mu}.
\end{aligned}
\]
For parts \(\mathrm{(II)}\) and \(\mathrm{(III)}\), sum the bound already
obtained for each period:
\[
\sum_{t=1}^T\left(
|g_t-C_\tau(g_t)|+\frac{\mu}{2}C_\tau(g_t)^2
\right)
\le \frac{K_2(1,\mu\tau/2)}{\tau}
\sum_{t=1}^T g_t^2.
\]
Combining the three sums gives
\[
\log\frac{W_T(u)}{W_T}
\le \frac{(u-w_1)^2-(u-w_{T+1})^2}{2\mu}
+\frac{K_2(1,\mu\tau/2)}{\tau}\sum_{t=1}^T g_t^2.
\]
Dropping the nonpositive term
\(-\frac{(u-w_{T+1})^2}{2\mu}\) yields~\eqref{en:eq:app-trading-ru},
completing the proof.
\end{proof}

\noindent\textit{Remark.} If parts
\(\mathrm{(II)}\) and \(\mathrm{(III)}\) are bounded separately, their
sum is at most \(\frac{1}{\tau}(\frac14+\frac{\mu\tau}{2})g_t^2\), since
\(|g_t-C_\tau(g_t)|\le\frac{g_t^2}{4\tau}\) and \(C_\tau(g_t)^2\le g_t^2\).
The joint bound replaces the coefficient \(\frac14+\frac{\mu\tau}{2}\)
by \(K_2(1,\mu\tau/2)\).

\noindent\textit{Comment and intuition:} We can now interpret the bound:
for a prescribed comparison allocation \(u\), step size \(\mu\), and
threshold \(\tau\), the \textcolor{black}{logarithmic} wealth shortfall of the
fund with a changing allocation relative to the fund with a constant allocation
is bounded in terms of the initial allocation and the sum of squared
signals over all periods.


\section{Related work and motivation}
\label{en:sec:scope-ru}

{\color{black}
The idea for this work arose while the author was reading papers on stochastic
gradient descent with clipping, in particular those by Koloskova et al.
and Nguyen et al.~\citep{en:koloskova2023revisitingRU,en:nguyen2023improvedRU}.
In such problems, one regularly has to account for two effects: the bias of the mean
after clipping and the second moment of the remaining vector. Bounds on these quantities
affect the choice of clipping threshold and the convergence guarantees for algorithms. This led the author
to ask: when both quantities arise in the same argument, how sharply
can they be bounded jointly?
\par}

{\color{black}
These bounds continue to be refined. Li and Liu use bias
and variance bounds in their high-probability convergence
analysis~\citep{en:li2023clippingRU}. Parletta, Paudice, and Salzo study
methods with clipping under heavy-tailed noise and obtain refined convergence
bounds~\citep{en:parletta2024clippedRU}. He studies the dependence of bias
and mean squared error on the threshold for coordinatewise clipping
under additional assumptions on the noise
tails~\citep{en:he2025biasvarianceRU}. In these problems, a more precise understanding
of the effects of clipping helps justify parameter choices and assess
algorithmic behavior.
\par}

{\color{black}
The work of Li and Liu~\citep{en:li2023clippingRU} illustrates what a
joint bound provides. In Appendix~A.1,
equations~(3) and (6), for \(1<p\le2\), \(\tau>0\), and
\(\E[\norm{X}^p]\le G^p\), the authors bound the bias separately by
\(G^p\tau^{1-p}\) and the variance after clipping by
\(G^p\tau^{2-p}\). Adding these bounds with nonnegative weights
\(\alpha\) and \(\frac{\beta}{\tau}\) gives the coefficient \(\alpha+\beta\).
Theorem~\ref{en:thm:stochastic-ru} and
Proposition~\ref{en:prop:stochastic-sharp-ru} allow it to be replaced by the sharp
joint constant \(K_p(\alpha,\beta)\). For example, when \(p=2\) and
\(\alpha=\beta=1\), the bound \(\frac{2G^2}{\tau}\) becomes
\(\frac{G^2}{\tau}\).
\par}

{\color{black}
Our work is also related to sharp moment bounds for the residual removed by clipping.
The classical Scarf--Lo bounds, presented by
Marinelli~\citep[Theorem~4.1 and Corollary~4.2]{en:marinelli2024semiparametricRU},
determine the largest mean residual for a given mean and variance.
Kleer et al.~\citep[Theorem~3.4]{en:kleer2024expectationRU} consider
other moment constraints. In these problems, the sharp bounds are attained
by two-point distributions. For a nonnegative random variable,
the mean residual equals the bias after clipping. We bound it
jointly with the variance of the clipped vector under a \(p\)th-moment constraint.
\par}

{\color{black}
To the author's surprise, none of the works examined contained a direct analogue of
Theorem~\ref{en:thm:all-thresholds-ru} on a simultaneous three-dimensional realization
at all clipping thresholds.
\par}

\section{Formal verification in Lean 4}

All numbered theorems, \textcolor{black}{propositions, corollaries, and the remark} in Sections~2--5
(results~2.1--5.2) of this paper have been formalized in Lean~4 using the
mathlib library. \textcolor{black}{The proof of each of these results is written in Lean
and checked by the system.}

The formalization source files and instructions for reproducing the verification are available in
the public repository
{\hypersetup{urlcolor=geoPtwo}\url{https://github.com/Pure-Maths/Radial-Clipping-Verification}}.
The repository files are organized by purpose. The \texttt{Numbered} folder contains
a separate file for each numbered result, with its formal
statement and a list of the axioms used. The main proofs and
auxiliary results are collected in the \texttt{SharpRadialClipping} folder.

\textcolor{black}{After installing the required tools according to the README instructions,
the formalization can be checked with the command}
\texttt{lake build SharpRadialClipping}\textcolor{black}{. Other} commands are given in the repository's
README.

\section{Use of AI in this work}
\label{en:sec:ai-use-ru}

The author made extensive use of artificial intelligence while preparing
this work. The account below describes the tasks for which large language models were used,
the specific models involved, and the decisions \textcolor{black}{made by the author}.

\noindent\textit{Technical steps in the proofs.}
GPT-5.6 Luna and
GPT-5.6 Terra were used to work through individual technical steps. The author set the direction of the argument,
specified which facts and transformations to apply, and the models helped write out the steps in detail.
\textcolor{black}{The author determined the overall proof strategy, the role of intermediate statements, and
the subsequent direction of the research.} Proposed steps were included in a
proof only after the author had checked their logic on paper.

\medskip
\noindent\textit{Searching for possible applications.}
GPT-5.6 Luna and, for a more thorough search, GPT-5.6 Sol were used
to find papers and problems in which the results of this paper might apply.
\textcolor{black}{The author consulted the publications found, checked their content,
and made his own assessment of whether a result of this paper applied to the problem in question.}

\medskip
\noindent\textit{Searching for related work and assessing novelty.}
For an extended search for possible analogues, the author used approximately 30--40 agents
running GPT-5.6 Sol. \textcolor{black}{Their responses helped identify related work for
further study. The author checked the primary sources, compared the assumptions and
conclusions of the results found with those of this paper, and used these comparisons
to formulate conclusions about their differences and novelty.}

\medskip
\noindent\textit{Formalization in Lean.}
The author turned to formalization after the main text of the paper had been
written. AI helped translate the proofs into Lean~4 and resolve technical
difficulties in the formalization. \textcolor{black}{This verification was an additional stage
of work on results that had already been obtained.}

\medskip
\noindent\textit{Preparing source files and text.}
AI was used to speed up the preparation of LaTeX code and the formatting of proofs,
diagrams, and illustrations. \textcolor{black}{When writing the initial text, in
some cases the author expressed the intended idea in his own words, and a model helped
shape it into a paragraph. The author then edited the proposed text and
checked that it accurately conveyed the intended meaning.}

\medskip
\noindent\textcolor{black}{\textit{Final proofreading and clarification of the exposition.}}
\textcolor{black}{After preparing the main text, the author used GPT-6.1 Sol
to review the paper systematically from a reader's perspective. The model helped
identify steps that needed more explanation, check the consistency of
notation, and suggest local changes to the text and LaTeX formatting.
Particular attention was paid to whether the reader could understand why a new
object was introduced, what the next step was for, and how it connected to the preceding steps. The author
discussed the suggestions, checked their mathematical meaning, and decided
which changes to make while preserving the overall intent and style of exposition.}

\textcolor{black}{Thus, AI was used to work through technical steps,
search the literature, carry out formalization, and prepare the text. The mathematical ideas,
overall proof strategy, choice of applications, and final decisions
were determined by the author. He checked the suggestions he accepted and is responsible
for the content of the paper.}

\section{Conclusion}

In this work, we sought a sharp description of two effects of radial
clipping: the part of the vector that is removed and the energy that remains.
For the deterministic problem in any nonzero real normed
space and for random vectors \textcolor{black}{with a finite \(p\)th moment}
in a Hilbert space, we obtained,
for \(1\le p\le2\) and \(\alpha,\beta\ge0\), the same optimal
constant \(K_p(\alpha,\beta)\). \textcolor{black}{By studying the geometry, we described
the attainable set of normalized energy--bias pairs and determined
which boundary points are attained and which can only be approached.} For
fixed \(\tau\), each such pair admits a very simple realization:
a random variable taking at most two values suffices.

\textcolor{black}{We then considered whether a single random vector could
reproduce the behavior at all thresholds simultaneously. For an integrable random
vector taking values in a Hilbert space, we constructed a single random
vector \(Z\) taking values in \(\mathbb R^3\) that reproduces the same bias magnitude
and the same energy for every \(\tau>0\). No new distribution need be chosen
for each new threshold. The distribution of the
norm of the original vector is also preserved.}

Our result applies in two different fields. In insurance theory,
it bounds risk-adjusted costs under reinsurance with a prescribed
limit. In the analysis of algorithmic portfolio management strategies,
it bounds the logarithmic shortfall of a strategy with clipping relative to
a strategy that invests any constant fraction of wealth in the risky asset. \textcolor{black}{Thus, considering the residual
and energy jointly yields sharp bounds, describes
the geometry of attainable pairs, and makes this connection useful in applications.}

\textcolor{black}{We also wish to state the limitations of this work explicitly.
We considered radial clipping and did not obtain bounds for arbitrary
truncation procedures. The stochastic results are proved in Hilbert
spaces. We did not investigate their extension to arbitrary normed
spaces.}

\textcolor{black}{In the applications, we showed how our bound can
sharpen the bounds on costs and strategy shortfall. We did not propose new methods
for choosing reinsurance contracts or new portfolio management
algorithms.}

\textcolor{black}{Finally, in Theorem~4.5 we showed why dimension three
suffices. The paper does not prove that this dimension is necessary.}

In preparing this work, AI helped work through technical steps,
search the literature, translate proofs into Lean, and edit the text.
The author checked the models' suggestions and made the final decisions.
Further details are given in Section~\ref{en:sec:ai-use-ru}.

\begingroup
\small
\hbadness=2500
\setlength{\bibsep}{1pt}

\endgroup

\clearpage
\endgroup

\setcounter{section}{0}
\setcounter{subsection}{0}
\setcounter{subsubsection}{0}
\setcounter{equation}{0}
\setcounter{figure}{0}
\setcounter{table}{0}
\setcounter{theoremru}{0}
\setcounter{theoremen}{0}
\setcounter{footnote}{0}

\begingroup
\selectlanguage{russian}
\renewcommand{\ArticleTheoremName}{Теорема}
\renewcommand{\ArticleLemmaName}{Лемма}
\renewcommand{\ArticleCorollaryName}{Следствие}
\renewcommand{\ArticlePropositionName}{Утверждение}
\renewcommand{\ArticleRemarkName}{Замечание}
\renewcommand{\norm}[1]{\left\lVert #1\right\rVert}
\pdfbookmark[0]{Русская версия}{language.ru}
\color{ink}
\makerussiantitle
\vspace{-1.2em}

\begin{abstract}
Радиальный клиппинг одновременно делает две вещи: отсекает часть вектора за пределами шара и оставляет внутри шара ограниченный вектор. Отсечённая часть задаёт остаток. Для случайного вектора мы изучаем величину смещения среднего и энергию, определяемую центрированным вторым моментом клиппированного вектора. Контроль этих двух эффектов важен в стохастической оптимизации с тяжёлыми хвостами, где выбор порога клиппинга влияет на оценки сходимости алгоритмов.

В детерминированной задаче для фиксированного порога \(\tau>0\), \(1\le p\le2\) и неотрицательных коэффициентов \(\alpha,\beta\) мы находим наименьшую константу, позволяющую совместно оценить остаток и энергию через \(\norm{x}^p\). Эта же константа остаётся точной для случайных векторов в гильбертовом пространстве. Также мы находим точный супремум соответствующей линейной комбинации нормированных смещения и энергии при произвольных \(\alpha,\beta\in\mathbb R\).

Далее мы строим геометрическую картину на плоскости «энергия--смещение»: описываем наименьшую замкнутую выпуклую область, содержащую все нормированные пары, и само множество достижимых пар, указывая недостижимые граничные значения. Каждую достижимую пару при фиксированном пороге можно получить с помощью вещественной случайной величины, принимающей не более двух значений.

Мы также рассматриваем все пороги сразу. Для любого интегрируемого случайного вектора в гильбертовом пространстве мы строим один случайный вектор со значениями в \(\mathbb R^3\), сохраняющий распределение нормы, величину смещения и энергию одновременно при всех порогах \(\tau>0\).

В заключение мы применяем совместную оценку к двум задачам: оценке расходов с учётом риска при перестраховании с заданным порогом собственного удержания и оценке логарифмического отставания портфельной стратегии с клиппированием от стратегии с любой постоянной долей. При подготовке работы использовался ИИ (см. раздел~\ref{ru:sec:ai-use-ru}).
\end{abstract}

\medskip
\noindent\textbf{Ключевые слова.} Радиальный клиппинг\textcolor{black}{,} тяжелые хвосты\textcolor{black}{,} точные
неравенства\textcolor{black}{,} смещение\textcolor{black}{,} дисперсия\textcolor{black}{,} опорные функции.

\noindent\textbf{Математическая классификация 2020.}
60E15 (основная)\textcolor{black}{,} 46N30 (дополнительная).

\section{Введение}

Радиальный клиппинг широко используется при работе с тяжелохвостными данными\textcolor{black}{,
в частности в стохастической оптимизации~{\hypersetup{citecolor=black}\citep{ru:koloskova2023revisitingRU,ru:nguyen2023improvedRU}}.}
\textcolor{black}{Он} заменяет слишком большие наблюдения ограниченными векторами, сохраняя их
направление. Однако при таком обрезании часть информации теряется: среднее
может сместиться, а величина и разброс оставшейся части данных --- измениться.
\textcolor{black}{Для случайного вектора нас интересуют величина смещения среднего
и энергия --- мера разброса клиппированного вектора вокруг своего среднего.}

{\hypersetup{citecolor=black,linkcolor=black}\color{black}
В исследованиях алгоритмов с клиппингом смещение и дисперсию часто оценивают
по отдельности~\citep{ru:li2023clippingRU}. В этой работе мы ставим вопрос об их
точной совместной оценке. Сравнение с конкретными оценками из литературы
приведено в разделе~\ref{ru:sec:scope-ru}.
\par}

Основной вклад настоящей работы состоит в точном совместном описании смещения
и энергии радиального клиппинга. Мы находим наименьшую константу в
соответствующем неравенстве и показываем, что она остаётся оптимальной не
только для фиксированных, но и для случайных векторов. Кроме того, для
рассматриваемого класса случайных векторов мы полностью описываем, как энергия
и смещение могут соотноситься друг с другом.

\subsection{Постановка задачи}\label{ru:subsec:intro-deterministic-ru}

Зафиксируем основной объект исследования. Пусть \(E\) --- ненулевое
вещественное нормированное пространство и \(\tau>0\). Определим оператор
радиального клиппинга \(C_\tau\colon E\to E\) формулой
\[
C_\tau(x)=
\begin{cases}
x, & \norm{x}\le\tau,\\[1mm]
\displaystyle \tau\frac{x}{\norm{x}}, & \norm{x}>\tau.
\end{cases}
\]
После этого естественно определить две величины. Первая --- остаток ---
показывает, насколько сильно клиппинг изменил исходный вектор. Вторая ---
энергия --- характеризует ту часть вектора, которая
сохранилась после клиппинга:
\[
R_\tau(x):=\norm{x-C_\tau(x)},
\qquad
Q_\tau(x):=\norm{C_\tau(x)}^2.
\]
Чтобы задать относительную важность остатка и энергии, введем неотрицательные
коэффициенты \(\alpha\) и \(\beta\) соответственно. Объединим эти величины в
согласованной по размерности форме:
\[
\alpha R_\tau(x)+\frac{\beta}{\tau}Q_\tau(x),
\]
где множитель \(1/\tau\) приводит энергию к той же размерности, что и остаток.

Зафиксируем \(1\le p\le2\). Нас интересует оценка рассматриваемой величины
через \(p\)-ю степень нормы входного вектора. Поэтому в качестве контролирующей
величины естественно взять
\[
\tau\left(\frac{\norm{x}}{\tau}\right)^p.
\]
Здесь \(\norm{x}/\tau\) характеризует относительную величину нормы \(\norm{x}\)
по сравнению с порогом клиппинга, а множитель \(\tau\) служит для сохранения
размерности.

Поэтому итоговый вопрос можно сформулировать следующим образом: какова
наименьшая константа \(K_p(\alpha,\beta)\), для которой при всех \(x\in E\)
выполняется неравенство
\[
\alpha R_\tau(x)+\frac{\beta}{\tau}Q_\tau(x)
\le
K_p(\alpha,\beta)\,
\tau\left(\frac{\norm{x}}{\tau}\right)^p.
\]

Оказывается, что при \(1<p\le2\) наименьшая возможная константа имеет вид
\[
K_p(\alpha,\beta)=
\begin{cases}
\beta, & \alpha\le p\beta,\\[2mm]
\displaystyle
\frac{(p-1)^{p-1}}{p^p}
\frac{\alpha^p}{(\alpha-\beta)^{p-1}},
& \alpha>p\beta,
\end{cases}
\]
В граничном случае \(p=1\) она принимает вид
\[
K_1(\alpha,\beta)=\max\{\alpha,\beta\}.
\]

\subsection{От одного вектора к случайному наблюдению}

После постановки задачи для одного фиксированного вектора \(x\) возникает
следующий вопрос: что изменится, если заменить \(x\) случайным вектором?
Останется ли найденная в детерминированном случае константа оптимальной, если
\(X\in L^p(\Omega;H)\), то есть \(X\) --- случайный вектор
\textcolor{black}{на произвольном вероятностном пространстве \(\Omega\)} со значениями в
\textcolor{black}{ненулевом} вещественном гильбертовом пространстве \(H\), для которого
\(\E[\norm{X}^p]<\infty\)?

После клиппинга мы получаем случайный вектор \(C_\tau(X)\), все значения
которого лежат внутри шара радиуса \(\tau\). Как и в детерминированном случае,
клиппинг создает два эффекта. Во-первых, он может сдвинуть среднее: вместо
исходного центра \(\E[X]\) мы получаем новый центр \(\E[C_\tau(X)]\).
Во-вторых, значения \(C_\tau(X)\) продолжают отклоняться от нового центра, и
величина этого разброса характеризует оставшуюся после клиппинга энергию.
Поэтому естественно рассмотреть две величины
\[
\norm{\E[C_\tau(X)]-\E[X]},
\qquad
\E\!\left[\norm{C_\tau(X)-\E[C_\tau(X)]}^2\right].
\]

Как и в подразделе~1.1, с помощью неотрицательных коэффициентов \(\alpha\) и
\(\beta\) учтем вклад смещения и энергии в общую оценку. Получившуюся величину
также будем контролировать \(p\)-м моментом исходного случайного вектора \(X\).

Поэтому возникает аналогичная задача: найти наименьшую константу \(K\),
\textcolor{black}{общую для всех таких вероятностных пространств,} при
которой для любого \(X\in L^p(\Omega;H)\) выполняется неравенство
\[
\alpha\norm{\E[C_\tau(X)]-\E[X]}
+\frac{\beta}{\tau}
\E\!\left[\norm{C_\tau(X)-\E[C_\tau(X)]}^2\right]
\le
K\,
\tau\E\!\left[\left(\frac{\norm{X}}{\tau}\right)^p\right].
\]
Оказывается, что наименьшая константа \(K\) в этом неравенстве в точности
совпадает с константой \(K_p(\alpha,\beta)\), найденной в
подразделе~\ref{ru:subsec:intro-deterministic-ru}.

Отдельно отметим, что мы также рассмотрим условную версию полученного
неравенства относительно некоторой \(\sigma\)-алгебры \(\mathcal G\), в которой
все ожидания берутся условно. Она позволяет применять неравенство, используя
только информацию, накопленную к данному моменту времени, и полезна при анализе
стохастических процессов и последовательных алгоритмов (например, в
стохастической оптимизации и онлайн-обучении).

\Needspace{12\baselineskip}
\subsection{Геометрия смещения и энергии}

До сих пор мы фиксировали коэффициенты \(\alpha\) и \(\beta\) и спрашивали,
насколько большой может быть соответствующая комбинация полученного смещения
и энергии. Теперь посмотрим на ту же задачу с другой стороны. Зададимся
вопросом: как смещение и энергия соотносятся друг с другом?

Для этого при \(0<\E[\norm{X}^p]<\infty\) нормируем энергию и смещение с помощью
\(p\)-го момента исходного вектора и радиуса клиппинга. При фиксированных
\(p\) и \(\tau\) введем координаты энергии \(V_p\) и смещения \(D_p\):
\[
V_p(X):=
\frac{\tau^{p-2}\E\!\left[\norm{C_\tau(X)-\E[C_\tau(X)]}^2\right]}
{\E[\norm{X}^p]},
\qquad
D_p(X):=
\frac{\tau^{p-1}\norm{\E[C_\tau(X)]-\E[X]}}
{\E[\norm{X}^p]}.
\]
Зафиксировав случайный вектор \(X\) и радиус клиппинга \(\tau\), мы получаем
точку \((V_p(X),D_p(X))\) на плоскости «энергия--смещение».

Также отметим, что такую нормировку мы выбрали для того, чтобы после деления
неравенства из подраздела~1.2 на \(\tau^{1-p}\E[\norm{X}^p]\) оно приняло простой
безразмерный вид:
\[
\boxed{\textcolor{black}{\beta V_p(X)+\alpha D_p(X)}\le K_p(\alpha,\beta).}
\]
Эта запись имеет естественную геометрическую интерпретацию. При фиксированных
\(\alpha,\beta\ge0\), не равных нулю одновременно, равенство
\[
\textcolor{black}{\beta v+\alpha d}=K_p(\alpha,\beta),
\]
задаёт опорную прямую замкнутой выпуклой оболочки достижимых пар
\((V_p(X),D_p(X))\). Иными словами, с помощью \(K_p(\alpha,\beta)\) мы можем
строить опорные прямые к \textcolor{black}{замкнутой выпуклой оболочке} множества достижимых
пар. \textcolor{black}{Далее мы сначала опишем эту оболочку, а затем --- само множество
достижимых пар и его границу.}

После того как мы опишем всю геометрию этого множества, мы ответим на
следующий вопрос: насколько сложно должна быть устроена случайная величина,
чтобы реализовать заданную достижимую пару \((V_p(X),D_p(X))\)? \textcolor{black}{Мы}
покажем, что каждую достижимую нормированную пару \(\textcolor{black}{(v,d)}\) можно реализовать
одномерной случайной величиной, принимающей не более двух значений. При этом
сохраняется именно пара \(\textcolor{black}{(v,d)}\), но не обязательно исходные моменты
распределения.

\textcolor{black}{Но пока мы воспроизводили одну пару при фиксированном пороге.
Далее мы выясним, можно ли заменить исходный случайный вектор
одним новым вектором, который при каждом пороге \(\tau>0\) даёт ту же величину
смещения и ту же энергию --- без необходимости подбирать его заново при
изменении порога. При этом мы потребуем сохранить и распределение нормы
исходного вектора. Мы покажем, что такой случайный вектор можно построить
в \(\R^3\).}

\subsection{Приложения}

Нашу совместную оценку можно использовать там, где клиппинг одновременно
меняет две величины: отсекает часть наблюдения и ограничивает то, что остаётся
после отсечения. В разделе~5 мы рассмотрим две такие задачи --- перестрахование
и управление портфелем.

В первой задаче страховая компания заранее выбирает лимит выплаты по каждому
страховому случаю, а сумму сверх лимита передаёт перестраховщику. Мы покажем,
как оценить расходы компании с учётом риска, если известны первые два момента
размеров убытков и числа страховых случаев. Во второй задаче фонд меняет долю
рискованного актива от периода к периоду, ограничивая клиппингом величину
сигнала для обновления. Здесь мы оценим, насколько \textcolor{black}{логарифм итогового капитала} этой
стратегии может уступить \textcolor{black}{логарифму капитала} стратегии с любой постоянной долей.

\subsection{Структура статьи}

Дальнейшее изложение построено следующим образом. В разделе~2 мы решаем
детерминированную задачу и находим точные константы для \(1\le p\le2\),
отдельно разбирая граничный случай \(p=1\). В разделе~3 мы
переносим оценку на случайные векторы, доказываем точность той же константы и
получаем условную форму неравенства. В разделе~4 исследуем нормированные пары
энергии и смещения: описываем их замкнутую выпуклую оболочку и само множество
достижимых пар, находим опорные прямые для произвольных весов и показываем,
что каждую достижимую пару можно реализовать двухатомной случайной величиной.
Там же доказываем, что один случайный вектор в \(\mathbb R^3\) воспроизводит
\textcolor{black}{величину смещения} и энергию при всех порогах, сохраняя распределение нормы.
В разделе~5 рассматриваем два приложения: перестрахование при случайном числе
убытков и заранее заданном лимите, а также алгоритм управления портфелем с
клиппированием. В разделе~6 сопоставляем наши результаты с близкими работами
и объясняем мотивацию исследования. Раздел~7 посвящён формальной проверке
нумерованных результатов разделов~2--5 в Lean~4 и тому, как её повторить.
В разделе~8 описываем использование ИИ при подготовке статьи. В разделе~9
подводим итоги работы.

\section{Точные константы в детерминированной задаче}

В этом разделе мы решаем детерминированную задачу, поставленную в
подразделе~\ref{ru:subsec:intro-deterministic-ru}, и находим точную константу
\(K_p(\alpha,\beta)\) для всех \(1\le p\le2\). Сначала мы рассмотрим случай
\(1<p\le2\), а затем отдельно разберём граничный случай \(p=1\) и объясним,
почему он требует отдельного рассмотрения.

\Needspace{12\baselineskip}
\subsection[Случай 1 < p <= 2]{Случай \(\bm{1<p\le2}\)}

Зафиксируем основные обозначения, которые будут использоваться далее.
Пусть \(E\) --- ненулевое вещественное нормированное пространство. Для
\(\tau>0\) определим оператор радиального клиппинга:
\vspace{0.35em}
\[
C_\tau(x)=
\begin{cases}
x,&\norm{x}\le\tau,\\[2mm]
\tau\dfrac{x}{\norm{x}},&\norm{x}>\tau.
\end{cases}
\]
\vspace{0.35em}
Также определим функцию \(c\colon(1,2]\to(0,\infty)\):
\[
c(p):=\frac{(p-1)^{p-1}}{p^p}.
\]
Для \(x\in E\) определим его нормированный радиус:
\[
r:=\frac{\norm{x}}{\tau}.
\]

\begin{theoremru}[Точная совместная огибающая]\label{ru:thm:deterministic-ru}
Пусть \(1<p\le2\) и \(\tau>0\). Для \(\alpha,\beta\ge0\) определим
\[
K_p(\alpha,\beta)=
\begin{cases}
\beta,&\alpha\le p\beta,\\[2mm]
c(p)\dfrac{\alpha^p}{(\alpha-\beta)^{p-1}},
&\alpha>p\beta.
\end{cases}
\]
Тогда \(K_p(\alpha,\beta)\) является наименьшей возможной равномерной
константой, при которой выполняется неравенство:
\begin{equation}\label{ru:eq:det-envelope-ru}
\boxed{
\forall\,x\in E\colon\quad
\alpha\norm{x-C_\tau(x)}
+\frac{\beta}{\tau}\norm{C_\tau(x)}^2
\le
K_p(\alpha,\beta)\,\tau
\left(\frac{\norm{x}}{\tau}\right)^p.}
\end{equation}
\end{theoremru}

\begin{proof}
При \(x=0\) утверждение очевидно. Пусть теперь \(x\ne0\). Рассмотрим
отдельно случаи, когда \(x\) лежит внутри замкнутого шара радиуса \(\tau\),
и когда он лежит вне этого шара.

\textit{1.} Вектор \(x\) лежит в замкнутом шаре радиуса \(\tau\). Тогда
\(0<r\le1\) и
\[
C_\tau(x)=x,\qquad
\norm{x-C_\tau(x)}=0,\qquad
\norm{C_\tau(x)}=r\tau.
\]
Поэтому неравенство~\eqref{ru:eq:det-envelope-ru} после деления на \(\tau\)
принимает вид
\[
\beta r^2\le K_p(\alpha,\beta)r^p.
\]
Следовательно, вклад первого случая в искомую константу равен
\[
\sup_{0<r\le1}\beta r^{2-p}=\beta,
\]
поскольку \(2-p\ge0\).

\textit{2.} Вектор \(x\) лежит вне замкнутого шара радиуса \(\tau\). Тогда
\(r>1\). В этом случае
\[
C_\tau(x)=\tau\frac{x}{\norm{x}},\qquad
\norm{x-C_\tau(x)}=(r-1)\tau,\qquad
\norm{C_\tau(x)}=\tau.
\]
После подстановки в~\eqref{ru:eq:det-envelope-ru} и деления на \(\tau\)
получаем неравенство
\[
\alpha(r-1)+\beta\le K_p(\alpha,\beta)r^p.
\]
\textcolor{black}{При \(r=1\) получаем ту же величину \(\beta\), что и в первом случае.
Поэтому оптимальная константа определяется формулой}
\[
K_p(\alpha,\beta)
=
\sup_{r\ge1}f(r),
\qquad
f(r):=\frac{\alpha(r-1)+\beta}{r^p}.
\]
\Needspace{7\baselineskip}
Поскольку
\[
f'(r)
=\frac{p(\alpha-\beta)-(p-1)\alpha r}{r^{p+1}}
=\frac{\alpha-p\beta-(p-1)\alpha(r-1)}{r^{p+1}},
\qquad r\in[1,\infty),
\]
знаменатель положителен, поэтому знак производной определяется числителем.
При \(\alpha\le p\beta\) функция \(f\) невозрастает и её максимум равен
\(f(1)=\beta\). Если \(\alpha>p\beta\), максимум достигается в единственной
точке, в которой производная обращается в нуль:
\[
r_*=\frac{p(\alpha-\beta)}{\alpha(p-1)}>1,
\qquad
f(r_*)=K_p(\alpha,\beta)
=c(p)\frac{\alpha^p}{(\alpha-\beta)^{p-1}}.
\]
Каждый \(r>0\) реализуется выбором \(x=r\tau e\), где \(\norm{e}=1\), поэтому
найденная константа является точной. Совмещая полученные результаты для двух
случаев, получаем искомую точную формулу для \(K_p(\alpha,\beta)\) и её
оптимальность.
\end{proof}

\subsection[Граничный случай p = 1]{Граничный случай \(p=1\)}

При \(p=1\) производная из предыдущего доказательства вырождается: член
\((p-1)\alpha r\) исчезает. По этой причине случай \(p=1\) рассмотрим
отдельно.

\begin{theoremru}[Граничный случай]\label{ru:thm:deterministic-p1-ru}
Для \(\alpha,\beta\ge0\) оптимальная константа в
неравенстве~\eqref{ru:eq:det-envelope-ru} при \(p=1\) равна
\(K_1(\alpha,\beta)=\max\{\alpha,\beta\}\).
\end{theoremru}

\begin{proof}
При \(x=0\) утверждение очевидно. Как и в доказательстве теоремы~\ref{ru:thm:deterministic-ru},
для \(x\ne0\) положим \(r=\norm{x}/\tau\). При \(0<r\le1\) нормированная
левая часть равна \(\beta r\), и её супремум равен \(\beta\). При \(r\ge1\)
она равна \(\alpha+(\beta-\alpha)/r\), и её супремум равен
\(\max\{\alpha,\beta\}\). Значение \(\beta\) достигается при \(r=1\), а
значение \(\alpha\) достигается в пределе при \(r\to\infty\), поэтому
\(K_1(\alpha,\beta)=\max\{\alpha,\beta\}\).
\end{proof}

\section{Точная стохастическая огибающая}

В этом разделе мы переносим теоремы~2.1 и~2.2 в стохастическую постановку,
заменяя фиксированный вектор \(x\) случайным вектором \(X\). Сначала докажем
стохастическую оценку при \(1\le p\le2\), затем установим оптимальность
полученной константы при \(1\le p\le2\) и в заключение получим условную форму
результата.

\textcolor{black}{Далее \(\Omega\) обозначает произвольное вероятностное
пространство. При выборе случайного вектора \(X\) это пространство может
меняться.}

\subsection{Стохастическая оценка}

\begin{theoremru}[Стохастическая огибающая]
\label{ru:thm:stochastic-ru}
Пусть \(\textcolor{black}{H}\) --- ненулевое вещественное гильбертово пространство,
\(1\le p\le2\), \(\tau>0\), а \textcolor{black}{\(X\in L^p(\Omega;H)\)}. Тогда для любых
детерминированных \(\alpha,\beta\ge0\)
\begin{equation}\label{ru:eq:stochastic-upper-ru}
\boxed{
\alpha\norm{\E[C_\tau(X)]-\E[X]}
+\frac{\beta}{\tau}
\E\!\left[\norm{C_\tau(X)-\E[C_\tau(X)]}^2\right]
\le
K_p(\alpha,\beta)\tau\,
\frac{\E[\norm{X}^p]}{\tau^p}.}
\end{equation}
\end{theoremru}

\begin{proof}
Зафиксируем \(\omega\in\Omega\). При \(p>1\) применим к вектору
\(X(\omega)\) теорему~\ref{ru:thm:deterministic-ru}, а при \(p=1\) ---
теорему~\ref{ru:thm:deterministic-p1-ru}. В обоих случаях получаем
\[
\alpha\norm{X(\omega)-C_\tau(X(\omega))}
+\frac{\beta}{\tau}\norm{C_\tau(X(\omega))}^2
\le
K_p(\alpha,\beta)\tau\,
\frac{\norm{X(\omega)}^p}{\tau^p}.
\]
Взяв математическое ожидание от обеих частей (это корректно, поскольку
математическое ожидание сохраняет неравенства между интегрируемыми случайными
величинами), получаем
\[
\begin{array}{c@{\quad}c@{\quad}c@{\qquad}c}
\displaystyle
\alpha\E[\norm{X-C_\tau(X)}]
&
 +
&
\displaystyle
\frac{\beta}{\tau}\E[\norm{C_\tau(X)}^2]
&
\displaystyle
\le K_p(\alpha,\beta)\tau\,
\frac{\E[\norm{X}^p]}{\tau^p}
\\[1mm]
\rotatebox[origin=c]{-90}{\(\geqslant\)}\;{\scriptstyle\text{\rm(I)}}
&
{}
&
\rotatebox[origin=c]{-90}{\(\geqslant\)}\;{\scriptstyle\text{\rm(II)}}
&
{}
\\[1mm]
\displaystyle
\alpha\norm{\E[C_\tau(X)]-\E[X]}
&
 +
&
\displaystyle
\frac{\beta}{\tau}
\E\!\left[\norm{C_\tau(X)-\E[C_\tau(X)]}^2\right].
&
{}
\end{array}
\]
Оценка \(\text{\rm(I)}\) верна, так как норма среднего вектора не превосходит
среднего его нормы:
\[
\norm{\E[C_\tau(X)]-\E[X]}
=\norm{\E[C_\tau(X)-X]}
\le\E[\norm{C_\tau(X)-X}].
\]
Оценка \(\text{\rm(II)}\) следует из тождества для дисперсии (оно применимо,
поскольку \(\norm{C_\tau(X)}\le\tau\), а значит
\(\E[\norm{C_\tau(X)}^2]\le\tau^2<\infty\)):
\[
\E\!\left[\norm{C_\tau(X)-\E[C_\tau(X)]}^2\right]
=\E[\norm{C_\tau(X)}^2]-\norm{\E[C_\tau(X)]}^2
\le\E[\norm{C_\tau(X)}^2].
\]
Складывая оценки \(\text{\rm(I)}\) и \(\text{\rm(II)}\), получаем
\eqref{ru:eq:stochastic-upper-ru}.
\end{proof}

\subsection[Точность константы при 1 <= p <= 2]
{Точность константы при \(1\le p\le2\)}

\begin{propositionru}[Точность константы]\label{ru:prop:stochastic-sharp-ru}
Пусть \(\textcolor{black}{H}\) --- ненулевое вещественное гильбертово пространство,
\(1\le p\le2\), \(\tau>0\) и \(\alpha,\beta\ge0\).
Пусть \textcolor{black}{\(X\in L^p(\Omega;H)\)} и \(\E[\norm{X}^p]\ne0\). Тогда
\begin{equation}\label{ru:eq:stochastic-sup-ru}
\boxed{
\sup_X
\frac{
\alpha\norm{\E[C_\tau(X)]-\E[X]}
+\frac{\beta}{\tau}
\E\!\left[\norm{C_\tau(X)-\E[C_\tau(X)]}^2\right]
}{
\displaystyle \tau\,\frac{\E[\norm{X}^p]}{\tau^p}
}
=K_p(\alpha,\beta).}
\end{equation}
\end{propositionru}

\begin{proof}
По теореме~\ref{ru:thm:stochastic-ru} величина \(K_p(\alpha,\beta)\) является
верхней гранью в~\eqref{ru:eq:stochastic-sup-ru}. Покажем, что эта грань точна.
Зафиксируем \(e\in\textcolor{black}{H}\), для которого \(\norm{e}=1\), и числа \(r\ge1\),
\(0<q<1\). Определим случайный вектор
\[
X_{q,r}=
\begin{cases}
r\tau e,&\text{с вероятностью }q,\\[1mm]
0,&\text{с вероятностью }1-q.
\end{cases}
\]
Тогда
\[
\E[X_{q,r}]=qr\tau e,
\qquad
\E[C_\tau(X_{q,r})]=q\tau e,
\]
\[
\E\!\left[
\norm{C_\tau(X_{q,r})-\E[C_\tau(X_{q,r})]}^2
\right]
=q(1-q)\tau^2,
\qquad
\E[\norm{X_{q,r}}^p]=qr^p\tau^p.
\]
Подставляя полученные значения в выражение под знаком супремума
в~\eqref{ru:eq:stochastic-sup-ru} и сокращая общий множитель \(q\tau\), получаем
\[
\frac{
\alpha\norm{q\tau e-qr\tau e}
+\dfrac{\beta}{\tau}q(1-q)\tau^2
}{
\displaystyle \tau\,\frac{qr^p\tau^p}{\tau^p}
}
=
\frac{q\tau\bigl(\alpha(r-1)+\beta(1-q)\bigr)}{q\tau r^p}
=
\frac{\alpha(r-1)+\beta(1-q)}{r^p}.
\]
При \(q\downarrow0\) последнее выражение стремится к
\(\dfrac{\alpha(r-1)+\beta}{r^p}\). Последнее выражение уже
рассматривалось в доказательствах теорем~\ref{ru:thm:deterministic-ru}
и~\ref{ru:thm:deterministic-p1-ru}: там было установлено, что его супремум
по \(r\ge1\) равен \(K_p(\alpha,\beta)\) для любого \(p\in[1,2]\).
Следовательно, константа в~\eqref{ru:eq:stochastic-sup-ru} точна.
\end{proof}

\subsection{Условная форма}

\begin{corollaryru}[Условная форма]\label{ru:cor:conditional-ru}
Пусть \(\textcolor{black}{H}\) --- ненулевое вещественное гильбертово пространство,
\(1\le p\le2\), а \textcolor{black}{\(X\in L^p(\Omega;H)\)}. Пусть \(\G\) ---
под-\(\sigma\)-алгебра, \(\alpha,\beta\ge0\) детерминированы, а
\(\tau:\Omega\to(0,\infty)\) конечен почти наверное и
\(\G\)-измерим. Тогда почти наверное
\[
\alpha\norm{\E[C_\tau(X)\mid\G]-\E[X\mid\G]}
+\frac{\beta}{\tau}
\E\!\left[
\norm{C_\tau(X)-\E[C_\tau(X)\mid\G]}^2
\,\middle|\,\G
\right]
\le
K_p(\alpha,\beta)\tau\,
\frac{\E[\norm{X}^p\mid\G]}{\tau^p}.
\]
\end{corollaryru}

\begin{proof}
Аналогично доказательству теоремы~3.1, при \(p>1\) применим
теорему~2.1 к вектору \(X(\omega)\) с порогом \(\tau(\omega)\), а при
\(p=1\) --- теорему~2.2. Почти наверное получаем
\[
\alpha\norm{X-C_\tau(X)}
+\frac{\beta}{\tau}\norm{C_\tau(X)}^2
\le
K_p(\alpha,\beta)\tau\,
\frac{\norm{X}^p}{\tau^p}.
\]
Для неотрицательных величин условные математические ожидания понимаются
в расширенном смысле; полученные ниже условные оценки конечны почти наверное.
Далее возьмём везде условное математическое ожидание относительно \(\G\)
и вынесем \(\tau\) из-под знака математического ожидания (мы имеем право
это сделать, поскольку \(\tau\) является \(\G\)-измеримой). Получаем
\[
\begin{array}{c@{\quad}c@{\quad}c@{\qquad}c}
\displaystyle
\alpha\E[\norm{X-C_\tau(X)}\mid\G]
&
 +
&
\displaystyle
\frac{\beta}{\tau}\E[\norm{C_\tau(X)}^2\mid\G]
&
\displaystyle
\le K_p(\alpha,\beta)\tau\,
\frac{\E[\norm{X}^p\mid\G]}{\tau^p}
\\[1mm]
\rotatebox[origin=c]{-90}{\(\geqslant\)}\;{\scriptstyle\text{\rm(I)}}
&
{}
&
\rotatebox[origin=c]{-90}{\(\geqslant\)}\;{\scriptstyle\text{\rm(II)}}
&
{}
\\[1mm]
\displaystyle
\alpha\norm{\E[C_\tau(X)\mid\G]-\E[X\mid\G]}
&
 +
&
\displaystyle
\frac{\beta}{\tau}
\E\!\left[
\norm{C_\tau(X)-\E[C_\tau(X)\mid\G]}^2
\,\middle|\,\G
\right].
&
{}
\end{array}
\]
Оценка \(\text{\rm(I)}\) верна, так как норма условного среднего вектора
не превосходит условного среднего его нормы:
\[
\norm{\E[C_\tau(X)\mid\G]-\E[X\mid\G]}
=\norm{\E[C_\tau(X)-X\mid\G]}
\le\E[\norm{C_\tau(X)-X}\mid\G].
\]
Оценка \(\text{\rm(II)}\) следует из условного тождества для дисперсии:
\[
\E\!\left[
\norm{C_\tau(X)-\E[C_\tau(X)\mid\G]}^2
\,\middle|\,\G
\right]
=\E[\norm{C_\tau(X)}^2\mid\G]
-\norm{\E[C_\tau(X)\mid\G]}^2
\le\E[\norm{C_\tau(X)}^2\mid\G].
\]
Складывая оценки \(\text{\rm(I)}\) и \(\text{\rm(II)}\), получаем, что
нижняя строка не превосходит верхней. Верхняя строка уже ограничена правой
частью, поэтому утверждение доказано.
\end{proof}

\noindent\textit{Замечание.}
Для применения условного тождества для дисперсии требуется, чтобы
\(C_\tau(X)\) был интегрируем и имел конечный условный второй момент. Оба
условия выполнены: \(\norm{C_\tau(X)}\le\norm{X}\) и \(\norm{X}\in L^1\),
а \(\E[\norm{C_\tau(X)}^2\mid\G]\le\tau^2<\infty\) почти наверное, поскольку
\(\norm{C_\tau(X)}\le\tau\), а \(\tau\) конечна почти наверное и
\(\G\)-измерима.

\noindent\textcolor{black}{\textit{Комментарий.} Интуитивно здесь происходит
следующее: мы учитываем уже известную нам информацию и усредняем только то,
что ещё остаётся случайным. Порог клиппинга при этом уже известен, поэтому
рассуждение остаётся тем же.}

\section[Геометрическая интерпретация множества достижимых пар]
{Геометрическая интерпретация\\множества достижимых пар}
\label{ru:sec:geometry-ru}

Мы научились оценивать взвешенную линейную комбинацию двух неотрицательных
эффектов --- энергии и смещения. Теперь довольно естественно спросить, как
эти две величины могут соотноситься друг с другом.

\subsection{Координаты и геометрическая интерпретация стохастического результата}
\label{ru:subsec:geometry-support-ru}

Сохраним обозначения и предположения
\textcolor{black}{утверждения}~\ref{ru:prop:stochastic-sharp-ru}. Для каждого допустимого случайного
вектора \(X\) определим нормированные координаты энергии и смещения формулами
\begin{center}
\makebox[\textwidth][c]{%
{\(
\displaystyle
V_p(X):=
\dfrac{
\dfrac{1}{\tau}\E\!\left[
\norm{C_\tau(X)-\E[C_\tau(X)]}^2\right]
}{
\tau\,\dfrac{\E[\norm{X}^p]}{\tau^p}
}
\)}%
\hspace{4em}
{\(
\displaystyle
D_p(X):=
\dfrac{
\norm{\E[C_\tau(X)]-\E[X]}
}{
\tau\,\dfrac{\E[\norm{X}^p]}{\tau^p}
}
\)}%
}
\end{center}
Далее определим отображение визуализации
\(\Phi_p(X):=\bigl(V_p(X),D_p(X)\bigr)\in\R_+^2\). Рассматривая все
допустимые случайные векторы \(X\), получим множество \(\mathcal A_p\),
состоящее из всех таких достижимых пар \((v,d)\):
\[
\mathcal A_p:=
\bigl\{(v,d)\in\R_+^2:
\text{существует допустимый случайный вектор }X,
\ (v,d)=\Phi_p(X)\bigr\}.
\]

Теперь вернёмся к результату \textcolor{black}{утверждения}~\ref{ru:prop:stochastic-sharp-ru}.
В введённых координатах оно утверждает, что для неотрицательных \(\alpha\)
и \(\beta\) выполнено следующее:
\begin{equation}\label{ru:eq:positive-support-ru}
\sup_{(v,d)\in\mathcal A_p}
\bigl\langle(\beta,\alpha),(v,d)\bigr\rangle
=\sup_{(v,d)\in\mathcal A_p}(\textcolor{black}{\beta v+\alpha d})
=K_p(\alpha,\beta),
\end{equation}
\textcolor{black}{При \((\alpha,\beta)\ne(0,0)\) прямые
\(\beta v+\alpha d=K_p(\alpha,\beta)\) являются опорными к замкнутой
выпуклой оболочке \(\mathcal A_p\). Найдём область в \(\R_+^2\), заданную
соответствующими неравенствами, и докажем, что она совпадает с этой оболочкой.}

\textcolor{black}{Далее считаем \(1<p\le2\)\textcolor{black}{. Случай} \(p=1\)
рассмотрим отдельно в замечании~\ref{ru:rem:attainable-p1-ru}.}

\noindent\textcolor{black}{\textit{Комментарий.} Из-за большого числа технических
преобразований дальше можно потерять основную мысль. На самом деле мы делаем
следующее: зная опорные прямые для направлений с \(\alpha,\beta\ge0\),
с помощью стандартных методов выпуклого анализа находим явную формулу
границы заданной ими области. Именно к этому и сводятся следующие вычисления.}

\textcolor{black}{При \(\alpha>0\) положим \(\lambda:=\frac{\beta}{\alpha}\ge0\).} Из явного
вида \(K_p\) получаем \(K_p(\alpha,\beta)=\alpha K_p(1,\lambda)\). Тогда
все опорные неравенства можно записать в виде
\(d+\lambda v\le K_p(1,\lambda)\), \(\lambda\ge0\), или, что равносильно,
\(d\le K_p(1,\lambda)-\lambda v\).
Кроме того, \(v\ge0\) по определению, а выбор \((\alpha,\beta)=(0,1)\)
даёт \(v\le K_p(0,1)=1\) (то есть точек нашего множества правее прямой
\(v=1\) нет). Тогда мы можем записать \textcolor{black}{верхнюю границу этой внешней области}
как нижнюю огибающую этих прямых (\(v\in[0,1]\)):
\[
\begin{aligned}
F_p(v)
&:=\inf_{\lambda\ge0}
\bigl\{K_p(1,\lambda)-\lambda v\bigr\}
\qquad\text{(раскрываем \(K_p(1,\lambda)\) по
теореме~\ref{ru:thm:deterministic-ru}):}\\
&=\min\left\{
\inf_{\textcolor{black}{0\le\lambda\le1/p}}
\bigl(c(p)(1-\lambda)^{1-p}-\lambda v\bigr),
\inf_{\lambda\ge1/p}\lambda(1-v)
\right\}
\qquad
\textcolor{black}{\begin{gathered}
\text{(раскроем \(c(p)\) и явно}\\
\text{запишем второе выражение):}
\end{gathered}}\\
&=\min\left\{
\inf_{\textcolor{black}{0\le\lambda\le1/p}}
\left(\frac{(p-1)^{p-1}}{p^p}(1-\lambda)^{1-p}-\lambda v\right),
\frac{1-v}{p}
\right\}.
\end{aligned}
\]
\textcolor{black}{Точку \(\lambda=1/p\) включили в обе ветви, поскольку их
значения в ней совпадают.} Остаётся вычислить первый инфимум. Дифференцируя выражение по \(\lambda\),
получаем
\[
\frac{d}{d\lambda}
\left[
\frac{(p-1)^{p-1}}{p^p}(1-\lambda)^{1-p}-\lambda v
\right]
=\left(\frac{p-1}{p}\right)^p(1-\lambda)^{-p}-v.
\]
Производная возрастает по \(\lambda\), а при \(\lambda=0\) равна
\(\left(\frac{p-1}{p}\right)^p-v\). Поэтому при
\(v\le\left(\frac{p-1}{p}\right)^p\) минимум достигается
при \(\lambda=0\). В противном случае производная единственный раз
обращается в нуль, и минимум достигается при
\(\lambda_*(v)=1-\frac{p-1}{p}v^{-1/p}\)\textcolor{black}{, в том числе на правом
конце \(\lambda_*(1)=1/p\)}. Осталось лишь подставить найденные
значения \(\lambda\) в формулу для \(F_p(v)\):
\[
F_p(v)=
\begin{cases}
\displaystyle\frac{(p-1)^{p-1}}{p^p},
&\displaystyle 0\le v\le\left(\frac{p-1}{p}\right)^p,\\[3mm]
v^{(p-1)/p}-v,
&\displaystyle \left(\frac{p-1}{p}\right)^p\le v\le1.
\end{cases}
\]
Первую ветвь функции \(F_p\) будем называть плоской частью границы, а вторую
--- её криволинейной частью. Нетрудно проверить, что \(F_p\) вогнута на
\([0,1]\)\textcolor{black}{. Следовательно,} её подграфик является выпуклым.

\textcolor{black}{Теперь обоснуем обратное включение. Детерминированные законы
\(X=\tau e\) и \(X=\frac{p}{p-1}\tau e\), где \(\norm{e}=1\), дают
точки \((0,0)\) и \((0,c(p))\), а симметричный закон на \(\pm\tau e\)
даёт \((1,0)\). Для законов \(X_{q,r}\) из доказательства
утверждения~\ref{ru:prop:stochastic-sharp-ru} при \(q\downarrow0\) пары сходятся
к \((r^{-p},(r-1)/r^p)\). При \(1\le r\le p/(p-1)\) эти точки заполняют
всю криволинейную часть границы. Плоская часть является отрезком между
\((0,c(p))\) и левым концом этой дуги, а нижняя сторона --- отрезком между
\((0,0)\) и \((1,0)\). Поэтому выпуклые комбинации граничных точек заполняют
всю область \(0\le v\le1\), \(0\le d\le F_p(v)\), и она содержится в
\(\overline{\operatorname{conv}}(\mathcal A_p)\).}
Таким образом,
\(\overline{\operatorname{conv}}(\mathcal A_p)
=\{(v,d):0\le v\le1,\ 0\le d\le F_p(v)\}\).

Ниже, на иллюстрации представлены замкнутые выпуклые оболочки для
\(\mathcal A_{3/2}\) и \(\mathcal A_{6/5}\). Плоские части границ выделены
зелёным цветом, а криволинейные --- оранжевым.

\begin{figure}[H]
\centering
\begin{tikzpicture}[
  font=\small,line cap=round,line join=round,
  support/.style={draw=black!60,line width=.7pt},
  regionedge/.style={draw=geoPthreehalf,line width=1pt},
  flatpart/.style={draw=green!55!black,line width=1.55pt},
  curvedpart/.style={draw=orange!88!black,line width=1.55pt}]
  \begin{scope}[x=5.3cm,y=4.1cm]
    \pgfmathsetmacro{\q}{1/3}
    \pgfmathsetmacro{\a}{pow(1/3,3/2)}
    \pgfmathsetmacro{\c}{2/(3*sqrt(3))}
    \pgfmathsetmacro{\vzero}{.3}
    \pgfmathsetmacro{\fzero}{pow(\vzero,\q)-\vzero}
    \pgfmathsetmacro{\slope}{\q*pow(\vzero,\q-1)-1}
    \pgfmathsetmacro{\intercept}{\fzero-\slope*\vzero}

    \path[fill=geoPthreehalf!15]
      (0,0)--(0,\c)--(\a,\c)--
      plot[domain=\a:1,samples=80,smooth] (\x,{pow(\x,\q)-\x})--cycle;
    \draw[support] (0,\c)--(1.04,\c)
      node[pos=.91,above,font=\scriptsize,text=black!78] {$\ell_0$};
    \draw[support] (0,\intercept)--(1.04,{\intercept+\slope*1.04})
      node[pos=.88,above right,font=\scriptsize,text=black!78] {$\ell_1$};
    \draw[support] (1,-.025)--(1,.96)
      node[above,font=\scriptsize,text=black!78] {$\ell_2$};
    \draw[->,draw=black!80,line width=.65pt] (0,0)--(1.08,0)
      node[right,font=\normalsize,text=black] {$v$};
    \draw[->,draw=black!80,line width=.65pt] (0,0)--(0,1.07)
      node[above,font=\normalsize,text=black] {$d$};
    \draw[regionedge] (0,0)--(0,\c);
    \draw[regionedge] (0,0)--(1,0);
    \draw[flatpart] (0,\c)--(\a,\c);
    \draw[curvedpart]
      plot[domain=\a:1,samples=80,smooth] (\x,{pow(\x,\q)-\x});
    \draw[draw=black!62,densely dotted,line width=.65pt] (\a,0)--(\a,\c);
    \draw[draw=black!70,line width=.55pt] (\a,0)--(\a,-.012);
    \fill[black!72] (\a,\c) circle[radius=1.05pt];
    \fill[black!62] (\vzero,\fzero) circle[radius=.9pt];
    \node[below left,inner sep=3pt,font=\footnotesize] at (0,0) {$0$};
    \node[below,inner sep=3pt,font=\footnotesize] at (1,0) {$1$};
    \draw[draw=black!70,line width=.55pt] (0,1)--(-.012,1);
    \node[left,inner sep=3pt,font=\footnotesize] at (0,1) {$1$};
    \node[anchor=east,inner sep=3pt,text=black!82] at (-.005,\c)
      {$\displaystyle\frac{(\frac12)^{1/2}}{(\frac32)^{3/2}}$};
    \node[anchor=north,inner sep=2pt,font=\scriptsize,xshift=4pt]
      at (\a,-.012)
      {$\displaystyle\left(\frac13\right)^{3/2}$};
    \node[text=geoPthreehalf,font=\small] at (.50,.15)
      {$\overline{\operatorname{conv}}(\mathcal A_p)$};
    \node[font=\small] at (.54,1.13) {$p=\frac32$};
  \end{scope}

  \begin{scope}[xshift=7.1cm,x=5.3cm,y=4.1cm]
    \pgfmathsetmacro{\q}{1/6}
    \pgfmathsetmacro{\a}{pow(1/6,6/5)}
    \pgfmathsetmacro{\c}{5/pow(6,6/5)}
    \pgfmathsetmacro{\vzero}{.3}
    \pgfmathsetmacro{\fzero}{pow(\vzero,\q)-\vzero}
    \pgfmathsetmacro{\slope}{\q*pow(\vzero,\q-1)-1}
    \pgfmathsetmacro{\intercept}{\fzero-\slope*\vzero}

    \path[fill=geoPthreehalf!15]
      (0,0)--(0,\c)--(\a,\c)--
      plot[domain=\a:1,samples=80,smooth] (\x,{pow(\x,\q)-\x})--cycle;
    \draw[support] (0,\c)--(1.04,\c)
      node[pos=.91,above,font=\scriptsize,text=black!78] {$\ell_0$};
    \draw[support] (0,\intercept)--(1.04,{\intercept+\slope*1.04})
      node[pos=.88,above right,font=\scriptsize,text=black!78] {$\ell_1$};
    \draw[support] (1,-.025)--(1,.96)
      node[above,font=\scriptsize,text=black!78] {$\ell_2$};
    \draw[->,draw=black!80,line width=.65pt] (0,0)--(1.08,0)
      node[right,font=\normalsize,text=black] {$v$};
    \draw[->,draw=black!80,line width=.65pt] (0,0)--(0,1.07)
      node[above,font=\normalsize,text=black] {$d$};
    \draw[regionedge] (0,0)--(0,\c);
    \draw[regionedge] (0,0)--(1,0);
    \draw[flatpart] (0,\c)--(\a,\c);
    \draw[curvedpart]
      plot[domain=\a:1,samples=80,smooth] (\x,{pow(\x,\q)-\x});
    \draw[draw=black!62,densely dotted,line width=.65pt] (\a,0)--(\a,\c);
    \draw[draw=black!70,line width=.55pt] (\a,0)--(\a,-.012);
    \fill[black!72] (\a,\c) circle[radius=1.05pt];
    \fill[black!62] (\vzero,\fzero) circle[radius=.9pt];
    \node[below left,inner sep=3pt,font=\footnotesize] at (0,0) {$0$};
    \node[below,inner sep=3pt,font=\footnotesize] at (1,0) {$1$};
    \draw[draw=black!70,line width=.55pt] (0,1)--(-.012,1);
    \node[left,inner sep=3pt,font=\footnotesize] at (0,1) {$1$};
    \node[anchor=east,inner sep=3pt,text=black!82] at (-.005,\c)
      {$\displaystyle\frac{(\frac15)^{1/5}}{(\frac65)^{6/5}}$};
    \node[anchor=north,inner sep=2pt,font=\scriptsize,xshift=4pt]
      at (\a,-.012)
      {$\displaystyle\left(\frac16\right)^{6/5}$};
    \node[text=geoPthreehalf,font=\small] at (.50,.23)
      {$\overline{\operatorname{conv}}(\mathcal A_p)$};
    \node[font=\small] at (.54,1.13) {$p=\frac65$};
  \end{scope}
\end{tikzpicture}
\vspace{-1.5mm}
\caption{Замкнутая выпуклая оболочка множества \(\mathcal A_p\)
и её опорные прямые при \(p=\frac32\) и \(p=\frac65\).}
\label{ru:fig:support-hull-schematic-ru}
\end{figure}

Итак, мы получили полное описание замкнутой выпуклой оболочки. Но как выглядит
само множество достижимых пар? Это мы и узнаем в следующем разделе.

\Needspace{14\baselineskip}
\subsection{Точное множество достижимых пар}
\label{ru:subsec:attainable-set-ru}

\begin{theoremru}[Точное описание достижимого множества]
\label{ru:thm:attainable-set-ru}
Для \(1<p\le2\) достижимое множество имеет вид
\[
\mathcal A_p=
\overline{\operatorname{conv}}(\mathcal A_p)
\setminus
\left\{
\left(v,v^{\frac{p-1}{p}}-v\right):
\left(\frac{p-1}{p}\right)^p\le v<1
\right\}.
\]
Иными словами, из найденного замыкания выпуклой оболочки мы удаляем лишь
криволинейную часть верхней границы, включая её левый конец. Точка
\((1,0)\) остаётся в множестве.
\end{theoremru}

\begin{proof}
\textit{1. Недостижимость дуги.}
Сначала покажем, что криволинейная часть верхней границы не входит в
\(\mathcal A_p\). Для этого нормируем порог до \(\tau=1\)
(это не ограничивает общности:
\(C_1\!\left(\frac{X}{\tau}\right)=\frac{C_\tau(X)}{\tau}\),
а обе нормированные координаты из подраздела 4.1 при такой замене
не меняются).
Предположим противное: существует допустимый \(X\), дающий точку
на криволинейном участке, то есть
\[
\left(\frac{p-1}{p}\right)^p\le V_p(X)<1,
\qquad
D_p(X)=V_p(X)^{\frac{p-1}{p}}-V_p(X)>0.
\]
\begingroup\color{black}
\begin{lemmaru}[Поточечная оценка]
Пусть \(1<p\le2\). Для любого \(x\in H\) выполнено
\[
\norm{x-C_1(x)}+\norm{C_1(x)}^2
\le \norm{x}\,\norm{C_1(x)}^{\frac{2(p-1)}{p}}.
\]
При \(1<p<2\) равенство достигается только при \(x=0\) или
\(\norm{x}\ge1\)\textcolor{black}{. При} \(p=2\) равенство выполняется для любого \(x\).
\end{lemmaru}
\endgroup

\begin{proof}
Если \(\norm{x}>1\), то \(\norm{C_1(x)}=1\), и обе части равны
\(\norm{x}\). Если \(\norm{x}\le1\), клиппинг не меняет вектор,
и оценка принимает вид
\(\norm{x}^2\le\norm{x}^{1+\frac{2(p-1)}{p}}\).
Она верна, поскольку \(1+\frac{2(p-1)}{p}\le2\). При \(p=2\)
этот показатель равен \(2\), поэтому здесь достигается равенство\textcolor{black}{.
При} \(1<p<2\) он меньше \(2\), и для \(0<\norm{x}<1\)
неравенство строго.
\end{proof}

Теперь применим поточечную лемму для оценки \(D_p(X)\):
\[
\begin{aligned}
D_p(X)
&=\frac{\norm{\E[C_1(X)]-\E[X]}}{\E[\norm{X}^p]}
\underset{\mathrm{I}}{\le}
\frac{\E[\norm{X-C_1(X)}]}{\E[\norm{X}^p]}=\\[0.45em]
&=\frac{\E[\norm{X-C_1(X)}]+\E[\norm{C_1(X)}^2]
-\E[\norm{C_1(X)}^2]}{\E[\norm{X}^p]}
\underset{\mathrm{II}}{\le}\\[0.45em]
&\underset{\mathrm{II}}{\le}
\frac{\E[\norm{X}\,\norm{C_1(X)}^{\frac{2(p-1)}{p}}]
-\E[\norm{C_1(X)}^2]}{\E[\norm{X}^p]}
\underset{\mathrm{III}}{\le}\\[0.45em]
&\underset{\mathrm{III}}{\le}
\frac{\E[\norm{X}^p]^{\frac1p}\E[\norm{C_1(X)}^2]^{\frac{p-1}{p}}
-\E[\norm{C_1(X)}^2]}{\E[\norm{X}^p]}=\\[0.45em]
&=\left(\frac{\E[\norm{C_1(X)}^2]}{\E[\norm{X}^p]}\right)^{\frac{p-1}{p}}
-\frac{\E[\norm{C_1(X)}^2]}{\E[\norm{X}^p]}
\underset{\mathrm{IV}}{\le}\\[0.45em]
&\underset{\mathrm{IV}}{\le}
V_p(X)^{\frac{p-1}{p}}-V_p(X)=D_p(X).
\end{aligned}
\]
\noindent В переходе \(\mathrm{I}\) использовано неравенство
\(\norm{\E[Y]}\le\E[\norm{Y}]\).

\noindent В переходе \(\mathrm{II}\) использована поточечная лемма
после взятия математических ожиданий.

\noindent В переходе \(\mathrm{III}\) неравенство Гёльдера применено
к множителям \(\norm{X}\) и
\(\norm{C_1(X)}^{\frac{2(p-1)}{p}}\) с показателями
\(p\) и \(\frac{p}{p-1}\).

\noindent Переход \(\mathrm{IV}\) следует из тождества
дисперсии и предположения о точке на криволинейном участке:
\[
\makebox[\textwidth][c]{$\displaystyle
\frac{\E[\norm{C_1(X)}^2]}{\E[\norm{X}^p]}
=\frac{\E[\norm{C_1(X)-\E[C_1(X)]}^2]}{\E[\norm{X}^p]}
+\frac{\norm{\E[C_1(X)]}^2}{\E[\norm{X}^p]}
=V_p(X)+\frac{\norm{\E[C_1(X)]}^2}{\E[\norm{X}^p]}
\ge V_p(X)\ge\left(\frac{p-1}{p}\right)^p.$}
\]
\vspace{-2pt}
Функция \(f(s)=s^{\frac{p-1}{p}}-s\) строго убывает на
\(\left[\left(\frac{p-1}{p}\right)^p,\infty\right)\). Поэтому
предыдущая оценка даёт
\(f\!\left(\frac{\E[\norm{C_1(X)}^2]}{\E[\norm{X}^p]}\right)
\le f(V_p(X))\), что и даёт переход \(\mathrm{IV}\).

\bigskip

Крайние члены общей цепочки совпадают: поскольку мы рассматриваем
точку на криволинейной границе, \(D_p(X)=f(V_p(X))\).
Следовательно, в каждом из переходов \(\mathrm{I}\)--\(\mathrm{IV}\)
достигается равенство.

Поскольку \(f\) строго убывает, равенство в переходе
\(\mathrm{IV}\) требует, чтобы добавочное слагаемое
\(\dfrac{\norm{\E[C_1(X)]}^2}{\E[\norm{X}^p]}\) равнялось нулю.
Его знаменатель положителен, значит,~\(\E[C_1(X)]=0\).

\bigskip

Равенство в переходе \(\mathrm{II}\) означает, что равенство в
поточечной лемме выполняется почти наверное: разность её правой и левой
частей неотрицательна, а её ожидание равно нулю.

\medskip

Равенство в переходе \(\mathrm{III}\) и условие равенства в
неравенстве Гёльдера означают, что отношение
\(\frac{\norm{X}^p}{\norm{C_1(X)}^2}\) постоянно почти наверное
там, где \(X\ne0\). Поскольку \(D_p(X)>0\), имеем
\(\mathbb P(\norm{X}>1)>0\) (иначе \(\norm{X}\le1\) почти наверное
и \(C_1(X)=X\), что даёт \(D_p(X)=0\)). На этом множестве
\(\norm{C_1(X)}=1\), поэтому постоянство указанного отношения
требует одной и той же длины \(r>1\) почти наверное.

\Needspace{8\baselineskip}
Покажем, что других ненулевых исходов почти наверное нет. При
\(\norm{X}=1\) отношение из перехода \(\mathrm{III}\) равно \(1\),
тогда как на исходах с \(\norm{X}=r>1\) оно равно \(r^p>1\)\textcolor{black}{.
Это} противоречит постоянству отношения. При \(0<\norm{X}<1\)
рассмотрим два значения показателя: \mbox{если \(1<p<2\),} поточечное
неравенство строго, что противоречит равенству в переходе
\(\mathrm{II}\)\textcolor{black}{.} \mbox{\textcolor{black}{Если} \(p=2\), то}
\mbox{\(\dfrac{\norm{X}^2}{\norm{C_1(X)}^2}
=\dfrac{\norm{X}^2}{\norm{X}^2}=1\)}, что противоречит
переходу \(\mathrm{III}\). Следовательно,
почти наверное либо \(X=0\), либо \(\norm{X}=r>1\). Поэтому почти наверное,
включая исходы с \(X=0\), верно \(X=rC_1(X)\).

\begin{samepage}
Берём математические ожидания:
\[
\E[X]=r\E[C_1(X)]=0.
\]
Следовательно, \(D_p(X)=0\), что противоречит \(D_p(X)>0\). Значит,
ни одна точка указанной криволинейной дуги не достигается.
\end{samepage}

\textit{2. Достижимость каждой точки.}
Теперь покажем, что каждая точка найденной замкнутой выпуклой области
вне исключённой дуги достигается. Выберем единичный вектор \(e\in\textcolor{black}{H}\).
Зафиксируем высоту \(d\in[0,c(p)]\) и обозначим через \(s_d\) правый
конец горизонтального сечения области на этой высоте. Иными словами,
\(s_d\) показывает, насколько далеко вправо можно пройти при фиксированном
\(d\), оставаясь в найденной области:
\[
s_d:=\max\bigl\{t\in[0,1]:(t,d)\in
\overline{\operatorname{conv}}(\mathcal A_p)\bigr\}.
\]
Разберём два случая.

\textit{Случай 1: \(d=0\).}
Здесь \(0\le v\le1\). Пусть \(X=\tau e\) с вероятностью
\(\frac{1+\sqrt{1-v}}2\) и \(X=-\tau e\) с оставшейся вероятностью.
Тогда \(C_\tau(X)=X\), так что \(D_p(X)=0\), а прямое вычисление
даёт \(V_p(X)=v\). Тем самым достигается вся нижняя сторона области,
включая точку \((1,0)\).

\textit{Случай 2: \(0<d\le c(p)\).}
Правая ветвь найденной границы строго убывает от \(c(p)\) до нуля.
Поэтому её точка на высоте \(d\) имеет абсциссу \(s_d\), причём
\[
\left(\frac{p-1}{p}\right)^p\le s_d<1,
\qquad d=s_d^{\frac{p-1}{p}}-s_d
\quad\text{(по формуле криволинейной части границы).}
\]
На высоте \(d\) найденная область занимает отрезок \(0\le v\le s_d\)
(это верно и для плоской части границы). Поскольку правый конец
\((s_d,d)\) лежит на уже исключённой дуге, нам остаётся реализовать
точки \(0\le v<s_d\).

\textcolor{black}{Воспользуемся тем же построением, что и в доказательстве
утверждения~\ref{ru:prop:stochastic-sharp-ru}: радиус задаст высоту \(d\),
а вероятность ненулевого значения --- положение на горизонтальном сечении.
Поэтому возьмём} \(X=s_d^{-1/p}\tau e\) с
вероятностью \(1-\frac{v}{s_d}>0\), а \(X=0\) с вероятностью
\(\frac{v}{s_d}\). Этот закон допустим, поскольку
\(\E[\norm{X}^p]=\frac{\tau^p}{s_d}\left(1-\frac{v}{s_d}\right)>0\).
Вычисляя нормированные координаты, получаем искомую точку \((v,d)\):
\begin{align*}
V_p(X)
&=\dfrac{\left(1-\dfrac{v}{s_d}\right)\dfrac{v}{s_d}}
{\left(1-\dfrac{v}{s_d}\right)\dfrac{1}{s_d}}=v,\\[6pt]
D_p(X)
&=\dfrac{\left(1-\dfrac{v}{s_d}\right)\left(s_d^{-1/p}-1\right)}
{\left(1-\dfrac{v}{s_d}\right)\dfrac{1}{s_d}}
=s_d^{\frac{p-1}{p}}-s_d=d.
\end{align*}
Таким образом, реализована каждая точка левее \((s_d,d)\), а
недостижимость самого правого конца уже доказана. Это завершает
доказательство.

\end{proof}

\noindent\textcolor{black}{\textit{Комментарий.} За этими вычислениями стоит
простая идея. Если смещение должно быть нулевым, берём два противоположных
значения на сфере клиппинга с дополняющими друг друга вероятностями.
Если нужно положительное смещение, берём одно значение за пределами шара
и нулевое значение.}

\begin{remarkru}[Отдельный случай \(p=1\)]
\label{ru:rem:attainable-p1-ru}
В этом случае
\[
\mathcal A_1=
\{(v,d)\in\R_+^2:v+d\le1\}\setminus\{(0,1)\}.
\]
\end{remarkru}

\begin{proof}
Из теоремы~\ref{ru:thm:stochastic-ru} и \(K_1(1,1)=1\) следует
\(V_1(X)+D_1(X)\le1\). Кроме того,
\[
D_1(X)\le1-\frac{\E[\norm{C_\tau(X)}]}{\E[\norm{X}]}<1,
\]
поэтому \((0,1)\) недостижима.
Зафиксируем единичный вектор \(e\). Для \(d<1\) и \(v<1-d\)
нужную пару даёт \(X=\frac{\tau}{1-d}e\) с вероятностью
\(1-\frac{v}{1-d}\) и \(X=0\) иначе.
На стороне \(v+d=1\), \(v>0\), берём равновероятный закон на
\(\frac{1+d}{1-d}\tau e\) и \(-\tau e\).
\end{proof}

После полного описания \(\mathcal A_p\) мы можем найти опорные прямые
со всех сторон, то есть и для коэффициентов \(\alpha,\beta\) произвольных
знаков, а не только при \(\alpha,\beta\ge0\), как выше:

\begin{corollaryru}[Полная опорная функция]
\label{ru:cor:signed-support-ru}
Пусть \(1<p\le2\) и \(\alpha,\beta\in\R\). Тогда
\[
\sup_{(v,d)\in\mathcal A_p}(\textcolor{black}{\beta v+\alpha d})=
\begin{cases}
0, & \beta<0,\ \alpha\le0,\\[1mm]
\alpha c(p), & \beta<0,\ \alpha>0,\\[1mm]
\beta, & \beta\ge0,\ \alpha\le p\beta,\\[1mm]
\displaystyle
c(p)\frac{\alpha^p}{(\alpha-\beta)^{p-1}},
& \beta\ge0,\ \alpha>p\beta.
\end{cases}
\]
При \(p=1\) этот супремум равен \(\max\{0,\alpha,\beta\}\).
\end{corollaryru}

После точного описания множества достижимых пар и его опорных прямых
посмотрим, как энергия и смещение могут соотноситься для случайных векторов
и что эта геометрия даёт для приложений:

\begin{figure}[H]
\centering
\begin{tikzpicture}[
  x=2.70cm,y=2.70cm,font=\scriptsize,line cap=round,line join=round]

  \begin{scope}[xshift=0cm]
    \path[fill=geoPtwo!18]
      (0,0)--(0,.25)--(.25,.25)
      -- plot[domain=.25:1,samples=90,smooth]
        (\x,{sqrt(\x)-\x}) --cycle;
    \draw[geo axis] (0,0)--(1.06,0) node[right] {\(v\)};
    \draw[geo axis] (0,0)--(0,1.06) node[above] {\(d\)};
    \draw[geo attainable,draw=geoPtwo]
      (0,0)--(0,.25)--(.25,.25);
    \draw[geo attainable,draw=geoPtwo] (0,0)--(1,0);
    \draw[geo unattainable]
      plot[domain=.25:1,samples=90,smooth]
        (\x,{sqrt(\x)-\x});
    \node[geo excluded] at (.25,.25) {};
    \node[geo included] at (1,0) {};
    \node at (.52,1.15) {\(\bm{p=2}\)};
  \end{scope}

  \begin{scope}[xshift=3.55cm]
    \def\a{0.19245009}
    \def\c{0.38490018}
    \path[fill=geoPthreehalf!21]
      (0,0)--(0,\c)--(\a,\c)
      -- plot[domain=\a:1,samples=90,smooth]
        (\x,{pow(\x,0.3333333333)-\x}) --cycle;
    \draw[geo axis] (0,0)--(1.06,0) node[right] {\(v\)};
    \draw[geo axis] (0,0)--(0,1.06) node[above] {\(d\)};
    \draw[geo attainable,draw=geoPthreehalf]
      (0,0)--(0,\c)--(\a,\c);
    \draw[geo attainable,draw=geoPthreehalf] (0,0)--(1,0);
    \draw[geo unattainable]
      plot[domain=\a:1,samples=90,smooth]
        (\x,{pow(\x,0.3333333333)-\x});
    \node[geo excluded] at (\a,\c) {};
    \node[geo included] at (1,0) {};
    \node at (.52,1.15) {\(\bm{p=\frac32}\)};
  \end{scope}

  \begin{scope}[xshift=7.10cm]
    \def\a{0.11647119}
    \def\c{0.58235593}
    \path[fill=geoPsixfifths!23]
      (0,0)--(0,\c)--(\a,\c)
      -- plot[domain=\a:1,samples=90,smooth]
        (\x,{pow(\x,0.1666666667)-\x}) --cycle;
    \draw[geo axis] (0,0)--(1.06,0) node[right] {\(v\)};
    \draw[geo axis] (0,0)--(0,1.06) node[above] {\(d\)};
    \draw[geo attainable,draw=geoPsixfifths]
      (0,0)--(0,\c)--(\a,\c);
    \draw[geo attainable,draw=geoPsixfifths] (0,0)--(1,0);
    \draw[geo unattainable]
      plot[domain=\a:1,samples=90,smooth]
        (\x,{pow(\x,0.1666666667)-\x});
    \node[geo excluded] at (\a,\c) {};
    \node[geo included] at (1,0) {};
    \node at (.52,1.15) {\(\bm{p=\frac65}\)};
  \end{scope}

  \begin{scope}[xshift=10.65cm]
    \path[fill=geoPone!32] (0,0)--(0,1)--(1,0)--cycle;
    \draw[geo axis] (0,0)--(1.06,0) node[right] {\(v\)};
    \draw[geo axis] (0,0)--(0,1.06) node[above] {\(d\)};
    \draw[geo attainable,draw=geoPone]
      (0,0)--(0,1)--(1,0)--(0,0);
    \node[geo excluded] at (0,1) {};
    \node[geo included] at (1,0) {};
    \node at (.52,1.15) {\(\bm{p=1}\)};
  \end{scope}
\end{tikzpicture}
\caption{Структура множества \(\mathcal A_p\) при
\(p=2,\ \frac32,\ \frac65,\ 1\).}
\label{ru:fig:attainable-sets-ru}
\end{figure}

При переходе от \(p=1\) к \(p=2\) область допустимых нормированных пар
сужается. Это согласуется с привычной вероятностной интуицией: конечность
более высокого момента накладывает более сильное ограничение на хвосты
распределения (хотя с изменением \(p\) меняется и нормировка). В частности,
если второй момент конечен, нормированное смещение \(D_2(X)\) не превосходит
\(1/4\), тогда как грубая оценка остатка дала бы лишь \(D_2(X)\leq 1\).

Также на графиках хорошо видно, что энергия и смещение не могут одновременно
принимать свои максимальные значения: при максимальной энергии смещение
равно нулю, а для большого смещения приходится уменьшать энергию. Именно
поэтому имеет смысл оценивать их совместно. Складывая отдельные наихудшие
оценки, мы не учитываем эту связь и можем завысить границу. Наш результат
позволяет учесть её и получить точную оценку их взвешенной суммы.

\Needspace{7\baselineskip}
Также хочется обратить внимание на то, что для получения достижимой пары
\((v,d)\) не требуются сложные случайные векторы. Явные построения выше уже
дают довольно сильный результат:

\begin{corollaryru}[Двухатомная реализация]
\label{ru:cor:two-atom-ru}
Каждая достижимая пара из \(\mathcal A_p\) реализуется вещественной случайной
величиной, принимающей не более двух значений.
\end{corollaryru}

\begin{proof}
Все реализации в доказательстве теоремы~\ref{ru:thm:attainable-set-ru}
и замечании~\ref{ru:rem:attainable-p1-ru}
одномерны и имеют не более двух атомов.
\end{proof}

\subsection{Одновременная реализация при всех порогах}
\label{ru:subsec:all-thresholds-ru}

Следствие~\ref{ru:cor:two-atom-ru} показывает, что при фиксированном пороге
каждую достижимую пару можно реализовать одномерной двухатомной случайной
величиной. Что изменится, если порог менять? Можно ли одним случайным
вектором воспроизвести пары сразу при всех \(\tau>0\)? Ответ даёт следующая
теорема.

\begin{theoremru}[Одновременная трёхмерная реализация]
\label{ru:thm:all-thresholds-ru}
Пусть \(\textcolor{black}{H}\) --- ненулевое вещественное гильбертово пространство,
\textcolor{black}{а \(X\in L^1(\Omega;H)\)}. Тогда существует случайный
вектор \(Z\) со значениями в \(\R^3\), удовлетворяющий следующим условиям
(вероятностные пространства, на которых заданы \(X\) и \(Z\), могут
различаться): случайные величины \(\norm{X}\) и \(\norm{Z}\) одинаково
распределены, и для любого \(\tau>0\) выполнено
\begin{align*}
\norm{\E[C_\tau(Z)]-\E[Z]}
&=\norm{\E[C_\tau(X)]-\E[X]},\\
\E\!\left[\norm{C_\tau(Z)-\E[C_\tau(Z)]}^2\right]
&=\E\!\left[\norm{C_\tau(X)-\E[C_\tau(X)]}^2\right].
\end{align*}
\end{theoremru}

Иными словами, каким бы сложным ни был закон исходного случайного вектора,
его смещение и энергию при всех порогах можно воспроизвести одним
трёхмерным случайным вектором.

\begin{proof}
\textit{1. Функция среднего остатка при изменении \(\tau\).}
В теореме мы меняем порог \(\tau\), оставляя случайный вектор \(X\)
фиксированным. Чтобы при каждом \(\tau\) отслеживать информацию об \(X\),
существенную для смещения, введём функцию
\(R(\tau)=\E[X-C_\tau(X)]\). Она показывает среднюю часть \(X\),
отрезанную клиппингом\textcolor{black}{. Её} норма и есть величина смещения. Ниже запишем
\textcolor{black}{вид и свойства этой функции.}

\textcolor{black}{Дальше мы построим в \(\R^3\) ломаную, по скачкам производной
которой восстановим вероятности противоположных атомов. Чтобы эти вероятности
были неотрицательными, норма каждого скачка должна быть не больше массы
соответствующей сферы. После конечного случая перейдём к произвольному закону.}

\textcolor{black}{Сначала предположим, что \(\norm{X}\) принимает конечное число
значений. Если \(X=0\) почти наверное, возьмём \(Z=0\). Иначе обозначим
положительные значения нормы, имеющие ненулевую вероятность, через
\(0<r_1<\cdots<r_m\).} Тогда вид
функции \(R\) задаётся формулой
\begin{gather*}
R(\tau)=\E[X-C_\tau(X)]
=\sum_{i=1}^{m}
\E\!\left[\frac{X}{r_i}\,\middle|\,\norm{X}=r_i\right]
\,\mathbb P(\norm{X}=r_i)\,\max\{r_i-\tau,0\},\\[-2pt]
\text{(по определению клиппинга, при }x\ne0\text{ верно }
x-C_\tau(x)=\frac{x}{\norm{x}}\,\max\{\norm{x}-\tau,0\}\text{\textcolor{black}{.}}\\[-2pt]
\text{\textcolor{black}{Разбив} исходы по значениям }\norm{X}=r_i\text{ и взяв ожидание, получаем формулу выше).}
\end{gather*}
Также определим \(R(0)=\E[X]\) \textcolor{black}{(это корректно, поскольку
при \(\tau\downarrow0\) длина клиппированного вектора не
превосходит \(\tau\), а значит, \(C_\tau(X)\) и его математическое
ожидание стремятся к нулю. Следовательно,
\(R(\tau)=\E[X]-\E[C_\tau(X)]\) стремится к \(\E[X]\)).}

Теперь посмотрим на свойства \(R\):

\textit{\textcolor{black}{(Св1)} Значения и непрерывность.}
При каждом \(\tau\ge0\) значение \(R(\tau)\) лежит в том же пространстве
\(\textcolor{black}{H}\), что и \(X\). Непрерывность \(R\) следует непосредственно из её
вида: каждый множитель \(\max\{r_i-\tau,0\}\) непрерывен по \(\tau\).

\textit{\textcolor{black}{(Св2)} Значение при больших порогах.}
Также при \(\tau\ge r_m\) имеем \(R(\tau)=0\) (тогда
\(C_\tau(X)=X\) почти наверное, поэтому клиппинг не создаёт смещения).

\textit{\textcolor{black}{(Св3)} Линейность между радиусами.}
На каждом промежутке \((r_j,r_{j+1})\) каждый множитель
\(\max\{r_i-\tau,0\}\) либо тождественно равен нулю, либо равен
\(r_i-\tau\). Поэтому \(R(\tau)\) на таком промежутке линейно зависит
от \(\tau\), возможно со свободным членом, а её
производная постоянна. То же верно на промежутках \((0,r_1)\) и
\((r_m,\infty)\). Таким образом, производная может измениться только
в точках \(r_i\).

\Needspace{8\baselineskip}
\textit{\textcolor{black}{(Св4)} Изменение производной.}
Теперь посмотрим, насколько сильно производная может измениться. На указанных
промежутках производная определена, и вычисление даёт
\[
R'(\tau)=-\sum_{r_i>\tau}\mathbb P(\norm{X}=r_i)\,
\E\!\left[\frac{X}{r_i}\,\middle|\,\norm{X}=r_i\right],
\qquad \tau>0,\quad \tau\notin\{r_1,\ldots,r_m\}.
\]
При переходе через \(r_i\) из выражения выше исчезает ровно одно
слагаемое. Поэтому для пределов производной справа и слева от \(r_i\),
обозначаемых \(R'(r_i+0)\) и \(R'(r_i-0)\), получаем:
\[
R'(r_i+0)-R'(r_i-0)
=\mathbb P(\norm{X}=r_i)\,
\E\!\left[\frac{X}{r_i}\,\middle|\,\norm{X}=r_i\right].
\]
Беря нормы обеих частей, получаем:
\[
\norm{R'(r_i+0)-R'(r_i-0)}
=\mathbb P(\norm{X}=r_i)\,
\norm{\E\!\left[\frac{X}{r_i}\,\middle|\,\norm{X}=r_i\right]}
\le\mathbb P(\norm{X}=r_i).
\]
\textcolor{black}{Последнее неравенство верно, поскольку на событии
\(\norm{X}=r_i\) вектор \(X/r_i\) имеет единичную норму, а норма его
условного среднего не превосходит единицы.}
Таким образом, мы описали функцию \(R(\tau)\): её значения, линейность
между возможными длинами \(X\) и изменения производной.

\textcolor{black}{Чтобы читателю было проще понять, как устроена эта функция,
рассмотрим пример в \(\R^3\).} Пусть
\(e_1,e_2,e_3\) --- стандартные базисные векторы, а \(X\) принимает
значения \(e_1,2e_2,3e_3,\frac45(3e_1-4e_2)\)
с вероятностью \(1/4\) каждое. Здесь \(r_1=1\), \(r_2=2\),
\(r_3=3\), \(r_4=4\).
При увеличении \(\tau\) конец вектора \(R(\tau)\) проходит четыре
прямых звена, а при \(\tau\ge4\) остаётся в нуле.

\begin{figure}[H]
\centering
\begin{tikzpicture}[
  x={(-3.22cm,-1.15cm)},y={(2.76cm,-1.15cm)},z={(0cm,4.6cm)},
  line cap=round,line join=round,font=\scriptsize,
  residualaxis/.style={draw=black!65,line width=.5pt,
    -{Stealth[length=3.3pt,width=2.8pt]}},
  residualpath/.style={draw=geoPtwo,line width=1.05pt},
  residualarrow/.style={decoration={markings,
    mark=at position .56 with {\arrow{Stealth[length=4pt,width=3pt]}}},
    postaction={decorate}}]
  \coordinate (O) at (0,0,0);
  \coordinate (Rfour) at (0,0,0);
  \coordinate (Rthree) at (0.15,-0.20,0);
  \coordinate (Rtwo) at (0.30,-0.40,0.25);
  \coordinate (RzeroXY) at (0.85,-0.30,0);
  \coordinate (RoneXY) at (0.45,-0.35,0);
  \coordinate (RtwoXY) at (0.30,-0.40,0);
  \coordinate (Rone) at (0.45,-0.35,0.50);
  \coordinate (Rzero) at (0.85,-0.30,0.75);
  \draw[residualaxis] (O)--(1.15,0,0) node[below left] {$x$};
  \draw[residualaxis] (O)--(0,0.90,0) node[below right] {$y$};
  \draw[residualaxis] (O)--(0,0,0.85) node[above] {$z$};
  \draw[draw=black!25,line width=.35pt]
    (RzeroXY)--(RoneXY)--(RtwoXY)--(Rthree);
  \draw[draw=black!35,line width=.35pt,densely dashed]
    (Rzero)--(RzeroXY) (Rone)--(RoneXY) (Rtwo)--(RtwoXY);
  \foreach \P in {RzeroXY,RoneXY,RtwoXY}
    \fill[black!40] (\P) circle[radius=.7pt];
  \foreach \A/\B in {Rzero/Rone,Rone/Rtwo,Rtwo/Rthree,Rthree/O}
    \draw[residualpath,residualarrow] (\A)--(\B);
  \foreach \P in {Rzero,Rone,Rtwo,Rthree,O}
    \fill[geoPtwo] (\P) circle[radius=1.25pt];
  \node[above left,inner sep=3pt] at (Rzero) {$R(0)$};
  \node[above right,inner sep=3pt] at (Rone) {$R(r_1)$};
  \node[left,inner sep=4pt] at (Rtwo) {$R(r_2)$};
  \node[above right,inner sep=3pt] at (Rthree) {$R(r_3)$};
  \node[above right,inner sep=3pt] at (O) {$R(r_4)=0$};
\end{tikzpicture}
\vspace{-1.5mm}
\caption{Пример траектории функции \(R(\tau)\) в \(\R^3\).}
\label{ru:fig:residual-path-ru}
\end{figure}

\textit{2. Геометрическая постановка задачи.}
Для искомого вектора \(Z\) обозначим его средний остаток через
\(\widetilde R(\tau):=\E[Z-C_\tau(Z)]\). Посмотрим, какие условия
на \(\widetilde R\) обеспечивают два равенства теоремы. Первое
равенство в точности означает, что при каждом \(\tau>0\) выполнено
\(\norm{\widetilde R(\tau)}=\norm{R(\tau)}\).

Для второго равенства заметим, что при одинаковом распределении
\(\norm{X}\) и \(\norm{Z}\) совпадают величины
\(\E[\norm{C_\tau(X)}^2]\) и \(\E[\norm{C_\tau(Z)}^2]\): ведь
\(\norm{C_\tau(X)}=\min\{\norm{X},\tau\}\), и то же верно для \(Z\).
\textcolor{black}{По тождеству дисперсии энергия равна
\[
\E[\norm{C_\tau(X)}^2]-\norm{\E[C_\tau(X)]}^2.
\]
Первое слагаемое уже совпадает для \(X\) и \(Z\), поэтому для сохранения
энергии достаточно сохранить норму среднего клиппированного вектора.}
Для будущего \(Z\) также полагаем \(\widetilde R(0)=\E[Z]\).
Поскольку \(\E[C_\tau(X)]=R(0)-R(\tau)\) и
\(\E[C_\tau(Z)]=\widetilde R(0)-\widetilde R(\tau)\), второе равенство
теоремы равносильно условию
\(\norm{\widetilde R(0)-\widetilde R(\tau)}
=\norm{R(0)-R(\tau)}\).

\textcolor{black}{Таким образом, нам нужно построить в \(\R^3\) новую ломаную:
при каждом значении \(\tau\) её точка должна находиться на том же расстоянии
от нуля и от начальной вершины, что и соответствующая точка исходной ломаной.
После этого нужно проверить, что новая ломаная действительно является
функцией среднего остатка некоторого случайного вектора \(Z\) с тем же
распределением нормы, что и у \(X\).}

\textit{3. Построение трёхмерной реализации.}
Теперь явно построим ломаную \(\widetilde R\) в \(\R^3\).
Положим \(0=r_0\). Исходные вершины имеют вид
\(R(r_0),R(r_1),\ldots,R(r_m)\), причём \(R(r_0)=\E[X]\),
а \(R(r_m)=0\).

Введём в \(\R^3\) систему координат и положим
\(0=(0,0,0)\), \(\widetilde R(0)=(\norm{\E[X]},0,0)\).
Также зафиксируем плоскость \(xOy\)\textcolor{black}{. Точки} \(0\) и
\(\widetilde R(0)\) уже лежат в ней. Разместим следующую вершину так, чтобы
\[
\begin{aligned}
\widetilde R(r_1)&\in xOy,\\
\norm{\widetilde R(r_1)}&=\norm{R(r_1)},\\
\norm{\widetilde R(r_1)-\widetilde R(0)}
&=\norm{R(r_1)-R(0)}.
\end{aligned}
\]
Это возможно: треугольник \(0,R(0),R(r_1)\) можно разместить
в плоскости \(xOy\) с сохранением длин его сторон, в том числе
если некоторые вершины совпадают.
Хочется отметить, что все следующие вершины будем добавлять в том же
зафиксированном выше пространстве \(\R^3\), не перемещая уже построенные точки.

\textcolor{black}{Для построения следующей вершины нам понадобится следующая лемма.}

\begingroup\color{black}
\begin{lemmaru}
Пусть \(1\le i<m\), а точки
\(0,\widetilde R(0),\widetilde R(r_{i-1}),\widetilde R(r_i)\)
(не обязательно различные) заданы в \(\R^3\), причём для всех
\(j,k\in\{0,i-1,i\}\) выполнено:
\begin{align*}
\norm{\widetilde R(r_j)}&=\norm{R(r_j)},\\
\norm{\widetilde R(r_j)-\widetilde R(r_k)}&=\norm{R(r_j)-R(r_k)}.
\end{align*}
(При первом применении \(i=1\) имеем
\(\widetilde R(r_{i-1})=\widetilde R(r_0)=\widetilde R(0)\), поэтому
достаточно уже размещённых точек \(0\), \(\widetilde R(0)\) и
\(\widetilde R(r_1)\).)
Тогда можно указать \(\widetilde R(r_{i+1})\in\R^3\) так,
чтобы выполнялись следующие три равенства:
\begin{equation}
\label{ru:eq:next-vertex-three-distances-ru}
\begin{aligned}
\norm{\widetilde R(r_{i+1})}&=\norm{R(r_{i+1})},\\
\norm{\widetilde R(r_{i+1})-\widetilde R(0)}&=\norm{R(r_{i+1})-R(0)},\\
\norm{\widetilde R(r_{i+1})-\widetilde R(r_i)}&=\norm{R(r_{i+1})-R(r_i)}.
\end{aligned}
\mathrlap{\hspace{0.25em}\color{black!75}\left(
\begin{gathered}
\text{\makebox[3.05cm][c]{сохраняются}}\\
\text{два расстояния}\\
\text{из пункта 2}\\
\text{и длина звена}
\end{gathered}
\right)}
\end{equation}
Также при линейном прохождении соседних звеньев за те же интервалы
порога \([r_{i-1},r_i]\) и \([r_i,r_{i+1}]\) для каждого
\(\tau\in[r_{i-1},r_{i+1}]\) выполняются равенства:
\begin{equation}
\label{ru:eq:next-vertex-interior-distances-ru}
\begin{aligned}
\norm{\widetilde R(\tau)}&=\norm{R(\tau)},\\
\norm{\widetilde R(\tau)-\widetilde R(0)}&=\norm{R(\tau)-R(0)}.
\end{aligned}
\mathrlap{\hspace{1.75em}\color{black!75}\left(
\begin{gathered}
\text{\makebox[3.05cm][c]{первые два}}\\
\text{равенства }\eqref{ru:eq:next-vertex-three-distances-ru}\\
\text{для всех точек}\\
\text{соседних звеньев}
\end{gathered}
\right)}
\end{equation}
При этом скачок производной удовлетворяет оценке
(тем самым сохраняется ограничение из свойства \textcolor{black}{(Св4)}):
\begin{equation}
\label{ru:eq:next-vertex-jump-bound-ru}
\norm{\widetilde R'(r_i+0)-\widetilde R'(r_i-0)}
\le \norm{R'(r_i+0)-R'(r_i-0)}.
\end{equation}
\end{lemmaru}
\endgroup

\begin{proof}
По условию леммы точки \(0\), \(\widetilde R(0)\) и
\(\widetilde R(r_i)\) уже зафиксированы. Выбираем следующую вершину
\(\widetilde R(r_{i+1})\) в том же \(\R^3\), сохраняя
три расстояния из~\eqref{ru:eq:next-vertex-three-distances-ru}.
Такое положение существует: четвёрка
\(0,R(0),R(r_i),R(r_{i+1})\) имеет аффинную размерность
не больше трёх, а соответствующие первым трём точки в \(\R^3\)
уже размещены с сохранением расстояний.
\textcolor{black}{Среди допустимых положений выберем новую вершину так,
чтобы для \(\widetilde R\) сохранялось ограничение на скачок производной
из свойства \textcolor{black}{(Св4)} функции \(R\). Затем проверим равенства~
\eqref{ru:eq:next-vertex-interior-distances-ru} и оценку~
\eqref{ru:eq:next-vertex-jump-bound-ru}. Разберём возможные расположения
опорных точек.}

\medskip
\noindent\textit{Первый случай:} аффинный ранг точек
\(0,\widetilde R(0),\widetilde R(r_i)\) равен двум, то есть они
задают плоскость.
\textcolor{black}{Три заданных расстояния определяют проекцию новой вершины
на эту плоскость и расстояние до неё. Поэтому вне плоскости остаются
два зеркальных положения, а в самой плоскости --- одно.}

\smallskip
\noindent
\begin{minipage}[t]{.46\textwidth}
\vspace{0pt}
\noindent\textit{а) \(\widetilde R(r_{i+1})\) лежит в этой плоскости.}
Её положение, удовлетворяющее~\eqref{ru:eq:next-vertex-three-distances-ru},
единственно\textcolor{black}{. Выбираем} его.

\smallskip
\noindent\textit{б) \(\widetilde R(r_{i+1})\) не лежит в этой плоскости.}
Возможны два зеркальных положения относительно неё.
Если \(\widetilde R(r_{i-1})\) находится вне плоскости, выбираем
положение по другую сторону от неё относительно
\(\widetilde R(r_{i-1})\). Этот выбор показан на
рис.~\ref{ru:fig:next-vertex-plane-ru}. Если \(\widetilde R(r_{i-1})\) лежит
в плоскости, подходят оба положения.
\end{minipage}\hfill
\begin{minipage}[t]{.51\textwidth}
\vspace{0pt}
\centering
\begin{tikzpicture}[
  x={(1.65cm,0cm)},y={(.45cm,.55cm)},z={(0cm,1.18cm)},
  line cap=round,line join=round,font=\scriptsize,
  chosenpath/.style={draw=geoPtwo,line width=1.05pt},
  patharrow/.style={decoration={markings,
    mark=at position .56 with {\arrow{Stealth[length=4pt,width=3pt]}}},
    postaction={decorate}}]
  \coordinate (planeA) at (-1.4,-1,0);
  \coordinate (planeB) at (1.25,-1,0);
  \coordinate (planeC) at (1.25,1,0);
  \coordinate (planeD) at (-1.4,1,0);
  \coordinate (origin) at (-1.08,-.52,0);
  \coordinate (start) at (-.95,.65,0);
  \coordinate (current) at (.15,-.10,0);
  \coordinate (previous) at (-.75,.10,1.15);
  \coordinate (previousProjection) at (-.75,.10,0);
  \coordinate (chosen) at (.75,.15,-1.15);
  \coordinate (chosenProjection) at (.75,.15,0);
  \coordinate (mirror) at (.75,.15,1.15);
  \fill[geoPtwo!6] (planeA)--(planeB)--(planeC)--(planeD)--cycle;
  \draw[draw=geoPtwo!35,line width=.5pt]
    (planeA)--(planeB)--(planeC)--(planeD)--cycle;
  \fill[cyan!40!blue,fill opacity=.035]
    (origin)--(start)--(mirror)--cycle;
  \fill[cyan!40!blue,fill opacity=.035]
    (start)--(current)--(mirror)--cycle;
  \fill[cyan!40!blue,fill opacity=.035]
    (current)--(origin)--(mirror)--cycle;
  \fill[cyan!40!blue,fill opacity=.05]
    (origin)--(start)--(chosen)--cycle;
  \fill[cyan!40!blue,fill opacity=.05]
    (start)--(current)--(chosen)--cycle;
  \fill[cyan!40!blue,fill opacity=.05]
    (current)--(origin)--(chosen)--cycle;
  \draw[draw=geoPtwo!24,line width=.45pt]
    (origin)--(start)--(current)--cycle;
  \draw[draw=cyan!40!blue,draw opacity=.23,line width=.4pt,densely dashed]
    (origin)--(mirror) (start)--(mirror) (current)--(mirror);
  \draw[draw=cyan!40!blue,draw opacity=.3,line width=.45pt]
    (origin)--(chosen) (start)--(chosen) (current)--(chosen);
  \draw[draw=geoPtwo!35,line width=.5pt,densely dashed]
    (previous)--(previousProjection)
    (mirror)--(chosenProjection)
    (chosenProjection)--(chosen);
  \draw[draw=geoPtwo!40,line width=.55pt,densely dashed]
    (current)--(mirror);
  \draw[chosenpath,patharrow] (previous)--(current);
  \draw[chosenpath,patharrow] (current)--(chosen);
  \foreach \P in {origin,start,current,previous,chosen}
    \fill[geoPtwo] (\P) circle[radius=1.25pt];
  \foreach \P in {previousProjection,chosenProjection}
    \fill[geoPtwo!55] (\P) circle[radius=.9pt];
  \draw[draw=geoPtwo!55,fill=white,line width=.55pt]
    (mirror) circle[radius=1.5pt];
  \node[below left,inner sep=3pt] at (origin) {$0$};
  \node[above left,inner sep=3pt] at (start) {$\widetilde R(0)$};
  \node[below left,inner sep=3pt] at (current) {$\widetilde R(r_i)$};
  \node[above left,inner sep=3pt] at (previous)
    {$\widetilde R(r_{i-1})$};
  \node[below right,inner sep=3pt] at (chosen)
    {$\widetilde R(r_{i+1})$};
  \node[above right,inner sep=3pt,text=black!55] at (mirror)
    {другое положение};
  \node[above right,xshift=2pt,yshift=4pt,inner sep=1pt,
    text=geoPtwo!70]
    at (previousProjection) {$\widetilde R_{\mathcal P}(r_{i-1})$};
  \node[above right,xshift=6pt,yshift=4pt,inner sep=1pt,
    text=geoPtwo!70]
    at (chosenProjection) {$\widetilde R_{\mathcal P}(r_{i+1})$};
\end{tikzpicture}
\par\vspace{1mm}
\refstepcounter{figure}\label{ru:fig:next-vertex-plane-ru}
{\small\raggedright Рис.~\thefigure: Выбор новой вершины по другую
сторону плоскости. \textcolor{black}{Здесь \(\mathcal P\) --- плоскость
через \(0,\widetilde R(0),\widetilde R(r_i)\), а индекс
\(\mathcal P\) обозначает ортогональную проекцию на неё.}\par}
\end{minipage}
\par\smallskip

\medskip
\noindent\textit{Второй случай:} точки
\(0,\widetilde R(0),\widetilde R(r_i)\) лежат на одной прямой
(их аффинный ранг равен~\(1\)), но среди них есть две различные
точки. Если \(R(r_{i+1})\) также лежит на прямой через
\(0,R(0),R(r_i)\), то условия~\eqref{ru:eq:next-vertex-three-distances-ru}
однозначно задают \(\widetilde R(r_{i+1})\).
\textcolor{black}{Далее рассмотрим случай, когда следующая вершина лежит вне опорной прямой.}
\textcolor{black}{Вращение вокруг опорной прямой сохраняет расстояния
до всех опорных точек. Поэтому вершину вне прямой можно повернуть
в нужную плоскость и выбрать сторону относительно прямой.
Для выбора такого положения разберём два подслучая.}

\smallskip
\noindent
\begin{minipage}[t]{.46\textwidth}
\vspace{0pt}
\noindent\textit{а)} Вершина \(\widetilde R(r_{i-1})\) не лежит
на прямой, проходящей через \(0,\widetilde R(0),\widetilde R(r_i)\).
Тогда эта прямая и точка \(\widetilde R(r_{i-1})\) задают плоскость.
Выберем \(\widetilde R(r_{i+1})\) в этой плоскости, по другую
сторону указанной прямой относительно \(\widetilde R(r_{i-1})\),
сохраняя расстояния из~\eqref{ru:eq:next-vertex-three-distances-ru}.
\end{minipage}\hfill
\begin{minipage}[t]{.51\textwidth}
\vspace{0pt}
\centering
\begin{tikzpicture}[
  x={(1.65cm,0cm)},y={(.55cm,.72cm)},
  line cap=round,line join=round,font=\scriptsize,
  chosenpath/.style={draw=geoPtwo,line width=1.05pt},
  patharrow/.style={decoration={markings,
    mark=at position .56 with {\arrow{Stealth[length=4pt,width=3pt]}}},
    postaction={decorate}}]
  \coordinate (planeA) at (-1.45,-1.05);
  \coordinate (planeB) at (1.3,-1.05);
  \coordinate (planeC) at (1.3,1.3);
  \coordinate (planeD) at (-1.45,1.3);
  \coordinate (origin) at (-1.25,0);
  \coordinate (start) at (-.6,0);
  \coordinate (current) at (.3,0);
  \coordinate (previous) at (-.4,.9);
  \coordinate (chosen) at (.95,-.85);
  \coordinate (mirror) at (.95,.85);
  \fill[geoPtwo!6] (planeA)--(planeB)--(planeC)--(planeD)--cycle;
  \draw[draw=geoPtwo!35,line width=.5pt]
    (planeA)--(planeB)--(planeC)--(planeD)--cycle;
  \fill[cyan!40!blue,fill opacity=.035]
    (origin)--(current)--(mirror)--cycle;
  \fill[cyan!40!blue,fill opacity=.05]
    (origin)--(current)--(chosen)--cycle;
  \draw[draw=cyan!40!blue,draw opacity=.23,line width=.4pt,densely dashed]
    (origin)--(mirror) (current)--(mirror);
  \draw[draw=cyan!40!blue,draw opacity=.3,line width=.45pt]
    (origin)--(chosen) (current)--(chosen);
  \draw[draw=geoPtwo!65,line width=.6pt]
    (-1.45,0)--(1.3,0);
  \draw[draw=geoPtwo!40,line width=.55pt,densely dashed]
    (current)--(mirror);
  \draw[chosenpath,patharrow] (previous)--(current);
  \draw[chosenpath,patharrow] (current)--(chosen);
  \foreach \P in {origin,start,current,previous,chosen}
    \fill[geoPtwo] (\P) circle[radius=1.25pt];
  \draw[draw=geoPtwo!55,fill=white,line width=.55pt]
    (mirror) circle[radius=1.5pt];
  \node[below left,inner sep=3pt] at (origin) {$0$};
  \node[above left,inner sep=3pt] at (start) {$\widetilde R(0)$};
  \node[below left,inner sep=3pt] at (current) {$\widetilde R(r_i)$};
  \node[above left,inner sep=3pt] at (previous)
    {$\widetilde R(r_{i-1})$};
  \node[below right,inner sep=3pt] at (chosen)
    {$\widetilde R(r_{i+1})$};
  \node[above right,inner sep=3pt,text=black!55] at (mirror)
    {другое положение};
\end{tikzpicture}
\par\vspace{1mm}
\refstepcounter{figure}\label{ru:fig:next-vertex-line-ru}
{\small\raggedright Рис.~\thefigure: Выбор новой вершины по другую
сторону прямой.\par}
\end{minipage}
\par\smallskip

\smallskip
\noindent\textit{б)} Вершина \(\widetilde R(r_{i-1})\) лежит
на прямой, проходящей через \(0,\widetilde R(0),\widetilde R(r_i)\).

Здесь не нужно выбирать сторону относительно прямой:
возьмём любое положение \(\widetilde R(r_{i+1})\), сохраняющее
расстояния из~\eqref{ru:eq:next-vertex-three-distances-ru}.

\medskip
\noindent\textit{Третий случай:}
\(0=\widetilde R(0)=\widetilde R(r_i)\).
Тогда по условию леммы \(R(0)=R(r_i)=0\), и от новой вершины
требуется лишь \(\norm{\widetilde R(r_{i+1})}=\norm{R(r_{i+1})}\).
При \(R(r_{i+1})=0\) возьмём \(\widetilde R(r_{i+1})=0\).
Иначе возможны два подслучая.

\smallskip
\noindent
\begin{minipage}[t]{.46\textwidth}
\vspace{0pt}
\noindent\textit{а)} Вершина \(\widetilde R(r_{i-1})\ne0\).

На прямой, проходящей через \(0\) и \(\widetilde R(r_{i-1})\),
выберем \(\widetilde R(r_{i+1})\) по другую сторону от нуля
относительно \(\widetilde R(r_{i-1})\), на расстоянии
\(\norm{R(r_{i+1})}\) от нуля. Этот выбор показан на
рис.~\ref{ru:fig:next-vertex-point-ru}.
\end{minipage}\hfill
\begin{minipage}[t]{.51\textwidth}
\vspace{0pt}
\centering
\begin{tikzpicture}[
  x={(1.65cm,0cm)},y={(0cm,1.15cm)},
  line cap=round,line join=round,font=\scriptsize,
  chosenpath/.style={draw=geoPtwo,line width=1.05pt},
  patharrow/.style={decoration={markings,
    mark=at position .56 with {\arrow{Stealth[length=4pt,width=3pt]}}},
    postaction={decorate}}]
  \coordinate (previous) at (-1.2,.48);
  \coordinate (current) at (0,0);
  \coordinate (chosen) at (1,-.4);
  \draw[draw=cyan!40!blue,draw opacity=.09,line width=9pt]
    (-1.42,.568)--(1.22,-.488);
  \draw[draw=geoPtwo!35,line width=.5pt]
    (-1.42,.568)--(1.22,-.488);
  \draw[chosenpath,patharrow] (previous)--(current);
  \draw[chosenpath,patharrow] (current)--(chosen);
  \foreach \P in {previous,current,chosen}
    \fill[geoPtwo] (\P) circle[radius=1.25pt];
  \node[above left,inner sep=3pt] at (previous)
    {$\widetilde R(r_{i-1})$};
  \node[below,inner sep=4pt,align=center] at (current)
    {$0=\widetilde R(0)$\\[-1pt]$=\widetilde R(r_i)$};
  \node[below right,inner sep=3pt] at (chosen)
    {$\widetilde R(r_{i+1})$};
\end{tikzpicture}
\par\vspace{1mm}
\refstepcounter{figure}\label{ru:fig:next-vertex-point-ru}
{\small\raggedright Рис.~\thefigure: Выбор новой вершины по другую
сторону нуля.\par}
\end{minipage}
\par\smallskip

\smallskip
\noindent\textit{б)} Вершина \(\widetilde R(r_{i-1})=0\).

Выберем \(\widetilde R(r_{i+1})\) в любом направлении от нуля
на расстоянии \(\norm{R(r_{i+1})}\).

\Needspace{7\baselineskip}
\medskip
После того как мы определили новую вершину, проверим
равенства~\eqref{ru:eq:next-vertex-interior-distances-ru}:
расстояния до нуля и до \(\widetilde R(0)\) сохраняются не только
в вершинах, но и во всех точках соседних звеньев.
\medskip

В любом гильбертовом пространстве \(\textcolor{black}{H}\) для любых
\(A,B,C\in\textcolor{black}{H}\) и \(s\in[0,1]\) выполнено тождество
(это раскрытие квадрата нормы):
\[
\norm{(1-s)A+sB-C}^2
=(1-s)\norm{A-C}^2+s\norm{B-C}^2-s(1-s)\norm{B-A}^2.
\]
Зафиксируем \(j\in\{i-1,i\}\) и
\(\tau\in[r_j,r_{j+1}]\). Положим
\(s=\dfrac{\tau-r_j}{r_{j+1}-r_j}\in[0,1]\).
Тогда по линейности на звене
\(\widetilde R(\tau)=(1-s)\widetilde R(r_j)+s\widetilde R(r_{j+1})\)
и \(R(\tau)=(1-s)R(r_j)+sR(r_{j+1})\).

Подставим \(A=\widetilde R(r_j)\),
\(B=\widetilde R(r_{j+1})\), \(C=0\).
Для обоих значений \(j\) нормы концов и длина звена совпадают
с исходными: при \(j=i-1\) это дано в условии леммы, а при
\(j=i\) следует из него и~\eqref{ru:eq:next-vertex-three-distances-ru}.
Поэтому получаем:
\begingroup
\renewcommand{\norm}[1]{\bigl\lVert #1\bigr\rVert}
\[
\begin{aligned}
\norm{\widetilde R(\tau)}^2
&=(1-s)\norm{\widetilde R(r_j)}^2+s\norm{\widetilde R(r_{j+1})}^2
-s(1-s)\norm{\widetilde R(r_{j+1})-\widetilde R(r_j)}^2\\
&=(1-s)\norm{R(r_j)}^2+s\norm{R(r_{j+1})}^2
-s(1-s)\norm{R(r_{j+1})-R(r_j)}^2\\
&=\norm{R(\tau)}^2.
\end{aligned}
\]

Теперь подставим \(A=\widetilde R(r_j)\),
\(B=\widetilde R(r_{j+1})\), \(C=\widetilde R(0)\)\textcolor{black}{.
Для} исходной функции берём \(A=R(r_j)\),
\(B=R(r_{j+1})\), \(C=R(0)\).
Аналогично получаем:
\[
\begin{aligned}
&\norm{\widetilde R(\tau)-\widetilde R(0)}^2\\
&=(1-s)\norm{\widetilde R(r_j)-\widetilde R(0)}^2
+s\norm{\widetilde R(r_{j+1})-\widetilde R(0)}^2
-s(1-s)\norm{\widetilde R(r_{j+1})-\widetilde R(r_j)}^2\\
&=(1-s)\norm{R(r_j)-R(0)}^2+s\norm{R(r_{j+1})-R(0)}^2
-s(1-s)\norm{R(r_{j+1})-R(r_j)}^2\\
&=\norm{R(\tau)-R(0)}^2.
\end{aligned}
\]
\endgroup
Извлекая квадратные корни, получаем оба
равенства~\eqref{ru:eq:next-vertex-interior-distances-ru}
для каждого \(\tau\in[r_{i-1},r_{i+1}]\).

\Needspace{17\baselineskip}
\medskip
Теперь проверим оценку~\eqref{ru:eq:next-vertex-jump-bound-ru}
для скачка производной.
Как и при выборе новой вершины, рассмотрим три случая в зависимости
от расположения точек \(0,\widetilde R(0),\widetilde R(r_i)\).

\medskip
\noindent\textit{Первый случай.}
Предположим, что точки \(0,\widetilde R(0),\widetilde R(r_i)\)
задают плоскость.
Из условия леммы знаем, что \(\norm{\widetilde R(0)}=\norm{R(0)}\),
\(\norm{\widetilde R(r_i)}=\norm{R(r_i)}\) и
\(\norm{\widetilde R(r_i)-\widetilde R(0)}=\norm{R(r_i)-R(0)}\).
Для наглядности мысленно совместим плоскость через
\(0,R(0),R(r_i)\) с плоскостью через
\(0,\widetilde R(0),\widetilde R(r_i)\), совмещая соответствующие
опорные точки. Обозначим совмещённую плоскость через \(\mathcal P\).
Ортогональную проекцию точки \(R(\tau)\) на \(\mathcal P\) обозначим
\(R_{\mathcal P}(\tau)\)\textcolor{black}{. Аналогично} обозначаем проекции
\(\widetilde R(\tau)\) и производных.
(Плоскость \(\mathcal P\) зависит от номера \(i\) и меняется при каждом шаге.)

\Needspace{6\baselineskip}
\textcolor{black}{Сначала сравним проекции скачков на плоскость
\(\mathcal P\).}
Из стандартного тождества для скалярного произведения,
равенств~\eqref{ru:eq:next-vertex-interior-distances-ru}
и сохранения длины звена \([R(r_{i-1}),R(r_i)]\) по условию леммы получаем:
\[
\begin{array}{@{}c@{\;\Longrightarrow\;}l@{}}
2\langle A,B\rangle=\norm{A}^2+\norm{B}^2-\norm{A-B}^2
&\begin{aligned}
\langle\widetilde R(r_{i-1}),\widetilde R(0)\rangle
&=\langle R(r_{i-1}),R(0)\rangle,\\
\langle\widetilde R(r_{i-1}),\widetilde R(r_i)\rangle
&=\langle R(r_{i-1}),R(r_i)\rangle.
\end{aligned}
\end{array}
\]
\textcolor{black}{После совмещения опорные векторы \(R(0)\) и \(R(r_i)\)
порождают плоскость \(\mathcal P\). Скалярные произведения с ними
однозначно определяют проекцию на эту плоскость. Поэтому
\(R_{\mathcal P}(r_{i-1})=\widetilde R_{\mathcal P}(r_{i-1})\).
Для следующих вершин из сохранения трёх расстояний
в~\eqref{ru:eq:next-vertex-three-distances-ru} аналогично получаем
\(R_{\mathcal P}(r_{i+1})=\widetilde R_{\mathcal P}(r_{i+1})\).
Текущие вершины \(R(r_i)\) и \(\widetilde R(r_i)\) совмещены,
поэтому их проекции также совпадают.}

\Needspace{7\baselineskip}
\medskip
Теперь явно выпишем производные на соседних звеньях.
Так как функция \(R\) линейна на каждом из них, получаем:
\begin{equation}\label{ru:eq:adjacent-link-derivatives-ru}
R'(r_i-0)=\frac{R(r_i)-R(r_{i-1})}{r_i-r_{i-1}},
\qquad
R'(r_i+0)=\frac{R(r_{i+1})-R(r_i)}{r_{i+1}-r_i}.
\end{equation}
Аналогичные формулы верны для \(\widetilde R\).
Имеем \(R_{\mathcal P}(r_{i-1})=\widetilde R_{\mathcal P}(r_{i-1})\)
и \(R_{\mathcal P}(r_i)=\widetilde R_{\mathcal P}(r_i)\).
Поэтому, вычитая проекции и деля на \(r_i-r_{i-1}\), получаем
\(R'_{\mathcal P}(r_i-0)=\widetilde R'_{\mathcal P}(r_i-0)\).
\textcolor{black}{Аналогично, используя совпадение проекций текущей и следующей
вершин и деля их разность на \(r_{i+1}-r_i\), получаем
\(R'_{\mathcal P}(r_i+0)=\widetilde R'_{\mathcal P}(r_i+0)\).}
Следовательно, проекции скачков совпадают:
\[
\widetilde R'_{\mathcal P}(r_i+0)-\widetilde R'_{\mathcal P}(r_i-0)
=R'_{\mathcal P}(r_i+0)-R'_{\mathcal P}(r_i-0).
\]
Значит, равны нормы проекций скачков на \(\mathcal P\).

\textcolor{black}{Теперь сравним перпендикулярные составляющие
скачков. Скачок производной --- это разность скоростей на соседних
звеньях, поэтому сначала установим равенство длин перпендикулярных
составляющих соответствующих скоростей для \(R\) и \(\widetilde R\).
Обозначим перпендикулярные составляющие индексом \(\perp\).}
Длина звена \([R(r_{i-1}),R(r_i)]\) сохранена по условию
леммы, а длина звена \([R(r_i),R(r_{i+1})]\) --- по~\eqref{ru:eq:next-vertex-three-distances-ru}.
В формулах~\eqref{ru:eq:adjacent-link-derivatives-ru} и аналогичных
формулах для \(\widetilde R\) соответствующие знаменатели совпадают:
для левых производных это \(r_i-r_{i-1}\), для правых ---
\(r_{i+1}-r_i\). \textcolor{black}{Поэтому
\(\norm{\widetilde R'(r_i-0)}=\norm{R'(r_i-0)}\) и
\(\norm{\widetilde R'(r_i+0)}=\norm{R'(r_i+0)}\).}
\textcolor{black}{У соответствующих производных совпадают нормы и проекции на
\(\mathcal P\). Поэтому по теореме Пифагора совпадают и длины их
перпендикулярных составляющих:
\[
\begin{aligned}
\norm{\widetilde R'_\perp(r_i\pm0)}^2
&=\norm{\widetilde R'(r_i\pm0)}^2
-\norm{\widetilde R'_{\mathcal P}(r_i\pm0)}^2\\
&=\norm{R'(r_i\pm0)}^2-\norm{R'_{\mathcal P}(r_i\pm0)}^2
=\norm{R'_\perp(r_i\pm0)}^2.
\end{aligned}
\]
Теперь сравним разности этих составляющих, используя выбранное
положение соседних вершин относительно плоскости \(\mathcal P\).}

\medskip

\textcolor{black}{Вершина \(\widetilde R(r_i)\) лежит в \(\mathcal P\).
Если обе перпендикулярные составляющие ненулевые, то соседние вершины
\(\widetilde R(r_{i-1})\) и \(\widetilde R(r_{i+1})\) находятся вне
плоскости и по построению лежат по разные стороны от неё.
При движении к \(\widetilde R(r_i)\) и от неё перпендикулярные
составляющие скоростей поэтому направлены в одну сторону.
Длина их разности равна модулю разности длин\textcolor{black}{. Это} равенство верно
и при нулевой составляющей. Используя уже установленное равенство
длин соответствующих составляющих для \(R\) и \(\widetilde R\), получаем:}
\[
\begin{aligned}
\norm{\widetilde R'_{\perp}(r_i+0)-\widetilde R'_{\perp}(r_i-0)}
&=\Bigl|\norm{R'_{\perp}(r_i+0)}-\norm{R'_{\perp}(r_i-0)}\Bigr|\\
&\le\norm{R'_{\perp}(r_i+0)-R'_{\perp}(r_i-0)}.
\end{aligned}
\]
Последнее неравенство является применением неравенства треугольника.

\begin{figure}[H]
\centering
\begin{tikzpicture}[x=1cm,y=1cm,font=\small,
  line cap=round,line join=round,
  component/.style={line width=1.1pt,-{Stealth[length=5pt,width=3.5pt]}}]
  \foreach \shift/\height/\heading in {0/3.2/R,7.6/2.0/\widetilde R}{
    \begin{scope}[xshift=\shift cm]
      \node at (2,4.4) {$\heading$};
      \filldraw[fill=geoPtwo!7,draw=black!35,line width=.5pt]
        (-.65,-.65)--(3.65,-.65)--(4.65,.95)--(.35,.95)--cycle;
      \node[text=black!65] at (4.05,.62) {$\mathcal P$};
      \coordinate (O) at (0,0);
      \coordinate (B) at (3,.3);
      \coordinate (A) at (3,{.3+\height});
      \fill[geoPtwo!9] (O)--(B)--(A)--cycle;
      \draw[component,draw=geoPtwo] (O)--(B)
        node[midway,sloped,below=5pt] {проекция на $\mathcal P$};
      \draw[component,draw=black!60] (B)--(A)
        node[midway,right=6pt,yshift=4pt,align=left,font=\footnotesize]
        {перпендикулярная\\составляющая};
      \draw[component,draw=geoPtwo,line width=1.5pt] (O)--(A)
        node[midway,sloped,above=5pt] {скачок};
      \draw[black!55,line width=.6pt]
        (2.72,.272)--(2.72,.552)--(3,.58);
      \fill[geoPtwo] (O) circle[radius=1.4pt];
    \end{scope}
  }
  \node[align=center,font=\small] at (5.8,-1.2)
    {Проекции совпадают\textcolor{black}{. Длина} перпендикулярной составляющей не увеличивается.};
\end{tikzpicture}
\caption{\textcolor{black}{Схема ортогонального разложения скачков производной
для \(R\) и \(\widetilde R\). Проекции совпадают, а длина перпендикулярной
составляющей не увеличивается. По теореме Пифагора не увеличивается
и длина скачка.}}
\label{ru:fig:normal-velocity-directions-ru}
\end{figure}

\textcolor{black}{Остаётся оценить скачок производной по теореме
Пифагора. Его проекция на \(\mathcal P\) сохранена, а норма
перпендикулярной составляющей не увеличилась. Поэтому
\[
\begin{aligned}
&\norm{\widetilde R'(r_i+0)-\widetilde R'(r_i-0)}^2\\
&\quad=\norm{\widetilde R'_{\mathcal P}(r_i+0)-\widetilde R'_{\mathcal P}(r_i-0)}^2
+\norm{\widetilde R'_\perp(r_i+0)-\widetilde R'_\perp(r_i-0)}^2\\
&\quad\le\norm{R'_{\mathcal P}(r_i+0)-R'_{\mathcal P}(r_i-0)}^2
+\norm{R'_\perp(r_i+0)-R'_\perp(r_i-0)}^2\\
&\quad=\norm{R'(r_i+0)-R'(r_i-0)}^2.
\end{aligned}
\]
Извлекая квадратные корни, получаем
оценку~\eqref{ru:eq:next-vertex-jump-bound-ru} для случая плоскости.}

\medskip
\noindent\textit{Второй случай.}
Если точки \(0,\widetilde R(0),\widetilde R(r_i)\) лежат на одной
прямой и не все совпадают, повторим рассуждение с проекциями
на эту прямую и составляющими, перпендикулярными ей.
\textcolor{black}{Как и в первом случае, сохранённые расстояния и длины звеньев
обеспечивают совпадение проекций соответствующих скоростей и длин их
перпендикулярных составляющих. Если обе составляющие ненулевые, соседние
вершины по построению лежат в одной плоскости с опорной прямой и по разные
стороны от неё, поэтому эти составляющие скоростей сонаправлены.
Применяя ту же оценку их разности и теорему Пифагора, получаем~
\eqref{ru:eq:next-vertex-jump-bound-ru}\textcolor{black}{. При} нулевой составляющей рассуждение
также остаётся верным.}

\medskip
\noindent\textit{Третий случай.}
Если \(0=\widetilde R(0)=\widetilde R(r_i)\), то при
\(\widetilde R(r_{i-1})\ne0\) следующая вершина выбрана на продолжении
движения от \(\widetilde R(r_{i-1})\) к нулю.
Поэтому ненулевые односторонние производные сонаправлены,
\textcolor{black}{а их нормы совпадают с исходными, поскольку длины звеньев
и промежутки порога сохранены. Поэтому
\[
\begin{aligned}
\norm{\widetilde R'(r_i+0)-\widetilde R'(r_i-0)}
&=\Bigl|\norm{R'(r_i+0)}-\norm{R'(r_i-0)}\Bigr|\\
&\le\norm{R'(r_i+0)-R'(r_i-0)}.
\end{aligned}
\]
Последнее неравенство верно, поскольку разность длин двух векторов
по модулю не превосходит расстояния между ними.
Если одна из производных нулевая, та же цепочка остаётся верной.}

\medskip
Таким образом, во всех случаях выполнено
неравенство~\eqref{ru:eq:next-vertex-jump-bound-ru}, что завершает
доказательство леммы.
\end{proof}
\medskip

\Needspace{12\baselineskip}
Последовательно применяя лемму при \(i=1,\ldots,m-1\),
разместим все оставшиеся вершины. На каждом шаге её условия
обеспечены предыдущим построением. Соединим соседние вершины
отрезками, проходя их линейно по \(\tau\) на тех же промежутках
\([r_i,r_{i+1}]\), и положим \(\widetilde R(\tau)=0\) при
\(\tau\ge r_m\). Так получим искомую функцию \(\widetilde R\):
два нужных расстояния сохраняются по \textcolor{black}{равенствам}~\eqref{ru:eq:next-vertex-interior-distances-ru}
леммы, а внутренние скачки производной не увеличиваются по
\textcolor{black}{оценке}~\eqref{ru:eq:next-vertex-jump-bound-ru}.
В точке \(r_m\) скачок тоже не увеличивается: справа производная
равна нулю, а длина последнего звена сохранена. Если \(m=1\),
единственное звено уже размещено при начальном построении\textcolor{black}{.
Тождество} выше сохраняет оба расстояния на всём звене.

\medskip
\noindent\textcolor{black}{\textit{4. Восстановление случайного вектора по ломаной.}}
Теперь построим случайный вектор \(Z\) со значениями в \(\R^3\),
для которого нормы \(X\) и \(Z\) одинаково распределены
и выполнены следующие свойства:
\begin{equation}
\label{ru:eq:constructed-z-residual-ru}
\E[\norm{Z}]<\infty,\qquad
\E[Z-C_\tau(Z)]=\widetilde R(\tau)\quad\text{для любого }\tau>0.
\end{equation}
\Needspace{8\baselineskip}
Выберем единичные векторы \(e_i\in\R^3\), \(i=1,\ldots,m\),
в направлениях соответствующих скачков производной:
\[
e_i=\frac{\widetilde R'(r_i+0)-\widetilde R'(r_i-0)}
{\norm{\widetilde R'(r_i+0)-\widetilde R'(r_i-0)}}.
\]
Если скачок равен нулю, положим \(e_i=(1,0,0)\).

\medskip
Зададим \(Z\) следующим образом (первые две строки берём
для каждого \(1\le i\le m\)):
\[
Z=\begin{cases}
r_i e_i,&\text{с вероятностью }
\dfrac{\mathbb P(\norm{X}=r_i)+
\norm{\widetilde R'(r_i+0)-\widetilde R'(r_i-0)}}{2},\\[3mm]
-r_i e_i,&\text{с вероятностью }
\dfrac{\mathbb P(\norm{X}=r_i)-
\norm{\widetilde R'(r_i+0)-\widetilde R'(r_i-0)}}{2},\\[3mm]
0,&\text{с вероятностью }\mathbb P(X=0).
\end{cases}
\]

\noindent\textcolor{black}{\textit{Комментарий:}
по сути, мы выбираем \(Z\) так, чтобы длина \(r_i\) встречалась с той же
вероятностью, что у \(X\), а скачок производной среднего остатка совпадал
со скачком построенной ломаной. Для этого всю вероятность
\(\mathbb P(\norm{X}=r_i)\) распределяем между значениями
\(r_i e_i\) и \(-r_i e_i\): первое должно быть вероятнее второго
ровно на \(\norm{\widetilde R'(r_i+0)-\widetilde R'(r_i-0)}\).}

\textcolor{black}{Проверим корректность задания \(Z\).
По \textcolor{black}{оценке скачков из пункта 3} и свойству (Св4) функции \(R\)
имеем \(\norm{\widetilde R'(r_i+0)-\widetilde R'(r_i-0)}
\le\mathbb P(\norm{X}=r_i)\), поэтому обе вероятности при каждой
длине \(r_i\) неотрицательны. Их сумма равна
\(\mathbb P(\norm{X}=r_i)\)\textcolor{black}{. Вместе} с вероятностью нулевого значения
получаем единицу. Таким образом, распределение \(Z\) задано корректно,
а нормы \(X\) и \(Z\) одинаково распределены.
Поскольку \(\norm{Z}\le r_m\), вектор \(Z\) интегрируем.}

\medskip
\textcolor{black}{Остаётся проверить равенство для среднего остатка
в~\eqref{ru:eq:constructed-z-residual-ru}. Подставляя значения \(Z\)
и их вероятности, получаем:}
\[
\begin{aligned}
\E[Z-C_\tau(Z)]
&=\sum_{i=1}^{m}e_i
\norm{\widetilde R'(r_i+0)-\widetilde R'(r_i-0)}
\max\{r_i-\tau,0\}\underset{\mathrm{I}}{=}{}\\[2mm]
&\underset{\mathrm{I}}{=}\sum_{i=1}^{m}
\bigl(\widetilde R'(r_i+0)-\widetilde R'(r_i-0)\bigr)
\max\{r_i-\tau,0\}\underset{\mathrm{II}}{=}\widetilde R(\tau).
\end{aligned}
\]
\noindent Равенство I верно по выбору \(e_i\): умножение этого
направления на норму скачка даёт сам скачок, в том числе нулевой.

\noindent Равенство II верно: \textcolor{black}{производная
\(\max\{r_i-\tau,0\}\) имеет в точке \(r_i\) скачок \(1\), поэтому
скачки производной суммы совпадают со скачками \(\widetilde R'\).
Производная} разности
этих двух непрерывных кусочно-линейных функций постоянна между
соседними \(r_i\) и не имеет скачков. Поэтому разность аффинна\textcolor{black}{.
Она} равна нулю при всех \(\tau\ge r_m\), а значит, и всюду.

При \(\tau\downarrow0\) имеем
\(\norm{\E[C_\tau(Z)]}\le\tau\), поэтому из
\eqref{ru:eq:constructed-z-residual-ru} и непрерывности \(\widetilde R\)
следует \(\E[Z]=\widetilde R(0)\).

Таким образом, построенный \(Z\) удовлетворяет всем условиям,
достаточность которых установлена в начале пункта 2.

Итак, мы доказали теорему для случая, когда \(\norm{X}\)
принимает конечное число значений. Теперь обобщим результат
на произвольный закон нормы случайного вектора в исходном
гильбертовом пространстве \(\textcolor{black}{H}\).

\textcolor{black}{\textit{5. Произвольный закон длины.}}
\textcolor{black}{Далее будем действовать по следующему плану: сначала приблизим \(X\)
векторами \(X_n\), к которым применимо построение из пунктов 1--4.
Из соответствующих трёхмерных реализаций \(Z_n\) выделим предельный
закон и выберем вектор \(Z\) с этим законом. Затем проверим, что
полученный вектор \(Z\) удовлетворяет условиям теоремы.}

\medskip
\noindent\textcolor{black}{\textit{5.1. Приближение исходного вектора.}}
Для \(n\ge1\) ограничим длину вектора числом \(n\), а затем округлим
её вниз до ближайшего кратного \(\frac1n\). Направление сохраним.
Иными словами, положим
\[
a_n(r):=\frac{\lfloor n\min\{r,n\}\rfloor}{n},\qquad r\ge0,
\qquad
X_n:=\begin{cases}
\displaystyle a_n(\norm{X})\frac{X}{\norm{X}},&X\ne0,\\[2mm]
0,&X=0.
\end{cases}
\]
Норма \(\norm{X_n}\) принимает конечное число значений:
\(0,\frac1n,\frac2n,\ldots,n\), причём
\(0\le\norm{X_n}\le\norm{X}\). При этом \(X_n\) интегрируем.
\textcolor{black}{Покажем, что \(X_n\to X\) в \(L^1\).}
\textcolor{black}{Ошибка складывается из округления длины и отсечения её части выше \(n\):}
\[
\norm{X-X_n}=\norm{X}-a_n(\norm{X})
\le\frac1n+\max\{\norm{X}-n,0\}.
\]
При этом \(\max\{\norm{X}-n,0\}\to0\) почти наверное и
\(0\le\max\{\norm{X}-n,0\}\le\norm{X}\).
\textcolor{black}{По теореме о мажорируемой сходимости, учитывая
\(\E[\norm{X}]<\infty\), получаем:}
\begin{equation}\label{ru:eq:radial-l1-error-ru}
\varepsilon_n:=\E[\norm{X-X_n}]
\le\frac1n+\E[\max\{\norm{X}-n,0\}]\longrightarrow0.
\end{equation}
Так как для каждого \(n\) норма \(\norm{X_n}\) принимает конечное
число значений, к \(X_n\) применим уже доказанный случай теоремы.
Получим случайные векторы \(Z_n\) со значениями в \(\R^3\),
для которых \(\norm{Z_n}\) и \(\norm{X_n}\) одинаково распределены.
\textcolor{black}{Кроме того, по уже доказанному конечному случаю} для любого \(\tau>0\) получаем:
\begin{equation}\label{ru:eq:finite-radial-approximation-ru}
\begin{aligned}
\norm{\E[Z_n-C_\tau(Z_n)]}
&=\norm{\E[X_n-C_\tau(X_n)]},\\
\norm{\E[C_\tau(Z_n)]}
&=\norm{\E[C_\tau(X_n)]}.
\end{aligned}
\end{equation}

\medskip
\noindent\textcolor{black}{\textit{5.2. Выбор предельного закона.}}
Сами векторы \(Z_n\) могут быть заданы на разных вероятностных
пространствах, поэтому будем рассматривать их распределения.
Для любого \(M>0\) имеем
\[
\mathbb P(\norm{Z_n}>M)
=\mathbb P(a_n(\norm{X})>M)
\le\mathbb P(\norm{X}>M)\longrightarrow0
\quad\text{при }M\to\infty.
\]
Правая часть не зависит от \(n\). Значит, для любого
\(\varepsilon>0\) можно выбрать один замкнутый шар в \(\R^3\),
в который каждый \(Z_n\) попадает с вероятностью не меньше
\(1-\varepsilon\). Такие шары компактны. Поэтому по теореме
Прохорова существует подпоследовательность \(Z_{n_k}\),
\textcolor{black}{распределения которой слабо сходятся к некоторому
вероятностному закону на \(\R^3\). Возьмём случайный вектор \(Z\)
с этим законом. Тогда \(\E[f(Z_{n_k})]\longrightarrow\E[f(Z)]\)
для любой ограниченной непрерывной функции \(f\colon\R^3\to\R\).}

\textcolor{black}{Проверим, что \(\norm{Z}\) и \(\norm{X}\) одинаково
распределены. Для этого достаточно сравнить ожидания всех ограниченных
непрерывных функций от этих величин. Для любой такой функции
\(g\colon[0,\infty)\to\R\) имеем:}
\[
\E[g(\norm{Z})]
=\lim_{k\to\infty}\E[g(\norm{Z_{n_k}})]
=\lim_{k\to\infty}\E[g(a_{n_k}(\norm{X}))]
=\E[g(\norm{X})].
\]
Первое равенство следует из слабой сходимости: функция
\(z\mapsto g(\norm{z})\) ограничена и непрерывна на \(\R^3\).
Второе следует из построения \(Z_n\): его норма имеет тот же
закон, что и \(\norm{X_n}=a_n(\norm{X})\).
Для последнего равенства заметим, что
\(a_{n_k}(\norm{X})\to\norm{X}\) почти наверное.
Непрерывность \(g\) даёт
\(g(a_{n_k}(\norm{X}))\to g(\norm{X})\), а её ограниченность
позволяет перейти к пределу под знаком ожидания.

Итого получаем, что \(\norm{Z}\) и \(\norm{X}\) одинаково распределены
(при этом \(Z\) интегрируем, так как
\(\E[\norm{Z}]=\E[\norm{X}]<\infty\)).

\medskip
\noindent\textcolor{black}{\textit{5.3. Сходимость математических ожиданий.}}
Чтобы перейти к пределу в~\eqref{ru:eq:finite-radial-approximation-ru},
нам потребуется сходимость \(\E[Z_{n_k}]\) к \(\E[Z]\).
\textcolor{black}{Сначала ограничим векторы клиппингом при вспомогательном
пороге \(\tau'>0\). Для клиппированных векторов воспользуемся слабой
сходимостью, а среднюю длину остатка оценим через \(\norm{X}\).
Оператор \(C_{\tau'}\colon\R^3\to\R^3\) непрерывен и ограничен.
Применяя слабую сходимость к каждой его координате, получаем
\(\E[C_{\tau'}(Z_{n_k})]\longrightarrow\E[C_{\tau'}(Z)]\)
при \(k\to\infty\).}
Теперь оценим средние длины частей векторов, отсекаемых этим клиппингом:
\[
\begin{aligned}
\E[\norm{Z_n-C_{\tau'}(Z_n)}]
&=\E[\max\{a_n(\norm{X})-\tau',0\}]
\le\E[\max\{\norm{X}-\tau',0\}],\\
\E[\norm{Z-C_{\tau'}(Z)}]
&=\E[\max\{\norm{X}-\tau',0\}].
\end{aligned}
\]
\textcolor{black}{По неравенству треугольника получаем:}
\[
\norm{\E[Z_{n_k}]-\E[Z]}
\le\norm{\E[C_{\tau'}(Z_{n_k})]-\E[C_{\tau'}(Z)]}
+2\E[\max\{\norm{X}-\tau',0\}].
\]
\textcolor{black}{Сначала выберем \(\tau'\) так, чтобы второе слагаемое
было сколь угодно малым. Это возможно по теореме о мажорируемой
сходимости, поскольку \(\E[\norm{X}]<\infty\).
При этом фиксированном \(\tau'\) первое слагаемое стремится к нулю
при \(k\to\infty\). Следовательно,
\(\E[Z_{n_k}]\longrightarrow\E[Z]\).}

\medskip
\noindent\textcolor{black}{\textit{5.4. Сохранение смещения и энергии.}}
\textcolor{black}{Осталось проверить сохранение смещения и энергии.
Зафиксируем произвольный порог \(\tau>0\) и перейдём к пределу
в равенствах~\eqref{ru:eq:finite-radial-approximation-ru}.}
\textcolor{black}{По построению \(X_n\) и \(X\) лежат на одном луче.
Вдоль этого луча клиппинг и остаток задаются функциями
\(r\mapsto\min\{r,\tau\}\) и \(r\mapsto\max\{r-\tau,0\}\), каждая из которых
имеет константу Липшица \(1\).}
Поэтому
\[
\begin{aligned}
\norm{C_\tau(X_n)-C_\tau(X)}&\le\norm{X_n-X},\\
\norm{(X_n-C_\tau(X_n))-(X-C_\tau(X))}&\le\norm{X_n-X}.
\end{aligned}
\]
\textcolor{black}{Норма математического ожидания разности не превосходит
средней нормы этой разности. Поэтому по оценке~\eqref{ru:eq:radial-l1-error-ru},
где \(\varepsilon_n\to0\), получаем}
\[
\E[C_\tau(X_n)]\longrightarrow\E[C_\tau(X)],
\qquad
\E[X_n-C_\tau(X_n)]\longrightarrow\E[X-C_\tau(X)].
\]
\textcolor{black}{Для \(Z_{n_k}\) непрерывность и ограниченность \(C_\tau\)}
дают \(\E[C_\tau(Z_{n_k})]\to\E[C_\tau(Z)]\).
Вместе с доказанной сходимостью \(\E[Z_{n_k}]\to\E[Z]\)
это позволяет перейти к пределу в обеих частях
равенств~\eqref{ru:eq:finite-radial-approximation-ru}:
\[
\begin{aligned}
\norm{\E[Z-C_\tau(Z)]}&=\norm{\E[X-C_\tau(X)]},\\
\norm{\E[C_\tau(Z)]}&=\norm{\E[C_\tau(X)]}.
\end{aligned}
\]
\textcolor{black}{Первое равенство сохраняет смещение. Для энергии
используем тождество дисперсии из пункта 2: вторые моменты
клиппированных векторов совпадают благодаря одинаковому распределению
\(\norm{Z}\) и \(\norm{X}\), а нормы их средних совпадают по второму
равенству. Поэтому энергии также равны.}

Подпоследовательность \(Z_{n_k}\) и её предельный закон были выбраны
до фиксации \(\tau\). Поэтому один и тот же вектор \(Z\)
сохраняет смещение и энергию для каждого \(\tau>0\).
Вместе с уже доказанным равенством
\(\mathbb P(\norm{Z}\le t)=\mathbb P(\norm{X}\le t)\)
для всех \(t\ge0\) это завершает доказательство теоремы.
\end{proof}

\noindent\textit{Комментарий и интуиция задачи:} Главное в этой теореме --- то, что один
и тот же трёхмерный вектор \(Z\) воспроизводит смещение и энергию
исходного вектора \(X\) одновременно при всех порогах клиппинга.
Для каждого нового порога не требуется выбирать новое распределение.
При этом сохраняется также распределение длины вектора.
Таким образом, для изучения всей зависимости смещения и энергии
от порога клиппинга оказывается, что достаточно рассматривать случайные
векторы в \(\R^3\), даже если исходное пространство бесконечномерно.

Итак, полученные результаты позволяют ясно описать, как клиппинг
влияет на энергию и смещение. Теорема~\ref{ru:thm:stochastic-ru}
даёт точную оценку их взвешенной суммы через порог клиппинга и момент
исходного вектора. Теорема~\ref{ru:thm:attainable-set-ru} описывает,
какими могут быть нормированные энергия и смещение относительно
друг друга в зависимости от параметра \(p\).
Наконец, теорема~\ref{ru:thm:all-thresholds-ru} показывает, что для
воспроизведения их зависимости от порога достаточно одного случайного
вектора в \(\R^3\). \textcolor{black}{Теперь посмотрим, как полученные совместные
оценки применяются в двух прикладных задачах.}
\section{Приложения совместной огибающей}
\label{ru:sec:applications-ru}

В этом разделе рассмотрим два примера, показывающих, как наши оценки
применяются в разных задачах. В первом мы оценим расходы страховой компании
при заранее заданном лимите перестрахования, учитывая одновременно ожидаемые
выплаты и риск. Во втором рассмотрим алгоритм управления портфелем, в котором
клиппирование ограничивает изменение доли рискованного актива, и оценим его
отставание от стратегии с постоянной долей. Задачи разные, но в обоих случаях
наши результаты помогают совместно оценить эффекты, возникающие из-за
клиппирования.

(При описании примеров автор старался сохранить их прикладной смысл и
необходимое математическое содержание, не перегружая изложение деталями,
которые затруднили бы понимание того, где именно используется совместная
оценка.)

\subsection{Перестрахование при случайном числе убытков}

Представим такую ситуацию: мы --- страховая компания, которая страхует
крупные фирмы и хочет защитить себя от слишком больших выплат. Для этого
договоримся с другой компанией --- перестраховщиком --- следующим образом:
по каждому страховому случаю мы оплачиваем убыток до заранее выбранного
уровня \(\tau\), а всю сумму сверх этого уровня оплачивает перестраховщик
(такой принцип действительно используется в перестраховании).

Однако за эту «защиту от риска» нужно платить дополнительно. Поэтому
возникает вопрос: как выбранный уровень \(\tau\) влияет на наши ожидаемые
расходы и разброс суммарных выплат? Чтобы исследовать этот вопрос,
переведём описанную ситуацию на математический язык.

Для простоты зафиксируем период наблюдения --- один год.
Пусть \(N\) --- случайная величина, равная числу страховых случаев
за этот период (то есть числу произошедших событий, по которым возникают
страховые выплаты, а не числу заключённых договоров).
Через \(X_1,X_2,\ldots\ge0\) обозначим размеры убытков от этих случаев.
В этой модели (для упрощения) предположим, что размеры убытков независимы,
одинаково распределены и вся их последовательность независима от \(N\).
\textcolor{black}{Предположим также, что \(\E[X_1^2]<\infty\) и \(\E[N^2]<\infty\).}
Введём стандартные обозначения: \(\mu_N=\E[N]\) и
\(\sigma_N^2=\operatorname{Var}(N)\).

Также зафиксируем лимит \(\tau>0\): по каждому страховому случаю наша
компания оплачивает не более \(\tau\), то есть
\(C_\tau(X_i)=\min\{X_i,\tau\}\) (это тот же оператор радиального
клиппинга), а перестраховщик оплачивает всю оставшуюся часть
\(X_i-C_\tau(X_i)\). \textcolor{black}{Стоимость договора примем равной
ожидаемой суммарной выплате перестраховщика с надбавкой \(\alpha\ge0\).
Сама ожидаемая выплата равна \(\mu_N\E[X_1-C_\tau(X_1)]\).}
Введём суммарную выплату нашей компании за период
и стоимость договора перестрахования:
\[
\mathrm{Claims}_\tau:=\sum_{i=1}^{N}C_\tau(X_i),
\qquad
\mathrm{Cost}_\tau:=(1+\alpha)\mu_N\E[X_1-C_\tau(X_1)].
\]
Добавляя штраф за риск с заданным весом \(\gamma\ge0\), получаем
\[
\mathrm{RiskAdjustedCost}_\tau=\E[\mathrm{Claims}_\tau]+\mathrm{Cost}_\tau
+\gamma\operatorname{Var}(\mathrm{Claims}_\tau).
\]
Таким образом, \(\mathrm{RiskAdjustedCost}_\tau\) учитывает наши ожидаемые расходы --- собственные
страховые выплаты и стоимость перестрахования --- и штраф за разброс
собственных выплат. Чем больше \(\gamma\), тем важнее для нас уменьшить
этот разброс.

Теперь поставим естественный вопрос: как достаточно просто оценить
\(\mathrm{RiskAdjustedCost}_\tau\)? Ответ даёт следующее утверждение.

\begin{corollaryru}
\label{ru:cor:app-reinsurance-ru}
Пусть \(N\ge0\) --- целочисленная случайная величина, для которой
\(\E[N^2]<\infty\). Пусть \((X_i)_{i\ge1}\) --- семейство независимых
одинаково распределённых неотрицательных случайных величин, независимое
от \(N\), причём \(\E[X_1^2]<\infty\). Также зафиксируем параметры
\(\tau>0\), \(\alpha,\gamma\ge0\). Тогда расходы с учётом риска можно
оценить следующим образом:
\begin{equation}\label{ru:eq:app-reinsurance-ru}
\mathrm{RiskAdjustedCost}_\tau\le \mu_N\E[X_1]
+\frac{\mu_N}{\tau}K_2(\alpha,\gamma\tau)\E[X_1^2]
+\gamma \sigma_N^2\bigl(\E[X_1]\bigr)^2.
\end{equation}
\end{corollaryru}

\begin{proof}
\textcolor{black}{Размеры убытков независимы между собой, одинаково распределены,
а вся их последовательность независима от \(N\). Поэтому при фиксированном
числе случаев выплаты независимы.} По формуле полного ожидания и формуле
полной дисперсии получаем две цепочки:
\[
\begin{aligned}[t]
\E[\mathrm{Claims}_\tau]
&=\E\!\left[\E[\mathrm{Claims}_\tau\mid N]\right]=\\
&=\E[N\E[C_\tau(X_i)]]=\\
&=\mu_N\E[C_\tau(X_i)].
\end{aligned}
\hspace{4em}
\begin{aligned}[t]
&\operatorname{Var}(\mathrm{Claims}_\tau)=\\
&=\E[N\operatorname{Var}(C_\tau(X_i))]
 +\operatorname{Var}(N\E[C_\tau(X_i)])=\\
&=\mu_N\operatorname{Var}(C_\tau(X_i))
 +\sigma_N^2(\E[C_\tau(X_i)])^2.
\end{aligned}
\]

Далее подставим полученные выражения в формулу для
\(\mathrm{RiskAdjustedCost}_\tau\):
\[
\begin{aligned}
\mathrm{RiskAdjustedCost}_\tau
&=\underbrace{\mu_N\E[X_i]}_{\mathrm{(I)}}
+\underbrace{\mu_N\Bigl(\alpha\E[X_i-C_\tau(X_i)]
+\gamma\operatorname{Var}(C_\tau(X_i))\Bigr)}_{\mathrm{(II)}}
+\underbrace{\gamma\sigma_N^2\bigl(\E[C_\tau(X_i)]\bigr)^2}_{\mathrm{(III)}}.
\end{aligned}
\]
Оценим часть \(\mathrm{(II)}\) с помощью
теоремы~\ref{ru:thm:stochastic-ru} при \(p=2\)
(поскольку \(0\le C_\tau(X_i)\le X_i\), имеем
\(\E[X_i-C_\tau(X_i)]=|\E[X_i]-\E[C_\tau(X_i)]|\)):
\[
\alpha\E[X_i-C_\tau(X_i)]+\gamma\operatorname{Var}(C_\tau(X_i))
\le\frac{K_2(\alpha,\gamma\tau)}{\tau}\E[X_i^2].
\]
Умножая эту оценку на \(\mu_N\ge0\), получаем границу для части
\(\mathrm{(II)}\).

Часть \(\mathrm{(III)}\) оценим стандартным образом
(\(0\le\E[C_\tau(X_i)]\le\E[X_i]\) и \(\gamma\sigma_N^2\ge0\)):
\[
\gamma\sigma_N^2\bigl(\E[C_\tau(X_i)]\bigr)^2
\le\gamma\sigma_N^2\bigl(\E[X_i]\bigr)^2.
\]
Складывая часть \(\mathrm{(I)}\) и полученные оценки частей
\(\mathrm{(II)}\) и \(\mathrm{(III)}\), получаем~\eqref{ru:eq:app-reinsurance-ru}.
\end{proof}

\noindent\textit{Замечание (раздельная оценка).} Если оценивать \textcolor{black}{слагаемые части}
\(\mathrm{(II)}\) по отдельности, её вклад ограничивается величиной
\(\frac{\mu_N}{\tau}(\frac{\alpha}{4}+\gamma\tau)\E[X_1^2]\): отдельно
используются оценки \(\E[X_1-C_\tau(X_1)]\le\frac{\E[X_1^2]}{4\tau}\) и
\(\operatorname{Var}(C_\tau(X_1))\le\E[X_1^2]\). Совместная оценка заменяет
коэффициент \(\frac{\alpha}{4}+\gamma\tau\) на \(K_2(\alpha,\gamma\tau)\).

\noindent\textit{Комментарий и интуиция задачи:} Что нам даёт эта оценка?
Даже не зная распределения убытков полностью, мы можем оценить сверху
расходы компании с учётом риска. Для этого достаточно знать первые два
момента убытков и числа страховых случаев. Главное здесь, что с помощью
нашего результата мы смогли оценить надбавку за перестрахование и разброс
собственных выплат вместе, а не по отдельности. Так мы получаем \textcolor{black}{не менее}
точную границу для того же договора перестрахования.

Хочется отметить, что оценка расходов и риска при перестраховании давно изучается. Li, Li и
Young рассматривают \textcolor{black}{выбор стратегии перестрахования и инвестирования
по критерию, учитывающему среднее и дисперсию конечного капитала}~\citep{ru:li2017excesslossRU}.
Hu, Yang и Zhang исследуют \textcolor{black}{выбор порога перестрахования по первым
двум моментам и диапазону возможных значений суммарного убытка}~\citep{ru:hu2015retentionRU}. При построении нашей
страховой модели мы отталкивались от работы Bernis, Di Girolami и
Scotti~\citep{ru:bernis2026reinsuranceRU}: они рассматривают убытки, возникающие
группами, и ищут оптимальный договор перестрахования. В нашем приложении мы
рассматриваем упрощённую модель со случайным числом страховых случаев,
независимыми от него размерами убытков и заранее фиксированным лимитом
покрытия.

\subsection{Обучение торговой позиции с клиппингом}

Представим инвестиционный фонд, который в начале каждого периода решает,
какую часть капитала вложить в рискованный актив, а какую оставить в денежных
средствах. Фонду нужно найти баланс: вложить в рискованный актив достаточно,
чтобы получить прибыль, но при этом не подвергать слишком большую часть
капитала риску. Нас будет интересовать \textcolor{black}{разность логарифмов двух итоговых
капиталов}: \textcolor{black}{капитала} фонда, который меняет долю вложений от периода к
периоду, и \textcolor{black}{капитала} фонда, который все периоды держит одну и ту же долю в
рискованном активе.

\textcolor{black}{(В финансовой математике удобно пользоваться логарифмической
доходностью: коэффициенты роста капитала за отдельные периоды перемножаются,
а их логарифмы складываются. Поэтому сравнение итоговых капиталов можно
свести к сумме сравнений за отдельные периоды.)}

Именно такое сравнение рассматривают Helmbold, Schapire, Singer и
Warmuth~\citep[разделы~2.1--2.2]{ru:helmbold1998portfolioRU}: в их работе доли
обновляются мультипликативно. Здесь мы рассмотрим простой вариант задачи: один
рискованный актив и денежный счёт, а вместо их правила будем изучать
градиентное обновление с проекцией и клиппингом.

(Задача выбора активов при управлении портфелем широко рассматривается.
Например, Ito и соавторы~\citep{ru:ito2018cardinalityRU} изучают
последовательный выбор портфеля при ограничении на число активов.)

Теперь запишем задачу математически. Пусть \(T\in\mathbb N_+\) --- общее число
рассматриваемых периодов, а \(t\) --- номер текущего периода.
Рассматриваем портфель из одного рискованного актива
и денежного счёта.
\textcolor{black}{(Для простоты считаем, что денежный счёт не приносит процентов,
актив не предусматривает промежуточных выплат, а сделки совершаются без
издержек. Это позволяет сосредоточиться на том, как правило изменения доли
влияет на капитал.)}
Обозначим через \(P_{t-1}\) цену актива в начале периода
\(t\), а через \(P_t\) --- в его конце, и положим
\(r_t:=P_t/P_{t-1}>0\). До того как станет известно \(r_t\), выбираем долю капитала
\(w_t\in[0,1]\), вложенную в рискованный актив\textcolor{black}{. Остальную} часть держим на
денежном счёте. В силу условия \(w_t\in[0,1]\) модель не допускает заёмных
средств, торговли с плечом и коротких продаж. Изначальный капитал считаем
равным единице. Обозначим через
\(W_T\) итоговый капитал этой последовательной стратегии, а через \(W_T(u)\)
--- капитал для сравнения (при каждом периоде портфель возвращается к
фиксированной доле \(u\in[0,1]\) в рискованном активе).
\textcolor{black}{За один период денежная часть капитала сохраняется,
а вложенная в актив часть умножается на \(r_t\). Поэтому весь капитал
умножается на \((1-w_t)+w_t r_t=1+w_t(r_t-1)\).} Тогда
\[
W_T=\prod_{t=1}^T\bigl(1+w_t(r_t-1)\bigr),
\qquad
W_T(u)=\prod_{t=1}^T\bigl(1+u(r_t-1)\bigr).
\]
Теперь зададим правило обновления доли. До начала торговли зафиксируем
начальную долю \(w_1\in[0,1]\), шаг реакции \(\mu>0\) и порог клиппинга
\(\tau>0\). В конце периода \(t\), когда \textcolor{black}{коэффициент \(r_t\) уже известен},
вычислим сигнал \(g_t\). Он показывает, как \textcolor{black}{изменился бы} логарифм капитала за
\textcolor{black}{прошедший} период при небольшом увеличении доли, вложенной в актив. Поскольку
множитель капитала за период равен \(1+w_t(r_t-1)\), сигнал имеет вид
\[
g_t=\frac{r_t-1}{1+w_t(r_t-1)}.
\]
Тогда долю, вложенную в актив, на следующий период зададим правилом:
\[
w_{t+1}=\mathcal{P}_{[0,1]}\bigl(w_t+\mu C_\tau(g_t)\bigr).
\]
\Needspace{4\baselineskip}
Параметр \(\mu\) задаёт размер шага: чем он больше, тем сильнее сигнал
\(g_t\) сдвигает долю до проекции. Клиппинг ограничивает сигнал, чтобы резкий
переход рынка в режим шока не вызвал чрезмерного изменения доли. Оператор
\(\mathcal{P}_{[0,1]}\) проецирует полученное значение на допустимый интервал
\([0,1]\), так что доля остаётся неотрицательной и не превышает весь капитал.

Насколько \textcolor{black}{логарифм капитала} \(W_T\) уступает \textcolor{black}{логарифму капитала} \(W_T(u)\), \textcolor{black}{соответствующего}
постоянной доле \(u\)? Следующее следствие даёт такую оценку сразу для любого
\(u\in[0,1]\).

\begin{corollaryru}
\label{ru:cor:app-trading-ru}
Для любых \(r_1,\ldots,r_T>0\), \(\mu,\tau>0\) и \(w_1\in[0,1]\)
описанное правило даёт одновременно для всех \(u\in[0,1]\)
\begin{equation}\label{ru:eq:app-trading-ru}
\log\frac{W_T(u)}{W_T}
\le \frac{(u-w_1)^2}{2\mu}
+\frac{K_2(1,\mu\tau/2)}{\tau}\sum_{t=1}^T g_t^2.
\end{equation}
\end{corollaryru}

\begin{proof}
Чтобы оценить суммарное \textcolor{black}{логарифмическое} отставание за все \(T\) периодов, сначала оценим
отставание за один период. Обозначим
\(f_t(w)=\log(1+w(r_t-1))\). Функция \(f_t\) вогнута и
\(f_t'(w_t)=g_t\), \textcolor{black}{поэтому}\par
\vspace{0.3\baselineskip}
\noindent
\begin{minipage}[c]{0.66\textwidth}
\vspace{0pt}
\[
\begin{aligned}
f_t(u)-f_t(w_t)&\le g_t(u-w_t)\\
&=C_\tau(g_t)(u-w_t)
+\bigl(g_t-C_\tau(g_t)\bigr)(u-w_t)\\
&\le C_\tau(g_t)(u-w_t)
+|g_t-C_\tau(g_t)|\,|u-w_t|\\
&\le C_\tau(g_t)(u-w_t)+|g_t-C_\tau(g_t)|.
\end{aligned}
\]
\end{minipage}\hfill
\begin{minipage}[b]{0.31\textwidth}
\vspace{0pt}
\centering
\(\left(
\begin{gathered}
\text{в последнем переходе}\\
\text{использовано: }u,w_t\in[0,1],\\
\text{поэтому }|u-w_t|\le1
\end{gathered}
\right)\)
\end{minipage}\par
\vspace{0.3\baselineskip}

\textcolor{black}{Теперь сравним квадраты расстояний от долей \(w_t\) и \(w_{t+1}\)
до доли \(u\). Их разность позволит оценить слагаемое
\(C_\tau(g_t)(u-w_t)\), а при суммировании соседние квадраты сократятся:}
\[
\begin{gathered}
(w_{t+1}-u)^2
\le
(w_t+\mu C_\tau(g_t)-u)^2
=(w_t-u)^2+2\mu C_\tau(g_t)(w_t-u)
+\mu^2 C_\tau(g_t)^2.\\[0.6em]
\left(\begin{gathered}
\text{В первом переходе использовано свойство проекции:}\\
|\mathcal{P}_{[0,1]}(x)-u|\le|x-u|\quad (u\in[0,1])
\end{gathered}\right)
\end{gathered}
\]
Перенесём слагаемые с квадратами в другую часть неравенства и получим
оценку для слагаемого \(C_\tau(g_t)(u-w_t)\):
\[
\begin{aligned}
2\mu C_\tau(g_t)(u-w_t)
&\le (u-w_t)^2-(u-w_{t+1})^2+\mu^2 C_\tau(g_t)^2,\\
C_\tau(g_t)(u-w_t)
&\le \frac{(u-w_t)^2-(u-w_{t+1})^2}{2\mu}
+\frac{\mu}{2}C_\tau(g_t)^2.
\end{aligned}
\]
Теперь подставим последнюю оценку в предыдущее неравенство для
\(f_t(u)-f_t(w_t)\). Получаем:
\[
\begin{aligned}
f_t(u)-f_t(w_t)
&\le C_\tau(g_t)(u-w_t)+|g_t-C_\tau(g_t)|\\
&\le \frac{(u-w_t)^2-(u-w_{t+1})^2}{2\mu}
+|g_t-C_\tau(g_t)|+\frac{\mu}{2}C_\tau(g_t)^2\\
&=\underbrace{\frac{(u-w_t)^2-(u-w_{t+1})^2}{2\mu}}_{\mathrm{(I)}}
+\underbrace{|g_t-C_\tau(g_t)|}_{\mathrm{(II)}}
+\underbrace{\frac{\mu}{2}C_\tau(g_t)^2}_{\mathrm{(III)}}.
\end{aligned}
\]
Части \(\mathrm{(II)}\) и \(\mathrm{(III)}\) оценим совместно с помощью
теоремы~\ref{ru:thm:deterministic-ru}:
\[
|g_t-C_\tau(g_t)|+\frac{\mu}{2}C_\tau(g_t)^2
\le\frac{K_2(1,\mu\tau/2)}{\tau}g_t^2.
\]
\Needspace{20\baselineskip}
Теперь просуммируем полученное неравенство по всем периодам
\(t=1,\ldots,T\). Сумма левых частей равна разности логарифмов капиталов:
\[
\begin{aligned}
\sum_{t=1}^T\bigl(f_t(u)-f_t(w_t)\bigr)
&=\sum_{t=1}^T\log\bigl(1+u(r_t-1)\bigr)
-\sum_{t=1}^T\log\bigl(1+w_t(r_t-1)\bigr)\\
&=\log W_T(u)-\log W_T
=\log\frac{W_T(u)}{W_T}.
\end{aligned}
\]
В сумме части \(\mathrm{(I)}\) соседние квадраты сокращаются:
\[
\begin{aligned}
\sum_{t=1}^T
\frac{(u-w_t)^2-(u-w_{t+1})^2}{2\mu}
&=\frac{(u-w_1)^2-(u-w_{T+1})^2}{2\mu}.
\end{aligned}
\]
Для частей \(\mathrm{(II)}\) и \(\mathrm{(III)}\) складываем уже полученную
оценку по каждому периоду:
\[
\sum_{t=1}^T\left(
|g_t-C_\tau(g_t)|+\frac{\mu}{2}C_\tau(g_t)^2
\right)
\le \frac{K_2(1,\mu\tau/2)}{\tau}
\sum_{t=1}^T g_t^2.
\]
Объединяя эти три суммы, получаем
\[
\log\frac{W_T(u)}{W_T}
\le \frac{(u-w_1)^2-(u-w_{T+1})^2}{2\mu}
+\frac{K_2(1,\mu\tau/2)}{\tau}\sum_{t=1}^T g_t^2.
\]
Отбросив неположительное слагаемое
\(-\frac{(u-w_{T+1})^2}{2\mu}\), получаем оценку~\eqref{ru:eq:app-trading-ru},
что завершает доказательство.
\end{proof}

\noindent\textit{Замечание.} Если оценивать части
\(\mathrm{(II)}\) и \(\mathrm{(III)}\) по отдельности, их сумма не превосходит
\(\frac{1}{\tau}(\frac14+\frac{\mu\tau}{2})g_t^2\), поскольку
\(|g_t-C_\tau(g_t)|\le\frac{g_t^2}{4\tau}\) и \(C_\tau(g_t)^2\le g_t^2\).
Совместная оценка заменяет коэффициент \(\frac14+\frac{\mu\tau}{2}\) на
\(K_2(1,\mu\tau/2)\).

\noindent\textit{Комментарий и интуиция задачи:} Теперь мы понимаем
следующее: при заданных доле сравнения \(u\), шаге \(\mu\) и пороге \(\tau\)
\textcolor{black}{логарифмическое} отставание капитала фонда с меняющейся долей от капитала фонда с постоянной
долей оценивается через начальную долю и сумму квадратов сигналов за все
периоды.


\section{Близкие работы и мотивация исследования}
\label{ru:sec:scope-ru}

{\color{black}
Идея этой работы возникла у автора при чтении работ о стохастическом
градиентном спуске с клиппингом, в частности работ Koloskova и соавторов
и Nguyen и соавторов~\citep{ru:koloskova2023revisitingRU,ru:nguyen2023improvedRU}.
В таких задачах регулярно приходится учитывать два эффекта: смещение среднего
после отсечения и второй момент оставшегося вектора. От их оценок зависят
выбор порога клиппинга и гарантии сходимости алгоритмов. Это привело автора
к вопросу: если обе величины возникают в одном рассуждении, насколько точно
можно оценить их вместе?
\par}

{\color{black}
Эти оценки продолжают уточняться. Li и Liu используют оценки смещения
и дисперсии при анализе сходимости с высокой
вероятностью~\citep{ru:li2023clippingRU}. Parletta, Paudice и Salzo исследуют
методы с клиппингом при тяжёлых хвостах шума и получают уточнённые оценки
сходимости~\citep{ru:parletta2024clippedRU}. He изучает зависимость смещения
и среднеквадратичной ошибки от порога при покоординатном клиппинге
и дополнительных предположениях о хвостах
шума~\citep{ru:he2025biasvarianceRU}. В этих задачах более точное понимание
последствий клиппинга помогает обосновывать выбор параметров и оценивать
поведение алгоритма.
\par}

{\color{black}
Покажем на примере работы Li и Liu~\citep{ru:li2023clippingRU}, что даёт
совместная оценка. В приложении~A.1,
формулы~(3), (6), при \(1<p\le2\), \(\tau>0\) и
\(\E[\norm{X}^p]\le G^p\) авторы отдельно ограничивают смещение через
\(G^p\tau^{1-p}\), а дисперсию после клиппинга --- через
\(G^p\tau^{2-p}\). Сложение этих границ с неотрицательными весами
\(\alpha\) и \(\frac{\beta}{\tau}\) даёт коэффициент \(\alpha+\beta\).
Теорема~\ref{ru:thm:stochastic-ru} и
утверждение~\ref{ru:prop:stochastic-sharp-ru} позволяют заменить его точной
совместной константой \(K_p(\alpha,\beta)\). Например, при \(p=2\) и
\(\alpha=\beta=1\) вместо границы \(\frac{2G^2}{\tau}\) получаем
\(\frac{G^2}{\tau}\).
\par}

{\color{black}
Наша работа связана также с точными моментными оценками отсечённого остатка.
Классические оценки Scarf--Lo, изложенные у
Marinelli~\citep[теорема~4.1 и следствие~4.2]{ru:marinelli2024semiparametricRU},
определяют наибольший средний остаток при заданных среднем и дисперсии.
Kleer и соавторы~\citep[теорема~3.4]{ru:kleer2024expectationRU} рассматривают
другие моментные ограничения. В этих задачах точные границы достигаются
на двухточечных распределениях. Для неотрицательной случайной величины
средний остаток совпадает со смещением после клиппинга. Мы оцениваем его
совместно с дисперсией клиппированной части при ограничении на \(p\)-й момент.
\par}

{\color{black}
К удивлению автора, среди рассмотренных работ не нашлось прямого аналога
теоремы~\ref{ru:thm:all-thresholds-ru} об одновременной трёхмерной реализации
при всех порогах клиппинга.
\par}

\section{Формальная верификация в Lean 4}

Все нумерованные теоремы, \textcolor{black}{утверждения, следствия и замечание} статьи в разделах~2--5
(результаты~2.1--5.2) формализованы в Lean~4 с использованием библиотеки
mathlib. \textcolor{black}{Доказательство каждого из этих результатов записано в Lean
и проверяется системой.}

Исходные файлы формализации и инструкции для повторной проверки находятся в
открытом репозитории
{\hypersetup{urlcolor=geoPtwo}\url{https://github.com/Pure-Maths/Radial-Clipping-Verification}}.
Файлы в репозитории разделены по назначению. В папке \texttt{Numbered} для
каждого нумерованного результата есть отдельный файл с его формальной
формулировкой и списком используемых аксиом. Основные доказательства и
вспомогательные результаты собраны в папке \texttt{SharpRadialClipping}.

\textcolor{black}{После установки необходимых инструментов по инструкции в README
формализацию можно проверить командой}
\texttt{lake build SharpRadialClipping}\textcolor{black}{. Другие} команды приведены в README
репозитория.

\section{Использование ИИ в работе}
\label{ru:sec:ai-use-ru}

В ходе подготовки этой работы автор активно использовал искусственный
интеллект. Ниже описано, для каких задач привлекались большие языковые модели,
какие именно модели использовались и какие решения \textcolor{black}{принимал автор}.

\noindent\textit{Технические шаги в доказательствах.}
Для проработки отдельных технических переходов использовались GPT-5.6 Luna и
GPT-5.6 Terra. Автор задавал направление рассуждения, указывал, какие факты и
преобразования следует применить, а модели помогали подробно записать эти шаги.
\textcolor{black}{Общую стратегию доказательств, роль промежуточных утверждений и
дальнейшее направление исследования определял автор.} Предложенные переходы включались в
доказательство только после того, как автор проверял их логику на бумаге.

\medskip
\noindent\textit{Поиск возможных приложений.}
Для поиска работ и задач, в которых могут применяться результаты статьи,
использовались GPT-5.6 Luna и, для более тщательного поиска, GPT-5.6 Sol.
\textcolor{black}{Автор обращался к найденным публикациям, проверял их содержание
и самостоятельно решал, применим ли результат статьи в соответствующей задаче.}

\medskip
\noindent\textit{Поиск близких работ и оценка новизны.}
Для расширенного поиска возможных аналогов автор привлёк около 30--40 агентов
на GPT-5.6 Sol. \textcolor{black}{Их ответы помогали находить близкие работы для
дальнейшего изучения. Автор проверял первоисточники, сопоставлял условия и
выводы найденных результатов с результатами статьи и на этой основе
формулировал выводы об их отличиях и новизне.}

\medskip
\noindent\textit{Формализация в Lean.}
К формализации автор обратился после того, как основной текст статьи был
написан. ИИ помогал переносить доказательства в Lean~4 и разбирать технические
трудности формализации. \textcolor{black}{Эта проверка стала дополнительным этапом
работы над уже полученными результатами.}

\medskip
\noindent\textit{Подготовка исходников и текста.}
ИИ использовался для ускорения подготовки LaTeX-кода, оформления доказательств,
диаграмм и иллюстраций. \textcolor{black}{При первоначальном написании текста в
отдельных случаях автор своими словами излагал нужную мысль, а модель помогала
оформить её в виде абзаца. Затем автор редактировал предложенный текст и
проверял, что он точно передаёт задуманное.}

\medskip
\noindent\textcolor{black}{\textit{Итоговая вычитка и уточнение изложения.}}
\textcolor{black}{После подготовки основного текста автор использовал GPT-6.1 Sol
для последовательной проверки статьи с точки зрения читателя. Модель помогала
находить переходы, которым недоставало объяснения, проверять согласованность
обозначений и предлагать локальные изменения текста и LaTeX-оформления.
Особое внимание уделялось тому, понятно ли читателю, зачем вводится новый
объект, для чего нужен очередной шаг и как он связан с предыдущими. Автор
обсуждал предложенные замечания, проверял их математический смысл и выбирал,
какие изменения внести, сохраняя общий замысел и стиль изложения.}

\textcolor{black}{Таким образом, ИИ использовался при проработке технических шагов,
поиске литературы, формализации и подготовке текста. Математические идеи,
общую стратегию доказательств, выбор приложений и окончательные решения
определял автор. Он проверял принимаемые предложения и несёт ответственность
за содержание статьи.}

\section{Заключение}

В этой работе мы искали точный способ описать два последствия радиального
клиппинга: часть вектора, которая отсекается, и энергию, которая остаётся.
Для детерминированной задачи в любом ненулевом вещественном нормированном
пространстве и для случайных векторов \textcolor{black}{с конечным \(p\)-м моментом}
в гильбертовом пространстве мы получили
при \(1\le p\le2\) и \(\alpha,\beta\ge0\) одну и ту же оптимальную
константу \(K_p(\alpha,\beta)\). \textcolor{black}{Исследуя геометрию, мы описали
множество достижимых нормированных пар энергии и смещения и установили,
какие точки его границы достигаются, а какие можно лишь приблизить.} Оказалось, что при
фиксированном \(\tau\) каждую такую пару можно реализовать очень просто:
достаточно взять случайную величину, принимающую не более двух значений.

\textcolor{black}{Затем мы рассмотрели, можно ли одним случайным вектором
воспроизвести поведение при всех порогах сразу. Для интегрируемого случайного
вектора со значениями в гильбертовом пространстве мы построили один случайный
вектор \(Z\) со значениями в \(\mathbb R^3\), который при каждом \(\tau>0\)
воспроизводит ту же величину смещения и ту же энергию. Для каждого нового
порога выбирать новый закон не требуется. При этом сохраняется и распределение
нормы исходного вектора.}

Наш результат оказался применим в двух разных областях. В теории страхования
он даёт оценку расходов с учётом риска при перестраховании и заранее заданном
лимите. При анализе алгоритмических стратегий управления портфелем он
позволяет оценить логарифмическое отставание стратегии с клиппированием от
стратегии с любой фиксированной долей. \textcolor{black}{Таким образом, совместное
рассмотрение остатка и энергии позволяет получить точные оценки, описать
геометрию достижимых пар и использовать эту связь в прикладных задачах.}

\textcolor{black}{Также хотелось бы отдельно отметить ограничения нашей работы.
Мы рассмотрели радиальный клиппинг, а для произвольных способов отсечения
оценок не получили. Стохастические результаты доказаны в гильбертовых
пространствах. Их перенос на произвольные нормированные пространства мы
не исследовали.}

\textcolor{black}{В прикладных примерах мы показали, как наша оценка позволяет
уточнить границы расходов и отставания стратегии. При этом новых способов
выбора договора перестрахования или новых алгоритмов управления портфелем
мы не предлагали.}

\textcolor{black}{Наконец, в теореме~4.5 мы показали, почему размерности три
достаточно. Необходимость этой размерности в статье не доказана.}

При подготовке работы ИИ помогал прорабатывать технические шаги,
искать литературу, переносить доказательства в Lean и редактировать текст.
Автор проверял предложения моделей и принимал окончательные решения.
Подробнее это описано в разделе~\ref{ru:sec:ai-use-ru}.

\begingroup
\small
\hbadness=2500
\setlength{\bibsep}{1pt}

\endgroup

\clearpage
\endgroup
\end{document}